\documentclass[10pt, reqno]{amsart}
\usepackage[utf8]{inputenc}
\usepackage{new-style}
\title[List Coloring Uncrowded Hypergraphs at the Shattering Threshold]{List Coloring Uncrowded Hypergraphs at the\\
Shattering Threshold}

\author{Abhishek Dhawan$^{\star}$}
\address{$^{\star}$Department of Mathematics, University of Illinois Urbana--Champaign}
\email{adhawan2@illinois.edu}
\thanks{Abhishek Dhawan is partially supported by the NSF RTG grant DMS-1937241.}

\author{Abhishek Methuku$^{\star}$}
\email{methuku@illinois.edu}
\thanks{Abhishek Methuku is supported by the UIUC Campus Research Board Award RB25050 and a Simons award SFI-MPS-TSM-00025583.}

\author{Minh-Quan Vo$^{\star}$}
\email{vo15@illinois.edu}

\date{}

\begin{document}

\begin{abstract}
A foundational theorem of Ajtai, Koml\'os, Pintz, Spencer, and Szemer\'edi asserts that every $n$-vertex $k$-uniform uncrowded hypergraph (i.e., with girth at least $5$) of maximum degree at most $\Delta$ contains an independent set of size $cn(\log\Delta/\Delta)^{\frac{1}{k-1}}$ for some absolute constant $c>0$. Improving an earlier bound of Frieze and Mubayi, Iliopoulos showed that every such hypergraph has list chromatic number at most $(1+o(1))(k-1)(\Delta/\log\Delta)^{\frac{1}{k-1}}$. Determining the optimal leading constant in this bound remains a major open problem. A natural target, suggested by the coupon-collector heuristic, is $(k-1)^{\frac{1}{k-1}}$, which also corresponds to the shattering threshold for coloring sparse random hypergraphs.

Our main result shows that this target is attained: for every integer $k\ge2$, every $k$-uniform uncrowded hypergraph of maximum degree at most $\Delta$ 
satisfies
\[
\chi_\ell(\mathcal{H}) \le (1+o(1)){\left((k-1)\,\frac{\Delta}{\log\Delta}\right)^{\frac{1}{k-1}}}.
\]
In particular, this confirms a 2019 conjecture of Molloy for uncrowded hypergraphs.
Building on this result, we prove that every $k$-uniform linear hypergraph of maximum degree at most $\Delta$ has chromatic number at most $(c_k+o(1))(\Delta/\log\Delta)^{\frac{1}{k-1}}$, where $c_k\to1$ as $k\to\infty$, thereby proving a conjecture of Verstra\"ete and Wilson on the independence number of linear hypergraphs in the stronger form of colorings. Our techniques also yield a palette sparsification theorem for uncrowded hypergraphs, partially answering a question of Casselgren and Eriksson. 
\end{abstract}

\maketitle

\sloppy

\pagenumbering{gobble}
\newpage
\tableofcontents
\newpage
\pagenumbering{arabic}

\section{Introduction}\label{sec: intro}

\subsection{Background}\label{subsection: background and results}

A \emph{hypergraph} is an ordered pair $H=(V,E)$, where $V$ is a finite set of \emph{vertices} and $E\subseteq 2^V$ is a family of \emph{edges}. We say that $H$ is \emph{$k$-uniform} (or a \emph{$k$-graph}) if every edge has exactly $k$ elements; thus, a graph is a $2$-uniform hypergraph. A set $I\subseteq V$ is \emph{independent} if it contains no edge of $H$, and the \emph{independence number} $\alpha(H)$ is the maximum size of an independent set. A \emph{proper coloring} of $H$ is an assignment of colors to the vertices of $H$ such that no edge is monochromatic, and the \emph{chromatic number} $\chi(H)$ is the minimum number of colors in a proper coloring.
 
It is well known that every $n$-vertex graph $G$ of maximum degree $\Delta$ and average degree $d$ satisfies
\begin{align}\label{eq:greedy}
\alpha(G)\ge \frac{n}{d+1} \qquad\text{and}\qquad \chi(G)\leq \Delta+1.
\end{align}
These are commonly referred to as the \emph{greedy bounds}. A central theme in extremal and probabilistic combinatorics is to determine which structural assumptions allow one to improve on these bounds. In this direction, a fundamental result of Ajtai, Koml\'os, and Szemer\'edi~\cite{AKS1980, ajtai1981dense} shows that every $n$-vertex triangle-free graph $G$ of average degree 
$d$ satisfies
\begin{align}\label{eq:AKS}
\alpha(G) \ge c\, n\, \frac{\log d}{d}
\end{align}
for some absolute constant $c>0$. This influential result was originally used to establish the celebrated Ramsey-number bound $R(3,k)=O(k^2/\log k)$, and has since inspired extensive research over the last four decades with striking applications to number theory and discrete geometry, among other areas. Shearer~\cite{Sh83, Sh91} showed that the leading constant $c$ in~\eqref{eq:AKS} can be improved to $1-o_d(1)$. Determining whether this constant can be improved further remains a major open problem; indeed, Shearer's result implies the best known upper bound $R(3,k) \leq (1+o(1)) \, k^2/\log k$, and thus any improvement to the leading constant in~\eqref{eq:AKS} would also improve this bound on Ramsey numbers. See also the recent breakthroughs~\cite{campos2025new, hefty2025improving}, which improved the lower bound to $R(3,k) \ge (1/2+o(1))\, k^2/\log k$.

Introduced independently by Vizing~\cite{vizing1976coloring} and Erd\H{o}s, Rubin, and Taylor~\cite{erdos1979choosability}, 
the \emph{list chromatic number} of a hypergraph $H$, denoted $\chi_\ell(H)$, is the minimum integer $q$ such that, if every vertex $v \in V(H)$ is assigned a list of colors $L(v) \subseteq \N$ of at least $q$ colors, then $H$ admits a proper coloring in which every vertex $v$ receives a color from $L(v)$. It is not difficult to see that $\chi(H) \le \chi_\ell(H)$. 
In 1995, Kim~\cite{Kim95} proved that every graph of girth at least $5$ and maximum degree at most $\Delta$ satisfies
\begin{align}\label{eq:kim}
\chi_\ell(G) \leq (1+o(1)) \,  \frac{\Delta}{\log \Delta}.
\end{align}
Independently, Johansson~\cite{Joh_triangle} showed that every triangle-free graph of maximum degree at most $\Delta$ satisfies $\chi(G) = O(\Delta/\log\Delta)$. 
In 2019, Molloy~\cite{Molloy} unified and strengthened these two results by improving the leading constant in Johansson's bound to match that in~\eqref{eq:kim};
see~\cite{bernshteyn2019johansson, hurley2021first, DKPS, bonamy2022bounding, martinsson2021simplified} for alternate proofs. More generally, the setting of \emph{locally sparse} graphs---where the neighborhood of every vertex spans few edges---has been extensively studied~\cite{AKSConjecture, vu2002general, DKPS, hurley2021first, anderson2024coloring, anderson2025coloring, bradshaw2025toward}; see also the survey~\cite{KangKelly2023nibble} and the references therein.
 
For $k$-uniform hypergraphs, the natural analogue of the greedy bound for the independence number in~\eqref{eq:greedy} is the following result of Spencer~\cite{spencer1972turan}, which extends Tur\'an's theorem~\cite{Turan1941} to higher uniformities: every $n$-vertex $k$-uniform hypergraph $H$ of average degree $d \in \N$ satisfies
\begin{align}\label{eq:greedy-hyper}
\alpha(H)\ge \left(1-\frac{1}{k}\right) \frac{n}{d^{\frac{1}{k-1}}}.
\end{align}
It is natural to ask whether, as in the graph case, this bound can be improved under suitable local sparsity assumptions defined as follows. For $r\ge 2$, an \emph{$r$-cycle} is a hypergraph with $r$ distinct edges $e_0,e_1,\dots,e_{r-1}$ for which there exist distinct vertices $v_0,v_1,\dots,v_{r-1}$ such that $v_i\in e_i\cap e_{i+1}$ for every $0\le i\le r-1$, with subscripts taken modulo~$r$. The \emph{girth} of a hypergraph is the length of its shortest cycle. A hypergraph is \emph{linear} if it contains no $2$-cycle, and, following~\cite{AKPSS1982}, it is \emph{uncrowded} if it has girth at least~$5$, that is, if it contains no $r$-cycle for $2\leq r\leq 4$.
 
The study of independence numbers of uncrowded hypergraphs was initiated by Koml\'os, Pintz, and Szemer\'edi~\cite{komlos1982lower} in their work on Heilbronn's triangle problem, which asks for the largest possible minimum area of a triangle spanned by three of $n$ points in the unit disk.
Heilbronn conjectured that this quantity is $O(n^{-2})$; Koml\'os, Pintz, and Szemer\'edi disproved this by constructing point sets in which every such triangle has area $\Omega(n^{-2}\log n)$, using a lower bound on the independence number of uncrowded $3$-graphs.
Ajtai, Koml\'os, Pintz, Spencer, and Szemer\'edi~\cite{AKPSS1982} later extended this bound to all uniformities, proving the following foundational theorem.
 
\begin{theorem}[Ajtai--Koml\'os--Pintz--Spencer--Szemer\'edi~\cite{AKPSS1982}]\label{theorem: AKPSS}
Let $k\ge 2$ be an integer. There exists $\Delta_0\in\N$ such that, for every $\Delta\ge \Delta_0$ and $n\in\N$, every $n$-vertex $k$-uniform uncrowded hypergraph $\cH$ of maximum degree at most $\Delta$ satisfies 
\[ 
\alpha(\cH)\ge (c_k -o_\Delta(1)) \,n \, \left(\frac{\log \Delta}{\Delta}\right)^{\frac{1}{k-1}}, 
\] where $c_k=\frac{0.98}{e}{\left(\frac{1}{10^5(k-1)}\right)^{\frac{1}{k-1}}}$.
\end{theorem}
 
While the bound in Theorem~\ref{theorem: AKPSS} is optimal up to the value of $c_k$ (see, e.g.,~\cite{bennett2024chromatic}), determining the optimal leading constant has been a major open problem. 
In the graph case ($k=2$), the constant $c_2=1$ is a natural benchmark: it is attained for several classes of graphs capturing the typical local structure of the sparse Erd\H{o}s--R\'enyi random graph~\cite{Kttt, Molloy, DKPS, AndersonBernshteynDhawan} (see also~\cite{bonamy2022bounding, anderson2025coloring} for related bounds)
and coincides with the well-known \emph{shattering threshold} for independent sets and colorings in sparse random graphs~\cite{coja2015independent, Achlioptas, Zdeborova, molloy2018freezing}. 
For $k\ge 3$, the analogous shattering threshold is $\left(1/(k-1)\right)^{\frac{1}{k-1}}$; see our recent work~\cite[\S2]{DMV26ind} for an informal justification of this barrier.
This constant and its reciprocal $(k-1)^{\frac{1}{k-1}}$, which plays the same role for colorings, arise in several seemingly distinct settings: the geometry of the solution space for colorings of Erd\H{o}s--R\'enyi hypergraphs~\cite{Achlioptas, gabrie2017phase, dyer2015chromatic, ayre2019hypergraph}, structural properties of multipartite hypergraphs~\cite{dhawan2025list, dhawan2025balanced}, statistical physics~\cite{michelen2022strong, cannon2026pirogov}, and average-case complexity theory~\cite{dhawan2026low, dhawan2026algorithmic}, among others. 
Iliopoulos~\cite{iliopoulos2021improved} improved the leading constant $c_k$ in Theorem~\ref{theorem: AKPSS} (for $3 \le k \le 11$) to $1/(k-1)$ as a consequence of a bound on the list chromatic number of uncrowded hypergraphs (Theorem~\ref{theorem: fotis} below).
In our recent work~\cite{DMV26ind}, to appear in SODA 2027, we closed the gap to the shattering threshold by proving the following bound. This result was independently and concurrently proved for $k \ge 3$ in~\cite{yu2026ind}.
 
\begin{theorem}[\cite{DMV26ind}]\label{theorem: independence main}
Let $k \ge 2$ be an integer, and let $\eps > 0$ be arbitrary.
There exists $\Delta_0\in\N$ such that, for every
$\Delta\ge \Delta_0$ and $n\in\N$, every $n$-vertex $k$-uniform
uncrowded hypergraph $\cH$ of maximum degree at most~$\Delta$ satisfies
\begin{align}\label{eq:ind-shattering}
\alpha(\cH)\ge (1-\eps)\,n \,
\left(\frac{1}{k-1}\,
\frac{\log \Delta}{\Delta}\right)^{\frac{1}{k-1}}.
\end{align}
\end{theorem}

Nearly all of the results discussed above on the independence number have analogues in the chromatic number setting.
In 2013, Frieze and Mubayi~\cite{frieze2013coloring} generalized Johansson's bound on triangle-free graphs to the hypergraph setting,
proving that every $k$-uniform hypergraph with maximum degree at most $\Delta$ and girth at least $4$ satisfies
\begin{align}\label{eq:FM}
\chi(\cH) \leq (c + o(1)) {\left(\frac{\Delta}{\log \Delta}\right)^{\frac{1}{k-1}}},
\end{align}
for some constant $c > 0$ depending only on $k$.
Analogous bounds in the broader setting of triangle-free hypergraphs have also been obtained by Cooper and Mubayi~\cite{cooper2015list} for $3$-graphs and by Li and Postle~\cite{lipostle2022} for all uniformities $k\ge 3$, where a triangle is a collection of three distinct edges $e,f,g$ and three distinct vertices $u,v,w$ such that $u\in e\cap f$, $v\in f\cap g$, $w\in e\cap g$, and $\{u, v, w\}\cap e \cap f \cap g = \varnothing$.
 
While the bound in~\eqref{eq:FM} is optimal in order of magnitude (see, e.g.,~\cite{bennett2024chromatic}), determining the optimal leading constant has remained a major open problem even in the uncrowded setting.
In 2021, Iliopoulos~\cite{iliopoulos2021improved} established the following result, which extends the bound in~\eqref{eq:kim} to all uniformities with an explicit leading constant.

\begin{theorem}[Iliopoulos~\cite{iliopoulos2021improved}]\label{theorem: fotis}
Let $k \ge 3$ be an integer, and let $\eps > 0$ be arbitrary. There exists $\Delta_0\in\N$ such that every $k$-uniform uncrowded hypergraph $\cH$ of maximum degree at most $\Delta \ge \Delta_0$ satisfies
\[
  \chi_\ell(\cH) \le (1+\eps)\, (k-1) {
  \left(\frac{\Delta}{\log\Delta}\right)^{\frac{1}{k-1}}}.
\]
\end{theorem}

We note that the bound in Theorem~\ref{theorem: fotis} is far from the scale suggested by the shattering threshold. In fact, the leading constant $k-1$ is substantially larger than the constant $(k-1)^{\frac{1}{k-1}}$ corresponding to the shattering threshold for colorings, which tends to $1$ as $k \to \infty$. 
In the coloring setting, this prediction can also be understood through the following ``coupon-collector'' heuristic; see~\cite[Ch.~12]{MolloyReed} for the graph case. Suppose that $\cH$ is a $k$-uniform uncrowded hypergraph of maximum degree $\Delta$, and that each vertex is assigned a color independently and uniformly at random from a set of $q$ colors. 
A color $c$ is ``unavailable'' at a vertex $v$ if there is an edge $e$ containing $v$ such that every vertex in $e - v$ is assigned the same color $c$; for a fixed edge, this occurs with probability $q^{-(k-1)}$. We say that $v$ ``sees'' the color $c$ in its neighborhood if such an edge $e$ exists. 
The coupon-collector heuristic suggests that if $q^{k-1} \log q \lesssim \Delta$, then one should expect some vertex $v$ to see all $q$ colors in its neighborhood,
so that assigning any color to $v$ would create a monochromatic edge, thereby preventing $v$ from being properly colored.
Rearranging $q^{k-1}\log q \gtrsim \Delta$ suggests the threshold
\[
q \gtrsim \left((k-1)\,\frac{\Delta}{\log \Delta}\right)^{\frac{1}{k-1}}.
\]
Indeed, Molloy explicitly conjectured the corresponding bound for $3$-uniform hypergraphs of girth at least $4$ during his plenary lecture at CanaDAM 2019.

\begin{conjecture}[Molloy~\cite{MolloyCanaDAM}]\label{conj: Molloy}
For every $\eps>0$, there exists $\Delta_0 \in \N$ such that every $3$-uniform hypergraph $\cH$ of girth at least $4$ and maximum degree at most $\Delta \ge \Delta_0$ satisfies \[
\chi(\cH) \leq (1+\eps)\sqrt{\dfrac{2\Delta}{\log \Delta}}.
\]
\end{conjecture}

In this paper, we confirm Molloy's conjecture for uncrowded hypergraphs, establish improved upper bounds on the chromatic number of uncrowded and linear $k$-uniform hypergraphs, and prove a palette sparsification theorem for uncrowded hypergraphs. Before stating these results formally in the subsequent subsections, we highlight our core contributions below.

\smallskip
\noindent 
\textit{Our main contributions.}\quad
Improving upon the bound of Iliopoulos~\cite{iliopoulos2021improved} (Theorem~\ref{theorem: fotis}), our first main result confirms Conjecture~\ref{conj: Molloy} for uncrowded $3$-graphs and, more generally, establishes the corresponding shattering-threshold bound for uncrowded hypergraphs of every uniformity $k \ge 3$.
Specifically, we show that every $k$-uniform uncrowded hypergraph $\cH$ of maximum degree at most $\Delta$ satisfies
\[
\chi(\cH) \le (1+o(1))
{\left((k-1)\,\frac{\Delta}{\log\Delta}\right)^{\frac{1}{k-1}}}.
\]
The same bound in fact holds for the list chromatic number; see Theorem~\ref{theorem: list chromatic}.
More generally, our proof, a variant of the R\"odl Nibble, yields this bound under a weaker color-degree condition in place of the maximum-degree hypothesis; see Theorem~\ref{theorem: main}.

Color-degree bounds are closely tied to palette sparsification and sublinear algorithms (see Section~\ref{subsection: palette}). Leveraging this connection, we prove a palette sparsification theorem for uncrowded hypergraphs with a palette of size $o(\Delta^{\frac{1}{k-1}})$; to our knowledge, this is the first palette sparsification result for hypergraphs (of uniformity $k \ge 3$) with a palette of size $o(\Delta^{\frac{1}{k-1}})$.

Finally, via a refinement of a random partitioning procedure of Frieze and Mubayi~\cite{frieze2013coloring}, we show that every $k$-uniform linear hypergraph $\cH$ of maximum degree at most $\Delta$ satisfies
\[
\chi(\cH) \le (1+o_\Delta(1))(1+o_k(1)) {\left(\frac{\Delta}{\log\Delta}\right)^{\frac{1}{k-1}}}.
\]
This strengthens our recent bound on the independence number of linear hypergraphs~\cite{DMV26ind}, which resolved a conjecture of Verstra\"ete and Wilson~\cite[Conjecture~1]{verstraete2026independent}; see Section~\ref{subsection: main results coloring}.

The remainder of this introduction is organized as follows: Section~\ref{subsection: main results coloring} formally states our main results on the chromatic numbers of uncrowded and linear hypergraphs; Section~\ref{subsection: color-degree} discusses the color-degree setting and presents a more general theorem that implies our upper bound on $\chi_\ell(\cH)$ for uncrowded hypergraphs; Section~\ref{subsection: palette} presents our palette sparsification theorem; Section~\ref{subsection: summary} outlines several directions for future work; and Section~\ref{subsection: notation} collects the standard notation and terminology used throughout the paper.

\subsection{Coloring uncrowded and linear hypergraphs}\label{subsection: main results coloring}

In this subsection, we formally state our main results on the chromatic numbers of uncrowded and linear hypergraphs.
We begin with our bound for uncrowded hypergraphs, which sharpens Theorem~\ref{theorem: fotis} to match the threshold predicted by the coupon-collector heuristic.
 
\begin{tcolorbox}[enhanced, breakable]
\begin{theorem}\label{theorem: list chromatic}
Let $k \ge 2$ be an integer, and let $\eps > 0$ be arbitrary. There exists $\Delta_0\in\N$ such that every $k$-uniform uncrowded hypergraph $\cH$ of maximum degree at most $\Delta \ge \Delta_0$ satisfies
\[
  \chi_\ell(\cH) \le (1+\eps){
  \left((k-1)\,\frac{\Delta}{\log\Delta}\right)^{\frac{1}{k-1}}}.
\]
\end{theorem}
\end{tcolorbox}

Theorem~\ref{theorem: list chromatic} immediately implies Theorem~\ref{theorem: independence main}, since every color class in a proper coloring is an independent set.
The leading constant $(k-1)^{\frac{1}{k-1}}$ is precisely the value for which the implied independence number bound matches the shattering-threshold constant. 
Additionally, for $k=3$, this gives the leading constant $\sqrt{2}$, thereby confirming Conjecture~\ref{conj: Molloy} for uncrowded hypergraphs.
We note that Theorem~\ref{theorem: list chromatic} follows as a special case of a more general result, stated as Theorem~\ref{theorem: main} in Section~\ref{subsection: color-degree}, which replaces the maximum-degree hypothesis with a weaker color-degree condition.

Well-known results on random regular $k$-graphs (e.g.,~\cite{bennett2024chromatic}) yield the existence of a $k$-uniform uncrowded hypergraph of maximum degree at most $\Delta$ with 
\[
\chi(\cH) \ge (1-o(1)){\left(\frac{k-1}{k} \, \frac{\Delta}{\log\Delta}\right)^{\frac{1}{k-1}}};
\]
thus the leading constant in Theorem~\ref{theorem: list chromatic} cannot, in general, be improved beyond $\left(\frac{k-1}{k}\right)^{\frac{1}{k-1}}$.
Since both constants tend to~$1$ as $k\to\infty$, the bound in Theorem~\ref{theorem: list chromatic} is asymptotically tight as $k\to\infty$.
As remarked in \S\ref{subsection: background and results}, the constant $(k-1)^{\frac{1}{k-1}}$ in Theorem~\ref{theorem: list chromatic} coincides with the shattering threshold for colorings of sparse random $k$-graphs, and the bound on the independence number that it implies matches the shattering threshold for independent sets. For independent sets, the remaining multiplicative gap of $k^{\frac{1}{k-1}}$ reflects a fundamental computational barrier~\cite{RV, wein2020optimal, gamarnik2021overlap, gamarnik2025optimal, dhawan2026low, dhawan2026algorithmic}. Since every proper coloring with $q$ colors has a color class of size at least $n/q$, this suggests that closing the corresponding gap for colorings would likely require substantially new, non-constructive techniques. See also our recent work~\cite[\S2]{DMV26ind}, to appear in SODA 2027, for an informal justification of this barrier for independent sets.

To place Theorem~\ref{theorem: list chromatic} in context, we briefly compare its proof with the methods underlying the preceding results.
Frieze and Mubayi~\cite{frieze2013coloring} generalized Johansson's entropy-based nibble approach~\cite{Joh_triangle} from graphs to hypergraphs, whereas Iliopoulos~\cite{iliopoulos2021improved} extended Kim's argument~\cite{Kim95}.
Although Johansson's method applies to triangle-free graphs (unlike Kim's, which requires girth at least $5$), it does not achieve the shattering-threshold constant $1$ as in~\eqref{eq:kim}; the approach of Frieze and Mubayi inherits this limitation.
Molloy~\cite{Molloy} later achieved the constant $1$ for triangle-free graphs using entropy compression, but that approach does not appear to extend directly to hypergraphs.

Our proof instead follows the Kim--Iliopoulos line of argument.
At a high level, Kim's nibble tracks the number of available colors at each vertex alongside the number of incident edges capable of ruling out each vertex-color pair.
In hypergraphs, partially colored edges yield constraints spanning varying numbers of uncolored vertices, which can be encoded as edges of different sizes.
Whereas Iliopoulos tracks a separate maximum color degree for each edge size and controls the largest of them, we bound all of these color degrees simultaneously in terms of a single parameter, and show that it suffices to control their sum, suitably normalized by the list sizes.
Unlike our earlier work on independent sets presented in~\cite{DMV26ind} (Theorem~\ref{theorem: independence main})---where the core novelty was a carefully chosen damping factor in the underlying nibble process---the present paper introduces a fundamentally different framework, in which the nibble dynamically tracks list sizes and color degrees to obtain a list-coloring bound matching the shattering threshold (see Section~\ref{subsection: nibble overview} for a detailed overview).

We next turn to the linear setting.
Duke, Lefmann, and R\"odl~\cite{duke1995uncrowded} showed that Theorem~\ref{theorem: AKPSS} can be extended to linear hypergraphs via a random sampling technique inspired by Ajtai, Koml\'os, and Szemer\'edi's~\cite{AKS1980, ajtai1981dense} reduction of locally sparse graphs to the triangle-free setting.
In the graph case, Alon, Krivelevich, and Sudakov~\cite{AKSConjecture} designed a random partitioning technique that reduces the problem of coloring locally sparse graphs to that of coloring triangle-free graphs: rather than showing that a locally sparse graph contains a ``large'' triangle-free induced subgraph, as Ajtai, Koml\'os, and Szemer\'edi did, they partitioned its vertex set into ``few'' parts, each inducing a triangle-free subgraph, thereby extending Johansson's result on coloring triangle-free graphs~\cite{Joh_triangle} to this setting.
In a similar vein, Frieze and Mubayi~\cite{frieze2013coloring} used random partitioning to show that the bound in~\eqref{eq:FM} for $k$-graphs of girth at least $4$ extends to all linear $k$-graphs.

\begin{theorem}[{Frieze--Mubayi~\cite{frieze2013coloring}}]\label{theo: FM coloring}
    Let $k \ge 3$ be an integer. There exists $\Delta_0\in\N$ such that every $k$-uniform linear hypergraph $\cH$ of maximum degree at most $\Delta \ge \Delta_0$ satisfies
    \[
      \chi(\cH) \le (c_k + o(1)) {\left(\frac{\Delta}{\log\Delta}\right)^{\frac{1}{k-1}}},
    \]
    for some constant $c_k > 0$ depending only on $k$.
\end{theorem}

Informally, their reduction proceeds in the following three steps:
\begin{enumerate}[label=(FM\arabic*)]
    \item\label{FM1} For a well-chosen $m_1$, construct a partition $U_1 \cup \dots \cup U_{m_1}$ of $V(\cH)$ via random partitioning.
    An application of the Kim--Vu Concentration Inequality together with the Lov\'asz Local Lemma shows that each vertex $u \in U_i$ is contained in at most $O(\Delta^\eps)$ $3$-cycles in $\cH[U_i]$.

    \item\label{FM2} For a well-chosen $m_2$, construct a partition $V_{i, 1} \cup \dots \cup V_{i,m_2}$ of each $U_i$ via random partitioning.
    An application of the Lov\'asz Local Lemma shows that each vertex $v \in V_{i, j}$ is contained in at most $O(1)$ $3$-cycles in $\cH[V_{i, j}]$.

    \item\label{FM3} Finally, for some $m_3 = \Theta(k^2)$, form a partition $V_{i, j, 1} \cup \dots \cup V_{i, j, m_3}$ of each $V_{i, j}$ via a greedy argument such that $\cH[V_{i, j, l}]$ has girth at least $4$ for all $i, j, l$.
\end{enumerate}
They show that $\Delta(\cH[V_{i, j, l}]) = O(\Delta / (m_1m_2)^{k-1})$ for each $i, j, l$ and complete the proof by applying~\eqref{eq:FM} to each $\cH[V_{i, j, l}]$ using pairwise disjoint color sets. In their argument, both the constant $c$ in~\eqref{eq:FM} and the number $m_3$ of parts in step~\ref{FM3} grow with $k$, so the resulting constant $c_k$ tends to infinity as $k \to \infty$.

In Section~\ref{section: linear reduction}, we give a similar reduction from the linear setting to the uncrowded setting, which allows us to apply Theorem~\ref{theorem: list chromatic} in place of~\eqref{eq:FM}; recall that the leading constant $(k-1)^{\frac{1}{k-1}}$ in Theorem~\ref{theorem: list chromatic} tends to $1$ as $k\to\infty$.
Notably, our argument bypasses the final greedy step~\ref{FM3}.
Instead, a more careful analysis, with a sharper choice of $m_1$ in step~\ref{FM1} and the general form of the Lov\'asz Local Lemma in step~\ref{FM2} rather than its symmetric version, allows us to obtain a leading constant that also tends to $1$ as $k\to\infty$ (Corollary~\ref{corollary: linear} below).

\begin{tcolorbox}[enhanced, breakable] 
\begin{corollary}\label{corollary: linear}
Let $k \ge 3$ be an integer, and let $\eps > 0$ be arbitrary. There exists $\Delta_0\in\N$ such that every $k$-uniform linear hypergraph $\cH$ of maximum degree at most $\Delta \ge \Delta_0$ satisfies
\[
  \chi(\cH) \le (1+\eps){
  \left(\frac{(4k-5)(k-1)}{k-2}\,\frac{\Delta}{\log\Delta}\right)^{\frac{1}{k-1}}}.
\]
\end{corollary}
\end{tcolorbox}

Since every color class in a proper coloring is an independent set and $(1+\eps)^{-1} \ge 1-\eps$, Corollary~\ref{corollary: linear} implies that every $n$-vertex $k$-uniform linear hypergraph $\cH$ of maximum degree at most $\Delta \ge \Delta_0$ satisfies
\[
\alpha(\cH) \ge (1-\eps)\, n {\left(\frac{k-2}{(4k-5)(k-1)}\,\frac{\log\Delta}{\Delta}\right)^{\frac{1}{k-1}}}.
\]
This recovers the bound of~\cite[Corollary~1.5]{DMV26ind}, which resolved the following conjecture of Verstra\"ete and Wilson~\cite[Conjecture~1]{verstraete2026independent}: for every integer $k \ge 3$, there exists a constant $a_k > 0$ such that $\lim_{k \to \infty} a_k = 1$ and every $n$-vertex $k$-uniform linear hypergraph $\cH$ of maximum degree $\Delta$ satisfies
\[
\alpha(\cH) \ge (a_k - o_\Delta(1))\, n \, {\left(\frac{\log\Delta}{\Delta}\right)^{\frac{1}{k-1}}}.
\]
Indeed, the leading constant in the former bound tends to $1$ as $k \to \infty$.

The proof of Corollary~\ref{corollary: linear} applies Theorem~\ref{theorem: list chromatic} to the parts of a random partition of $\cH$ and then colors these parts using pairwise disjoint sets of colors.
This argument does not extend to list coloring: the available colors are prescribed in advance and need not be disjoint across the parts.
Thus, extending the bound in Corollary~\ref{corollary: linear} to the list-coloring setting would require new ideas.

\subsection{The color-degree setting}\label{subsection: color-degree}

Let $\cH$ be a hypergraph and let $L \colon V(\cH) \to 2^\N$ be a list assignment for $\cH$.
A useful feature of list coloring is that rather than working with $\cH$ itself, one may instead impose structural constraints on the hypergraphs $H_c \coloneqq \cH[\{v \in V(\cH) : c \in L(v)\}]$.
For instance, since $\Delta(H_c) \le \Delta(\cH)$, requiring an upper bound on $\Delta(H_c)$ for every color $c$ is weaker than imposing the same upper bound on $\Delta(\cH)$. Accordingly, there exist a number of results in the literature in which the number of available colors assigned to each vertex is a function of $\max_c \Delta(H_c)$ as opposed to $\Delta(\cH)$.
This framework, often referred to as the \emph{color-degree setting}, was pioneered by Kahn~\cite{kahn1996asymptotically}, Kim~\cite{Kim95}, and Reed~\cite{Reed}, among others. For a selection of more recent examples in the graph setting, see~\cite{BohmanHolzman, ReedSud, LohSudakov, alon2020palette, cambie2022independent, KangKelly, GlockSudakov, anderson2025coloring, anderson2024coloring, AndersonBernshteynDhawan, bradshaw2025toward}.

For higher uniformities, this setting has remained largely unexplored.
We note that both Iliopoulos's proof of Theorem~\ref{theorem: fotis} and Li and Postle's recent work on triangle-free hypergraphs~\cite{lipostle2022} extend naturally to this setting, although they elected not to state their results in this fashion.
The main technical result of this paper is a generalization of Theorem~\ref{theorem: list chromatic} in which the maximum degree condition is replaced by the weaker color-degree condition.
Before we state our result, we make a few definitions.
Let $\cH$ be a $k$-uniform hypergraph and let $L$ be a list assignment for $\cH$.
For a color $c \in L(v)$, the \emph{color degree} of a vertex $v$ with respect to $c$ is the number of edges $e \in E(\cH, v)$ such that $c \in \bigcap_{u \in e} L(u)$, where $E(\cH, v)$ denotes the set of edges containing $v$. Note that the color degree of $v$ with respect to any color is at most the degree of $v$ in $\cH$, but it may be substantially smaller when the lists are diverse.

\begin{tcolorbox}[enhanced, breakable]
\begin{theorem}\label{theorem: main}
Let $k \ge 2$ be an integer, and let $\eps > 0$ be arbitrary. There exists $d_0\in\N$ such that the following holds for every $d\ge d_0$.
Let $\cH$ be a $k$-uniform uncrowded hypergraph, and let $\cL\,:\,V(\cH) \to 2^\N$ be a list assignment for $\cH$ such that:
\begin{enumerate}[label=\textup{(C\arabic*)}]
\item\label{item: list size in main thm}
for every $v \in V(\cH)$, $|\cL(v)| \ge (1+\eps) {\left((k-1) (d/\log d)\right)^{\frac{1}{k-1}}}$; and
\item\label{item: color degree in main thm} for every $v \in V(\cH)$ and $c \in \cL(v)$, $|\{e \in E(\cH, v)\,:\,c \in \bigcap_{u \in e} \cL(u)\}| \le d$.
\end{enumerate}
Then $\cH$ admits an $\cL$-coloring.
\end{theorem}
\end{tcolorbox}

Setting $d = \Delta$ in the above theorem yields Theorem~\ref{theorem: list chromatic} as an immediate corollary.

\subsection{Palette sparsification}\label{subsection: palette}

We prove Theorem~\ref{theorem: main} via iterative applications of the Lov\'asz Local Lemma.
Implementing each application with the algorithm of Moser and Tardos~\cite{MT}, as in~\cite{iliopoulos2021improved}, yields a randomized algorithm that, in polynomial time, properly colors any $n$-vertex $k$-uniform uncrowded hypergraph of maximum degree at most $\Delta$ with $(1+o(1))\big((k-1)\Delta/\log\Delta\big)^{\frac{1}{k-1}}$ colors.

Our analysis also has implications for \emph{sublinear algorithms}, in which the available computational resources are substantially smaller than the size of the input.
Assadi, Chen, and Khanna~\cite{assadi2019sublinear} introduced a general framework for this regime, known as \emph{palette sparsification}, which works as follows.
Suppose we wish to properly color a hypergraph $\cH$ with $q$ colors, but memory constraints prevent us from storing the entire hypergraph at once.
Instead, for each vertex $v \in V(\cH)$, we independently sample a subset $L(v) \subseteq [q]$ of size $q' \ll q$ uniformly at random, and define the set of \emph{active} edges by
\[
    E_\mathrm{active} \coloneqq \left\{e \in E(\cH) \mathbin{:} \bigcap_{u \in e} L(u) \neq \emptyset\right\}.
\]
If each vertex $v$ receives a color from its list $L(v)$, then an edge can be monochromatic only if it lies in $E_\mathrm{active}$.
Consequently, the algorithm need only store and process the edges in $E_\mathrm{active}$, which typically form a much smaller set than $E(\cH)$.
To guarantee success, one must show that, with high probability, $\cH$ admits an $L$-coloring, i.e., a proper coloring in which each vertex $v$ receives a color from $L(v)$.
Statements of this form, known as \emph{palette sparsification theorems}, have proved instrumental in the design of sublinear graph algorithms; see, e.g.,~\cite{Sparsification1, alon2020palette, Sparsification2, Sparsification3, dhawan2024palette, bradshaw2025toward}.

These applications have driven considerable effort toward stronger sparsification bounds.
In their seminal paper, Assadi, Chen, and Khanna~\cite{assadi2019sublinear} established palette sparsification for $(\Delta + 1)$-coloring: in an $n$-vertex graph of maximum degree $\Delta$, sampling $\Theta(\log n)$ colors per vertex from a palette of size $\Delta + 1$ suffices for the resulting lists to admit a proper coloring with high probability.
Kahn and Kenney~\cite{KahnKenney} later sharpened this to the asymptotically tight bound of $(1+o(1))\log n$ colors per vertex.
This line of research has since grown substantially, motivated both by algorithmic applications and by purely structural interest in list coloring from random list assignments, particularly for edge-colorings; for the latter, see, e.g.,~\cite{KN1,KN2,C1,C2,C3,CH,C4,rand4,rand3,rand1,rand2}.

In the hypergraph setting, Casselgren and Eriksson~\cite{casselgren2026single} recently initiated the study of palette sparsification.
They showed that, for $k$-uniform linear hypergraphs of maximum degree $\Delta$, sampling $\Theta((\log n)^{\frac{1}{k}})$ colors per vertex from a palette of size $O(\Delta^{\frac{1}{k-1}})$ suffices, with high probability, for the hypergraph to be colorable from the sampled lists; this palette size matches the best known upper bound on the list chromatic number of linear hypergraphs.
They also showed that lists of this size do not suffice for general hypergraphs, and asked whether the palette size can be improved for linear hypergraphs.
Subsequently, Czumaj, Peng, Shi, and Sohler~\cite{czumaj2026streaming} showed that, for every $k$-uniform hypergraph, sampling $\Theta(\sqrt{\log n})$ colors per vertex from a palette of size $O(\Delta^{\frac{1}{k-1}})$ suffices, and that lists of this size are necessary in general.
For general hypergraphs, this palette size cannot be improved, as complete hypergraphs show that a palette of size $\Omega(\Delta^{\frac{1}{k-1}})$ is necessary.

As discussed in Section~\ref{subsection: background and results}, several graph classes are known to have list chromatic number $o(\Delta)$; for example, $\chi_\ell(G) = O(\Delta/\log \Delta)$ for every triangle-free graph $G$.
It is therefore natural to ask whether palette sparsification theorems hold in these settings as well.
Alon and Assadi~\cite{alon2020palette} initiated this line of study, showing that, for triangle-free graphs of maximum degree $\Delta$, sampling $O(\Delta^{\gamma}+\sqrt{\log n})$ colors per vertex from a palette of size $O(\Delta/(\gamma\log\Delta))$ suffices.
These bounds have since been extended to various locally sparse settings, following a by-now well-established template that derives palette sparsification theorems from list-coloring results in the color-degree setting~\cite{alon2020palette, anderson2025coloring, dhawan2024palette, bradshaw2025toward}.

One may ask the analogous question for hypergraphs: can the palette size $O(\Delta^{\frac{1}{k-1}})$ be reduced to $o(\Delta^{\frac{1}{k-1}})$ for natural classes of hypergraphs?
For linear hypergraphs, a positive answer would resolve the question of Casselgren and Eriksson mentioned above.
However, it is not known even whether linear hypergraphs satisfy $\chi_\ell(\cH) = o(\Delta^{\frac{1}{k-1}})$, and a color-degree version appears considerably harder.
As a corollary to Theorem~\ref{theorem: main}, we show that the answer is positive for uncrowded hypergraphs.
Specifically, let $\eps>0$ and $\gamma\in(0,1)$, and let $\cH$ be an $n$-vertex $k$-uniform uncrowded hypergraph of maximum degree at most $\Delta$, where $\Delta$ is sufficiently large and $\log n\le\Delta$.
Suppose that, for each vertex, we sample $\Delta^{\frac{\gamma}{k-1}}+C(\log n)^{\frac{1}{k}}$ colors uniformly at random from a palette of size $(1+\eps)\big((k-1)\Delta/(\gamma\log\Delta)\big)^{\frac{1}{k-1}}$, where $C>0$ depends only on $k$ and $\eps$.
Then, with high probability, $\cH$ is colorable from the sampled lists.
This partially answers the question of Casselgren and Eriksson: for uncrowded hypergraphs, it reduces the palette size in their theorem by a factor of order $(\gamma\log\Delta)^{\frac{1}{k-1}}$, at the cost of sampling more colors per vertex.
The formal statement and proof are given in Section~\ref{section: algorithms}.

\subsection{Summary and open problems}\label{subsection: summary}

In this paper, we establish improved upper bounds on the chromatic number of both uncrowded and linear hypergraphs.
For uncrowded hypergraphs, our bound matches the coupon-collector heuristic, whereas our bound for linear hypergraphs retains a loss in the leading constant.
This raises the broader question of whether the heuristic bound remains valid under weaker structural assumptions.
We propose two complementary conjectures in this direction, each strengthening Conjecture~\ref{conj: Molloy}.

First, we conjecture that the shattering-threshold bound continues to hold for linear hypergraphs, despite the possible presence of many short cycles.
This relaxes the high-girth hypothesis in Molloy's conjecture (Conjecture~\ref{conj: Molloy}).

\begin{conjecture}\label{conj: linear-shattering}
Let $k \ge 3$ be an integer, and let $\eps > 0$ be arbitrary. There exists $\Delta_0 \in \N$ such that every $k$-uniform linear hypergraph $\cH$ with maximum degree at most $\Delta \ge \Delta_0$ satisfies
\[
\chi(\cH) \le (1+\eps)
{\left((k-1)\,\frac{\Delta}{\log\Delta}\right)^{\frac{1}{k-1}}}.
\]
\end{conjecture}

The special case of Steiner triple systems is particularly interesting because of their central role in design theory. A Steiner triple system of order $N$ is a $3$-uniform linear hypergraph that is $(N-1)/2$-regular~\cite{eustis2013independence}, so Conjecture~\ref{conj: linear-shattering} would give $\chi(S)\leq(1+o(1))\sqrt{N/\log N}$. For comparison, specializing Corollary~\ref{corollary: linear} to $k=3$ gives $\chi(S)\leq(\sqrt{7}+o(1))\sqrt{N/\log N}$, which, to the best of our knowledge, is the strongest explicit asymptotic upper bound currently known for every Steiner triple system. On the other hand, Eustis and Verstra\"ete~\cite{eustis2013independence} constructed, for infinitely many admissible $N$, a Steiner triple system $S$ satisfying $\alpha(S)\leq(\sqrt{3}+o(1))\sqrt{N\log N}$. Since $\chi(S)\ge N/\alpha(S)$, their result provides the lower bound $\chi(S)\ge(1/\sqrt{3}-o(1))\sqrt{N/\log N}$. It would be interesting to determine the asymptotic value of the maximum chromatic number among Steiner triple systems of order $N$, or even to narrow the current gap between the constants $1/\sqrt{3}$ and $\sqrt{7}$.

The arguments of Duke, Lefmann, and R\"odl~\cite{duke1995uncrowded}, Frieze and Mubayi~\cite{frieze2013coloring}, and the present paper reduce the linear setting to the uncrowded setting through vertex partitioning. As discussed in Section~\ref{subsection: linear overview}, coloring polynomially many resulting induced subhypergraphs with disjoint palettes incurs a fixed loss in the leading constant. Thus, proving Conjecture~\ref{conj: linear-shattering} appears to require a method that treats short cycles directly, partitions the vertex set into only $\Delta^{o(1)}$ parts, or reuses colors across the parts.

A second direction is to assume only that the hypergraph is triangle-free.
Here, following Cooper and Mubayi~\cite{cooper2015list} and Li and Postle~\cite{lipostle2022}, a \emph{triangle} consists of three distinct edges $e,f,g$ and three distinct vertices $u,v,w$ such that $\{u,v\}\subseteq e$, $\{v,w\}\subseteq f$, $\{w,u\}\subseteq g$, and $\{u,v,w\}\cap e\cap f\cap g=\varnothing$.
Subsequent to Frieze and Mubayi's proof of the bound in~\eqref{eq:FM} for hypergraphs of girth at least $4$, Cooper and Mubayi~\cite{cooper2015list} proved an analogous bound for triangle-free hypergraphs of rank~$3$; Li and Postle~\cite{lipostle2022} recently extended this result to all ranks, and hence to every uniformity.
The argument of Li and Postle generalizes a nibble-based approach of Pettie and Su~\cite{PS15}, who showed that every triangle-free graph of maximum degree $\Delta$ has chromatic number at most $(4+o(1))\Delta/\log\Delta$.
Recall that Molloy's entropy-compression argument~\cite{Molloy} achieves the coupon-collector heuristic constant $1$ in the graph case, matching Kim's bound in~\eqref{eq:kim}.
While it remains unclear how to adapt these entropy-compression ideas directly to higher uniformities, combining Molloy's framework with techniques from recent work of Verstra\"ete and Wilson~\cite{verstraete2026independent} may yield a path toward the following conjecture:

\begin{conjecture}\label{conj: triangle-free-shattering}
Let $k \ge 3$ be an integer, and let $\eps > 0$ be arbitrary. There exists $\Delta_0 \in \N$ such that every $k$-uniform triangle-free hypergraph $\cH$ with maximum degree at most $\Delta \ge \Delta_0$ satisfies
\[
\chi(\cH) \le (1+\eps)
{\left((k-1)\,\frac{\Delta}{\log\Delta}\right)^{\frac{1}{k-1}}}.
\]
\end{conjecture}

We conclude this subsection with an algorithmic question.
When $\log n \le \Delta$, Theorem~\ref{theorem: palette} gives palette sparsification for uncrowded hypergraphs with a palette of size $o(\Delta^{\frac{1}{k-1}})$, whereas for linear hypergraphs the question of Casselgren and Eriksson remains open.
We ask whether our result extends to this more general setting.

\begin{question}
    Let $k \ge 3$ be an integer, and let $\gamma \in (0,1)$. Do there exist a constant $C > 0$ and a function $q(\Delta) = o(\Delta^{\frac{1}{k-1}})$ such that the following holds for all sufficiently large $\Delta$? Let $\cH$ be an $n$-vertex $k$-uniform linear hypergraph of maximum degree at most $\Delta$, where $\log n \le \Delta$. If each vertex samples $\Delta^{\frac{\gamma}{k-1}} + C(\log n)^{\frac{1}{k}}$ colors uniformly at random from a palette of size $q(\Delta)$, then, with high probability, $\cH$ is colorable from the sampled lists.
\end{question}

\noindent
\textbf{Note added.}
Theorem~\ref{theorem: list chromatic} was independently proved by Yu and Zhang~\cite{yu2026list} using different techniques.

\subsection{Notation and terminology}\label{subsection: notation}

Throughout the rest of the paper we use the following basic notation. For $n \in \N$, we let $[n] \defeq \set{1, \ldots, n}$.
Since a graph is a $2$-uniform hypergraph, all notation defined below for hypergraphs applies to graphs without further comment. For a hypergraph $H$, we denote its vertex and edge sets by $V(H)$ and $E(H)$, respectively. 
For an edge $e \in E(H)$, we let $V(e) \coloneqq \{v \in V(H) : v \in e\}$.
We occasionally write $v(H)$ for the number of vertices in $H$. For a subset $S \subseteq V(H)$, we write $H[S]$ for the subhypergraph of $H$ induced by $S$. Given $s \in \N$ and hypergraphs $H_1, \dots, H_s$ on the same vertex set $V$, the \emph{union} of these hypergraphs, denoted $\bigcup_{i \in [s]} H_i$, is the hypergraph with vertex set $V$ and edge set $\bigcup_{i \in [s]} E(H_i)$.

Fix $k \in \N$. We say that a hypergraph $H$ 
has rank $k$ if $2 \leq |e| \leq k$ for each $e \in E(H)$. 
Fix a rank-$k$ hypergraph $H$ and an integer $\ell \in \{2,\dots, k\}$. We let $E(H, \ell)$ denote the set of $\ell$-edges in $H$.
For a vertex $v \in V(H)$, we let $E(H, \ell, v)$ denote the set of $\ell$-edges in $H$ containing $v$ and let $\deg(H, \ell, v) = |E(H, \ell, v)|$; we call this the \textit{$\ell$-degree} of $v$.
For a vertex $v \in V(H)$, we let $E(H, v)$ denote the set of all edges in $H$ incident to $v$ and let $\deg(H, v) = |E(H, v)|$.

For a vertex $v\in V(H)$, we define the \emph{neighborhood} of $v$, denoted by $N_H(v)$, to be the set of all other vertices that share an edge with $v$ in $H$. We also let $N_H[v]\coloneqq N_H(v)\cup\{v\}$ denote the \emph{closed neighborhood} of $v$. 
For an integer $j \ge 1$, we let $N_H^j(v)$ denote the set of vertices other than $v$ at distance at most $j$ from $v$ in $H$. When the underlying hypergraph is clear, we omit the subscript $H$.

Given a finite set $U$, a \emph{list assignment on $U$} is a function $L : U \to 2^\N$; we refer to $U$ as the \emph{support} of $L$, denoted $\supp(L)$. We write $L(U) \coloneqq \bigcup_{v \in U} L(v)$. When $U = V(H)$ for a hypergraph $H$, we simply say $L$ is a list assignment \emph{for $H$}.
Fix a list assignment $L$ for a hypergraph $H$.
An \textit{$L$-coloring} is a proper coloring $\Phi : V(H) \to \N$ such that $\Phi(v) \in L(v)$ for every vertex $v \in V(H)$.
A \emph{partial $L$-coloring} of $H$ is a proper coloring $\Psi : S \to \N$ of $H[S]$, defined on some subset $S \subseteq V(H)$, such that $\Psi(v) \in L(v)$ for every $v \in S$. The set $S$ is called the \emph{support} of $\Psi$, denoted $\supp(\Psi)$. In particular, an $L$-coloring of $H$ is precisely a partial $L$-coloring with support $V(H)$.
A hypergraph $H$ is \emph{list $q$-colorable} if it admits an $L$-coloring whenever $|L(v)| \ge q$ for all $v \in V(H)$. The \emph{list chromatic number} of $H$, denoted $\chi_\ell(H)$, is the minimum $q$ such that $H$ is list $q$-colorable.

We use the following asymptotic notation. For functions $f$ and $g$ of a positive variable $x$ (usually taking values in $\N$, and often assumed to be large), we write $g(x)=O_x(f(x))$ if there exist constants $C>0$ and $x_0$ such that $|g(x)|\le C f(x)$ for all $x\ge x_0$, and write $g(x)=o_x(f(x))$ if, for every $\varepsilon>0$, there exists $x_0$ such that $|g(x)|\le \varepsilon f(x)$ for all $x\ge x_0$. 
We write $g(x)=\Omega_x(f(x))$ if $f(x)=O_x(g(x))$, $g(x)=\omega_x(f(x))$ if $f(x)=o_x(g(x))$, and $g(x)=\Theta_x(f(x))$ if both $g(x)=O_x(f(x))$ and $g(x)=\Omega_x(f(x))$ hold.  
Constants implicit in the notation are independent of the asymptotic variable and may depend on parameters that are regarded as fixed. We omit the subscript $x$ in the notations when the asymptotic variable $x$ is clear from context. 
We also write $g(x)\ll f(x)$ (equivalently, $f(x) \gg g(x)$) if $g(x)=o(f(x))$.
We use $\tilde{O}(f(x))$ to suppress logarithmic factors, i.e., $g(x) = \tilde{O}(f(x))$, if $g(x) = O(f(x) \cdot \mathrm{poly}\log f(x))$.
For convenience, we write $x = (1 \pm \beta)y$ if $(1-\beta)y \le x \le (1+\beta)y$.

For simplicity of exposition, we may  sometimes omit floors and ceilings when they do not affect the argument.

\subsection{Organization of the paper}
The rest of the paper is organized as follows.
Section~\ref{section: proof overview} gives an informal overview of the proofs of Theorem~\ref{theorem: list chromatic} and Corollary~\ref{corollary: linear}.
Section~\ref{section: prelim} collects the probabilistic tools used throughout the paper.
In Section~\ref{section: proof of main}, we prove Theorem~\ref{theorem: main}, deferring the proof of a key concentration lemma (Lemma~\ref{lemma: each E v c ell}) to Section~\ref{section: deg}.
Section~\ref{section: linear reduction} is devoted to the proof of Corollary~\ref{corollary: linear}.
Finally, Section~\ref{section: algorithms} proves the palette sparsification theorem discussed in Section~\ref{subsection: palette}.

\section{Proof overview}\label{section: proof overview}

In this section, we provide an overview of the key ideas behind the proofs of our main results.
It should be understood that the presentation in this section deliberately ignores certain minor technical issues, and so the actual arguments and formal definitions given in the rest of the paper may be slightly different from how they are described here. 
However, the differences do not affect the general conceptual framework underlying our approach.

Throughout this overview, all asymptotic notation is with respect to either $\Delta(G) \rightarrow \infty$ or $\Delta(H) \rightarrow \infty$.
Furthermore, we write $f \approx g$, $f \gtrsim g$, and $f \lesssim g$ if $f = (1 \pm o(1))\,g$, $f \ge \left(1-o(1)\right) g$, and $f \le \left(1+o(1)\right) g$, respectively.

\subsection{The R\"odl Nibble: proof of Theorem~\ref{theorem: list chromatic}}\label{subsection: nibble overview}

As mentioned earlier, we prove Theorem~\ref{theorem: list chromatic} using a variant of the R\"odl Nibble method.
We begin by discussing the classical approach due to Kim~\cite{Kim95} for coloring graphs of girth at least $5$; see \cite[Ch.~12]{MolloyReed} for a textbook treatment of the argument. 
We then discuss the extension of Iliopoulos~\cite{iliopoulos2021improved} to higher uniformities, before turning to our approach, which gives a bound matching the shattering threshold.

\subsubsection*{Graphs of girth at least $5$}
Fix $\eps > 0$. Let $G$ be a graph of maximum degree at most $\Delta$ and girth at least~$5$, and let $L$ be a list assignment of $G$ satisfying $|L(v)| \ge (1+\eps)(\Delta/\log\Delta)$ for every vertex $v \in V(G)$. Our goal is to find an $L$-coloring of $G$. To achieve this, we employ a variant of the R\"odl Nibble method for constructing a proper coloring, which typically proceeds in two phases: 

\begin{enumerate}[(1)]
\item Construct a `good' partial $L$-coloring of $G$ by randomly coloring a small portion of $V(G)$, and then iteratively repeating the procedure on the vertices that remain uncolored while updating the lists of available colors at a vertex during each step (a color $c$ is available at $v$ if no neighbor of $v$ is colored $c$).
\item Extend the partial coloring obtained from the first phase to a complete $L$-coloring.
\end{enumerate}
We note that, for this approach to work, it is crucial to design an effective nibble procedure so that, by the end of the first phase, the subgraph induced by the uncolored vertices is sufficiently sparse, while each uncolored vertex still has many available colors in its list.
More formally, let $\phi$ be the partial coloring obtained in the first phase.
For each vertex $v$ that is uncolored under $\phi$, let $L_\phi(v) \subseteq L(v)$ be a subset of available colors at $v$, and for $c \in L_{\phi}(v)$, let $d_{L_\phi}(v, c)$ denote the number of neighbors $u$ of $v$ satisfying $c \in L_\phi(u)$, i.e., $d_{L_\phi}(v, c)$ is the color degree of $c$ with respect to $v$.
At the end of the first phase, it is crucial that
\begin{align}\label{eq: graph termination}
    \frac{\max_{v, c}d_{L_\phi}(v, c)}{\min_v|L_\phi(v)|} \le \frac{1}{8}, \qquad \text{and} \qquad \min_v|L_\phi(v)| \ge 1.
\end{align}
The second phase follows by a simple Lov\'asz Local Lemma argument.

The heart of the first phase is the so-called ``Wasteful Coloring Procedure'', first introduced by Kim \cite{Kim95} (see also~\cite[Ch.~12]{MolloyReed} for a textbook treatment), for constructing a partial $L$-coloring of $G$ with desirable properties.
We proceed iteratively.
At first, all vertices of $G$ are uncolored.
During each iteration, we perform the following procedure:
\begin{enumerate}[label=\textup{(G\arabic*)}, leftmargin=40pt]
\item \label{G1} Each uncolored vertex is \emph{activated} independently with some small probability.
\item \label{G2} For each uncolored vertex $v$, select a color uniformly at random from its list $L(v)$.
\item \label{G3} A color $c$ is \emph{discarded} from $L(v)$ if some neighbor $u$ of $v$ is activated and selects the color $c$.
\item \label{G4} If the selected color of an activated vertex is not discarded, the vertex is permanently \emph{assigned} that color.
\end{enumerate}
Note that the above procedure is wasteful in the sense that we may discard a color $c$ from $L(v)$ when none of its neighbors is permanently assigned $c$. It turns out that this ``wastefulness'' greatly simplifies the analysis of this procedure (see \cite[Ch.~12]{MolloyReed} for a more in-depth discussion of the utility of such wastefulness).
More formally, let $U_i$ and $L_i$ respectively denote the set of uncolored vertices and the list assignment at the start of the $i$-th iteration. Set $G_i \coloneqq G[U_i]$. For a vertex $v \in U_i$ and a color $c \in L_i(v)$, we define the \emph{color degree} of $c$ at $v$, denoted $d_i(v, c)$, to be the number of neighbors of $v$ in $G_i$ whose lists contain the color $c$, that is, $d_i(v, c) \coloneqq |\{u \in N_{G_i}(v) : c \in L_i(u)\}|$.
We note that initially $U_1 = V(G)$, $L_1 = L$, $G_1 = G$, and $d_1(v, c) \le \Delta$ for every vertex $v \in U_1$ and color $c \in L_1(v)$.

The key quantity to track throughout the nibble procedure is the ratio of the maximum color degree $t_i \coloneqq \max_{v, c} d_i(v,c)$ to the minimum list size $q_i \coloneqq \min_{v} |L_i(v)|$. Initially, $t_1/q_1 \le \Delta/q_1 \approx \log\Delta/(1+\eps)$ and, keeping~\eqref{eq: graph termination} in mind, we run the procedure until this ratio drops below $1/8$.\footnote{We note that if we were to start with $(1-\eps)(\Delta/\log\Delta)$ colors, the first phase would fail. In particular, we cannot guarantee that the ratio drops below $1/8$ before running out of colors. This is precisely where the shattering threshold appears in the argument.}
The goal now is to show that, with positive probability, the ratio $t_{i+1}/q_{i+1}$ is noticeably smaller than $t_i/q_i$.
It turns out this holds locally in expectation:
\[
\frac{\E[d_{i+1}(v,c)]}{\E[|L_{i+1}(v)|]} \lesssim \uncolor_i \, \frac{t_i}{q_i},
\]
where $\uncolor_i$ approximates the probability that a vertex remains uncolored at the end of the $i$-th iteration, which is clearly strictly less than $1$. 
Therefore, it suffices to show that the random variables $|L_{i+1}(v)|$ and $d_{i+1}(v,c)$ are concentrated around their expected values, i.e., to show that the ratio drops locally with high probability.
An application of the Lov\'asz Local Lemma then implies that there exists an outcome of the Wasteful Coloring Procedure such that the ratio drops globally.

We note that the concentration arguments crucially rely on the girth condition: since $G$ has girth at least~$5$, any vertex $w \notin N[v]$ shares at most one common neighbor with $v$.
As a result, the random variables $|L_{i+1}(v)|$ and $d_{i+1}(v, c)$ satisfy the relevant Lipschitz constraints in order to apply Azuma-type concentration inequalities.

\subsubsection*{Uncrowded hypergraphs: Iliopoulos's proof of Theorem~\ref{theorem: fotis}}

Let $\cH$ be a $k$-uniform uncrowded hypergraph of maximum degree at most $\Delta$ as in the statement of Theorem~\ref{theorem: fotis}, and let $\cL$ be a list assignment for $\cH$ satisfying $|\cL(v)| \ge (1+\eps)(k-1){\left(\Delta/\log\Delta\right)^{\frac{1}{k-1}}}$ for each vertex $v \in V(\cH)$.
When adapting Kim's argument to higher uniformities, discarding a color $c$ at $v$ if a single neighbor of $v$ is activated and selects $c$ is rather wasteful.
For example, consider an edge $e = \{u, v, w\}$.
If $u$ is assigned $c$ during an iteration of the Wasteful Coloring Procedure, then $c$ need not be discarded at $v$ or $w$; instead, the edge $e$ can simply be replaced by the $2$-edge $\{v, w\}$ while ensuring that $v$ and $w$ do not both get assigned $c$ in future iterations.
Notably, this preserves $c$ at both vertices for later iterations of the nibble while ensuring that the union of the color assignments constructed forms a partial $\cL$-coloring of the original hypergraph.
This is precisely the strategy introduced by Iliopoulos.
Informally, he proceeds as follows: 
\begin{enumerate}[label=\textup{(H\arabic*)}, leftmargin=40pt]
\item\label{item: activate hypergraph} Each uncolored vertex is \emph{activated} independently with some small probability.
\item\label{item: assign hypergraph} For each uncolored vertex $v$, select a color uniformly at random from its list $\cL(v)$.
\item\label{item: discard hypergraph} We \emph{discard} a color $c$ from $\cL(v)$ if there is an edge containing $v$ such that every other vertex in the edge is activated and has selected the color $c$.
\item\label{item: color hypergraph} If the selected color of an activated vertex is not discarded, the vertex is permanently \emph{assigned} that color.
\end{enumerate}
As in the graph case, we let $U_i$ and $L_i$ respectively denote the set of uncolored vertices and the list assignment at the start of the $i$-th iteration. 
As mentioned earlier, in the situation where some vertices in an edge $e$ are permanently assigned $c$ while others remain uncolored, we replace $e$ with the edge $f\subseteq e$ consisting of the uncolored vertices in $e$, and ensure that not all of the vertices in $f$ are assigned $c$ in future iterations.
More generally, we say that a $j$-edge $e$ \emph{becomes} an $\ell$-edge by the end of an iteration if $e$ has exactly $j-\ell$ of its vertices assigned the color $c$, while $c$ also remains in the lists of its remaining $\ell$ uncolored vertices.
We track this dependence via the notion of the \emph{color $\ell$-degree}.
For a vertex $v \in U_i$ and a color $c \in L_i(v)$, the color $\ell$-degree of a color $c$ at a vertex $v$, denoted $d_i(v, c, \ell)$, counts the edges $e \in E(H, v)$ satisfying:
\begin{itemize}
\item exactly $\ell$ vertices $w$ in $e$ (including $v$) are uncolored and satisfy $c\in L_i(w)$, and
\item the remaining $k - \ell$ vertices have been permanently assigned the color $c$ during some earlier iteration of the procedure.
\end{itemize}
Note that initially $U_1 = V(\cH)$, $L_1 = \cL$, $d_1(v, c, k) \le \Delta$, and $d_1(v, c, \ell) = 0$ for $2\le\ell<k$.
Consequently, one must track not just a single maximum color degree, but all the maximum color $\ell$-degrees, given by
\begin{equation}\label{def: overview t i ell}
t_{i,\ell} \coloneqq \max_{v,c} d_i(v,c,\ell),
\end{equation}
for $2 \le \ell \le k$, throughout the nibble procedure.
Similar to the graph case, we set $q_i \coloneqq \min_v |L_i(v)|$ for each $i$.
Analogous to~\eqref{eq: graph termination}, Iliopoulos~\cite{iliopoulos2021improved} noted that it suffices to show that after a sufficient number $s$ of iterations, we have
\begin{align}\label{eq: iliopoulos ratio}
\max_{2 \le \ell \le k} \frac{t_{s,\ell}}{q_s^{\ell-1}} \le \frac{1}{10k^2} \qquad \text{and} \qquad q_s \ge 1.
\end{align}
At this point, one can extend the resulting partial coloring to a complete coloring by using the asymmetric version of the Lov\'asz Local Lemma.
We note that the symmetric version does not apply here as the probabilities of the ``bad events'' occurring vary with $\ell$.

\subsubsection*{Our main contribution}
The key distinction between our approach and Iliopoulos's lies in how we handle, at each iteration $i$ of the nibble, the ratios $t_{i,\ell}/q_i^{\ell-1}$ for $2\le\ell\le k$. Iliopoulos controls these quantities through their maximum; in particular, the bound on $\max_{2\le\ell\le k}t_{s,\ell}/q_s^{\ell-1}$ in~\eqref{eq: iliopoulos ratio} is essential to his final application of the asymmetric version of the Lov\'asz Local Lemma. We show, however, that it suffices to control their sum directly. More precisely, although the individual ratios may vary with $\ell$, the quantity $\sum_{\ell=2}^k t_{i,\ell}/q_i^{\ell-1}$ can be tracked throughout the nibble. For an appropriate choice of $s$, we show that after $s-1$ iterations,
\begin{align}\label{eq: our termination condition}
\sum_{\ell = 2}^k \frac{t_{s,\ell}}{q_s^{\ell-1}} = o(1) \qquad \text{and} \qquad q_s = \omega(1),
\end{align}
at which point, we may complete the coloring via the asymmetric version of the Local Lemma.

To explain how we control, at each iteration $i$, the sum $\sum_{\ell=2}^k t_{i,\ell}/q_i^{\ell-1}$, we now describe the parameters tracked during the nibble and give an intuitive account of their evolution from one iteration to the next.
We remark that our arguments are considerably more streamlined relative to earlier approaches.
This streamlining simplifies the recursive analysis of parameters, making the computations significantly less cumbersome.
Recall the Wasteful Coloring Procedure~\ref{item: activate hypergraph}--\ref{item: color hypergraph}.
Throughout the nibble, we use the same activation probability $\alpha=\eps/\log\Delta$ in step~\ref{item: activate hypergraph}, where $\eps$ is an arbitrarily small positive constant.
We let $\keep_i$ denote the probability that a color is kept at a vertex during the $i$-th iteration. After step~\ref{item: discard hypergraph}, we use independent equalizing coin flips to ensure that this probability is the same for every vertex--color pair; an explicit formula for $\keep_i$ is given in~\eqref{eq: keep_i overview}.
With this notation, we make the following recursive definitions:
\begin{align*}
    &t_1 \coloneqq \Delta, & & q_1 \coloneqq (1+\eps){\left((k-1) \, \frac{\Delta}{\log\Delta} \right)^{\frac{1}{k-1}}}, \\
    &t_{i+1} = \keep_i^{k-1}t_i, & & q_{i+1} = \keep_iq_i.
\end{align*}
Note that 
\begin{align}\label{eq: ratio_i ratio_1 in overview}
\frac{t_i}{q_i^{k-1}} = \frac{t_1}{q_1^{k-1}} = \frac{\log\Delta}{(1+\eps)^{k-1}(k-1)}.
\end{align}
This fact will be useful for subsequent computations.

Let us provide some intuition for the above parameters.
Since $\keep_i$ is the probability that a given color is kept at a vertex during the $i$-th iteration, the recurrence $q_{i+1}=\keep_iq_i$ describes the corresponding reduction in list size.
Thus, $q_i$ roughly represents the list size at the start of the $i$-th iteration.
The parameter $t_i$ roughly represents the number of edges containing a fixed vertex $v$ such that all other $k-1$ vertices have $c$ in their list at the start of the $i$-th iteration (regardless of whether those vertices are uncolored).
Note that a vertex is colored during step~\ref{item: color hypergraph} if it is activated and its selected color is kept at the vertex.
It follows that the probability a vertex remains uncolored at the end of the iteration is 
\begin{equation}
\label{uncoloriapprox}
  \uncolor_i \approx 1 - \alpha\keep_i.
\end{equation}

We will now express $t_{i, \ell}$ in terms of the parameters introduced thus far.
For an edge $e$ counted in $t_i$ to be included in $t_{i, \ell}$, it must be the case that $\ell-1$ vertices $u \in e - v$ remained uncolored and the remaining $k - \ell$ vertices $w \in e - v$ were assigned $c$ during some iteration $j_w < i$.
We will now estimate the probabilities of these events occurring.
\begin{itemize}
    \item For a vertex $u$ to be uncolored at the start of the $i$-th iteration, it must remain uncolored during each previous iteration.
    This occurs with probability approximately $\prod_{j < i}\uncolor_j$.
    \item For a vertex $w$ to be assigned $c$ during the $j_w$-th iteration, it must be the case that $w$ remained uncolored up until the $j_w$-th iteration during which it was assigned $c$.
    This occurs with probability approximately $\frac{\alpha}{q_{j_w}}\prod_{j < j_w}\uncolor_j$.
\end{itemize}
Letting 
\[
\beta_i \coloneqq \prod_{j < i}\uncolor_j \qquad \text{and} \qquad \gamma_i = \sum_{j < i}\beta_j,
\]
and noting that $(q_i)$ is a decreasing sequence, this suggests the heuristic estimate
\begin{align}\label{eq: ell-degree bound overview}
    t_{i, \ell} \approx t_i \, \binom{k-1}{\ell - 1} \, \beta_i^{\ell-1} {\left[\sum_{j_w = 1}^{i-1}{\left(\frac{\alpha}{q_{j_w}} \beta_{j_w}\right)}\right]}^{k-\ell} \le t_i \, \binom{k-1}{\ell - 1} \, \beta_i^{\ell-1} {\left(\frac{\alpha\gamma_i}{q_i}\right)^{k-\ell}}.
\end{align}
Note that for $i = 1$, the above is $0$ if $\ell < k$ and $t_1$ if $\ell = k$.
A central part of our argument shows that indeed the above bound holds throughout our procedure. We next argue that, for a suitable choice of $s$, the bound in~\eqref{eq: ell-degree bound overview} implies the desired bound in~\eqref{eq: our termination condition} after $s-1$ iterations.
Recall that, by~\eqref{eq: ratio_i ratio_1 in overview}, $t_i/q_i^{k-1} = t_1/q_1^{k-1}$ for every $i\in[s]$.
Applying~\eqref{eq: ell-degree bound overview}, we have
\begin{align*}
    \sum_{\ell = 2}^k\frac{t_{s, \ell}}{q_s^{\ell-1}} \;\overset{\eqref{eq: ell-degree bound overview}}{\le}\; \sum_{\ell = 2}^k\frac{t_s}{q_s^{\ell - 1}} \, \binom{k-1}{\ell - 1} \, \beta_s^{\ell-1} {\left(\frac{\alpha\gamma_s}{q_s}\right)^{k-\ell}}
    &\,=\, \frac{t_s}{q_s^{k-1}}\sum_{\ell = 2}^k\binom{k-1}{\ell - 1} (\alpha\gamma_s)^{k-\ell} \beta_s^{\ell-1}\\
    &\,=\, \frac{t_1}{q_1^{k-1}}\left((\alpha \gamma_{s} + \beta_s)^{k-1} - (\alpha\gamma_s)^{k-1}\right) \\
    &\overset{\eqref{eq: ratio_i ratio_1 in overview}}{\le} 
    (\log \Delta) (\alpha\gamma_s+\beta_s)^{k-2}\beta_s,
\end{align*}
where the last inequality follows from~\eqref{eq: ratio_i ratio_1 in overview} and the convexity of $x^{k-1}$ for $k \ge 2$ and $x \ge 0$.
Note that, by~\eqref{uncoloriapprox}, if $\keep_i$ is at least $1 - \eps$ for every $i < s$, then $\uncolor_i \lesssim 1 - \alpha$ and, consequently, $\beta_i \lesssim (1-\alpha)^{i - 1}$ and $\gamma_i \lesssim (1 - (1-\alpha)^{i-1})/\alpha \le 1/\alpha$ for every $i \le s$. 
With this assumption in hand, the above expression satisfies
\begin{align*}
\sum_{\ell=2}^k\frac{t_{s,\ell}}{q_s^{\ell-1}}
\,\lesssim\, (\log\Delta)\left(1-(1-\alpha)^{s-1}+(1-\alpha)^{s-1}\right)^{k-2}\beta_s 
&\,\lesssim\, (\log\Delta)(1-\alpha)^{s-1}\\
&\,\le\, (\log\Delta)\exp\left(-\alpha(s-1)\right),
\end{align*}
suggesting the choice
\begin{equation}
s \coloneqq 1+\left\lceil\frac{(1+\eps)\log \log \Delta}{\alpha}\right\rceil.
\end{equation}
Indeed, for this choice of $s$, the last expression is $O((\log\Delta)^{-\eps}) \ll 1$. Moreover, a formal argument shows that $\keep_i\ge1-\eps$ for every $i<s$, verifying the assumption above. It follows that the first condition in~\eqref{eq: our termination condition} holds.

Let us now turn our attention to $q_s$.
We begin by computing $\keep_i$.
Recall from step~\ref{item: discard hypergraph} that a color $c$ is discarded at $v$ if there is an edge containing $v$ such that all its other vertices are activated and have selected the color $c$.
By a union bound, this occurs with probability at most
\begin{align*}
    \sum_{\ell = 2}^kt_{i, \ell} \, {\left(\frac{\alpha}{q_i}\right)^{\ell-1}} \;\overset{\eqref{eq: ell-degree bound overview}}{\le}\; \sum_{\ell = 2}^kt_i \, \binom{k-1}{\ell - 1} \, \beta_i^{\ell-1} {\left(\frac{\alpha\gamma_i}{q_i}\right)^{k-\ell}} {\left(\frac{\alpha}{q_i}\right)^{\ell-1}}
    &=\, \frac{\alpha^{k-1}t_i}{q_i^{k-1}}\sum_{\ell = 2}^k\binom{k-1}{\ell - 1} \gamma_i^{k-\ell} \beta_i^{\ell-1}\\
    &=\, \frac{\alpha^{k-1}t_i}{q_i^{k-1}}\,{\left((\gamma_i + \beta_i)^{k-1} - \gamma_i^{k-1}\right)} \\
    &=\, \frac{\alpha^{k-1}t_i}{q_i^{k-1}} \, {\left(\gamma_{i+1}^{k-1} - \gamma_i^{k-1}\right)},
\end{align*}
where we use the upper bound on $t_{i, \ell}$ from~\eqref{eq: ell-degree bound overview} in the first inequality.
Consequently, by employing equalizing coin flips, we can set
\begin{equation}\label{eq: keep_i overview}
    \keep_i = 1 - \frac{\alpha^{k-1}t_i}{q_i^{k-1}} \, {\left(\gamma_{i+1}^{k-1} - \gamma_i^{k-1}\right)}.
\end{equation}
Using~\eqref{eq: keep_i overview} and that $t_i/q_i^{k-1} = t_1/q_1^{k-1}$ for $i < s$, we have
\[
q_s \ge q_1\prod_{i < s}\keep_i \approx q_1\exp\left(-\sum_{i < s}\frac{\alpha^{k-1}t_i}{q_i^{k-1}}\left(\gamma_{i+1}^{k-1} - \gamma_i^{k-1}\right)\right) = q_1\exp\left(-\frac{\alpha^{k-1}t_1\gamma_s^{k-1}}{q_1^{k-1}}\right).
\]
As remarked earlier, $\gamma_s \lesssim 1/\alpha$. Thus, we expect
\begin{align}
q_s
&\gtrsim q_1\exp\left(-\frac{t_1}{q_1^{k-1}}\right)
\overset{\eqref{eq: ratio_i ratio_1 in overview}}{=}
q_1\exp\left(-\frac{\log\Delta}
{(1+\eps)^{k-1}(k-1)}\right)\\
&=\Delta^{\frac{1-(1+\eps)^{-(k-1)}}{k-1}-o(1)}
\gg\Delta^{\frac{\eps}{2(k-1)}}\gg1,
\end{align}
as desired.

If $1+\eps$ were replaced by $1-\eps$ in the definition of $q_1$, the exponent of $\Delta$ in the last display would become
\[\frac{1 - o(1) - (1-\eps)^{-(k-1)}}{k - 1} < 0,\]
so we would run out of colors before the finishing condition is reached.
This calculation explains why a sharp transition occurs at the shattering threshold in our nibble.

We conclude this overview with a discussion regarding our concentration arguments.
The concentration of list sizes is fairly standard across applications of the nibble method.
To bound the maximum color $\ell$-degree $t_{i,\ell}$ given in~\eqref{def: overview t i ell}, however, the traditional approach (see, e.g.,~\cite{iliopoulos2021improved, lipostle2022}) proceeds as follows, for a fixed vertex--color pair $(v,c)$:
\begin{enumerate}
\item Express the color $\ell$-degree $d_{i+1}(v,c,\ell)$ as a sum of $k-\ell+1$ random contributions, one for each $\ell\leq j\leq k$, where the contribution corresponding to $j$ counts the $j$-edges at $(v,c)$ that become $\ell$-edges by the end of the $i$-th iteration.
\item Compute $\E[d_{i+1}(v,c,\ell)]$ by summing the expected values of these contributions.
\item Because these contributions cannot be concentrated directly, upper bound each by a linear combination of $\Theta(1)$ auxiliary random variables, each of which is amenable to concentration via Azuma-type martingale inequalities.
\end{enumerate}
The resulting bounds are then made to hold simultaneously for every vertex--color pair, yielding a bound on $t_{i+1,\ell}$.
In contrast, our approach yields a significantly more streamlined analysis.
This is particularly useful for simplifying the recursive computations, which can become quite involved even in the graph setting ($k=2$).
To this end, we establish substantially tighter concentration bounds for both the list sizes and the color $\ell$-degrees than those appearing in previous work.
While this tighter concentration leads to a far less cumbersome recursive analysis (see Section~\ref{subsection: properties}), the underlying concentration arguments require greater delicacy.
Specifically, computing only $\E[d_{i+1}(v,c,\ell)]$ no longer suffices; instead, we must track the expectations of $\Theta(k-\ell+1)$ auxiliary random variables to ensure that they remain well controlled relative to our refined concentration error terms.

\subsection{Random partitioning: proof of Corollary~\ref{corollary: linear}}\label{subsection: linear overview}

We reduce the linear case to the uncrowded setting (Theorem~\ref{theorem: list chromatic}) via a two-stage random partition.
In the first stage, we partition the vertex set independently at random into
\[
m_1 = \Delta^{\frac{3}{4k-5} + o(1)}
\]
parts.
By a standard application of the Chernoff Bound, the maximum degree in each induced subhypergraph drops to $(1+o(1))\Delta / m_1^{k-1}$.
To control the density of short cycles, we express the number of $3$- and $4$-cycles containing a fixed vertex within its assigned part as a low-degree polynomial in the assignment indicator variables.
We verify that this polynomial representation satisfies the hypotheses required for concentration via the Kim--Vu Inequality.
An application of the symmetric version of the Lov\'asz Local Lemma then guarantees a partition in which every vertex is contained in at most $\mathrm{polylog}(\Delta)$ short cycles.

In the second stage, each part from the first stage is further partitioned independently into $m_2 = (\log\Delta)^5$ parts.
The maximum degree bounds are preserved via Chernoff bounds.
Moreover, the probability that any remaining short cycle at a vertex stays intact within its part is bounded by a reciprocal polynomial in $\log\Delta$ of sufficiently high degree.
By the general form of the Lov\'asz Local Lemma, there exists a joint partition outcome in which every part simultaneously respects the degree bound and contains no short cycles. Consequently, each part induces an uncrowded hypergraph of maximum degree at most $(1+o(1))\Delta / (m_1 m_2)^{k-1}$.

Finally, we color each induced uncrowded subhypergraph independently using pairwise disjoint color palettes via Theorem~\ref{theorem: list chromatic}.
Optimizing the choice of $m_1$ against the resulting maximum degrees yields the leading factor
\[
\left(\frac{(4k-5)(k-1)}{k-2}\right)^{\frac{1}{k-1}},
\]
which approaches $1$ as $k \to \infty$, as desired.

\section{Probabilistic tools}
\label{section: prelim}

Throughout the paper, we use several probabilistic tools, which will be introduced now in this section. We start with the celebrated Lov\'asz Local Lemma.
In its most general form, it can be stated as follows.

\begin{theorem}[{Lov\'asz Local Lemma, general version; \cite[Ch.~19]{MolloyReed}}]\label{theorem: general LLL}
Let $A_1, \dots, A_n$ be events in a probability space. For each $i$, let $D_i \subseteq [n] \setminus\{i\}$ be a set such that $A_i$ is mutually independent of $\{A_j : j \notin D_i \cup \{i\}\}$. 
If there exist real numbers $x_1,\dots,x_n \in [0, 1)$ such that $\pr(A_i) \le x_i \prod_{j \in D_i} (1-x_j)$ for all $i$, then the probability that none of the events $A_1$, \dots, $A_n$ occur is at least $\prod_{i=1}^n (1-x_i) > 0$.
\end{theorem}

For many applications, the following two versions of the Lov\'asz Local Lemma suffice.

\begin{theorem}[{Lov\'asz Local Lemma, symmetric version; \cite[Ch.~4]{MolloyReed}}]\label{theorem: sym LLL}
    Let $A_1$, $A_2$, \ldots, $A_n$ be events in a probability space. Suppose there exists $p \in [0, 1)$ such that for all $1 \leq i \leq n$ we have $\pr(A_i) \leq p$. Further suppose that each $A_i$ is mutually independent of a set of all but at most $d_{\mathrm{LLL}}$ of other events $A_j$, $j\neq i$, for some $d_{\mathrm{LLL}} \in \N$. If $4pd_{\mathrm{LLL}} \leq 1$, then with positive probability none of the events $A_1$, \ldots, $A_n$ occur.
\end{theorem}

\begin{theorem}[{Lov\'asz Local Lemma, asymmetric version; \cite[Ch.~19]{MolloyReed}}]\label{theorem: LLL}
Let $A_1, \dots, A_n$ be events in a probability space. For each $i$, let $D_i \subseteq [n] \setminus\{i\}$ be a set such that $A_i$ is mutually independent of $\{A_j : j \notin D_i \cup \{i\}\}$. 
If $\pr(A_i) \le 1/4$ and $\sum_{j \in D_i} \pr(A_j) \le 1/4$ for all $i$, then with positive probability none of the events $A_1$, \dots, $A_n$ occur.
\end{theorem}

In addition to the Local Lemma, we will use a few standard concentration inequalities. The first is a generalization of the Chernoff Bound for a sum of independent $\{0,1\}$-valued random variables, stated below (although we only apply it in the binomial case).

\begin{theorem}[{Chernoff Bound; \cite[Theorem 2.3]{mcdiarmid1998concentration}}]\label{theorem: chernoff}
    Let $X$ be a sum of independent random variables, each of which takes values in 
    $\{0,1\}$.
    Then, for any $t > 0$, we have
    \begin{align*}
        \pr{\left(X - \E[X] \ge t\right)} \le \exp\left(-\frac{t^2}{2\E[X]+2t/3}\right). 
    \end{align*}
\end{theorem}

We also need the following extension of the Chernoff Bound to negatively correlated random variables.
We say that $\{0,1\}$-valued random variables $X_1, \ldots, X_n$ are \emph{negatively correlated} if, for every $I \subseteq [n]$,
\[
\E{\left[\prod_{i \in I} X_i\right]} \le \prod_{i \in I} \E[X_i].
\]
Panconesi and Srinivasan~\cite{panconesi1997randomized} showed that the standard moment-generating-function proof of the Chernoff Bound for independent random variables remains valid under this weaker hypothesis.

\begin{theorem}[{Chernoff Bound for negatively correlated variables; \cite{panconesi1997randomized}}]\label{lemma:NegativeChernoff}
Let $X_1, \dots, X_n$ be negatively correlated random variables taking values in $\{0,1\}$, and set $X \coloneqq \sum_{i=1}^n X_i$. Then, for every $t > 0$,
\[
\pr(X - \E[X] \ge t) \le \exp\left(-\frac{t^2}{2\E[X] + 2t/3}\right).
\]
\end{theorem}

We also use Talagrand's Inequality to establish concentration. Before we state the inequality, we need a few definitions. 

\begin{definition}\label{def: lipshcitz and certifiable}
    Let $X$ be a random variable defined as a function of $n$ independent trials $T_1, \ldots, T_n$. We say $X$ is \textit{$\mu$-Lipschitz} if changing the outcome of any single trial changes $X$ by at most $\mu$, i.e., if $(t_1, \ldots, t_n)$ and $(t_1', \ldots, t_n')$ differ in exactly one coordinate, then
    \[|X(t_1, \ldots, t_n) - X(t_1', \ldots, t_n')| \leq \mu.\]
    
    We say $X$ is \textit{$r$-certifiable} if for every integer $s \ge 1$, whenever $X \ge s$, there exists a set of $t \leq rs$ trials $T_{i_1}, \ldots, T_{i_t}$ whose outcomes \textit{certify} $X \ge s$, i.e., changing the outcomes of any of the other trials cannot cause $X$ to be less than $s$.
\end{definition}

One can view $r$-certifiability as follows: in order to ``prove'' to someone that $X \ge s$, it is enough to show them just the outcomes of $T_{i_1}, \ldots, T_{i_t}$ as opposed to all outcomes $T_1, \ldots, T_n$.
We are now ready to state Talagrand's Inequality.

\begin{theorem}[{Talagrand's Inequality; \cite{molloy2014colouring}}]\label{theorem: Talagrand}
    Let $X$ be a random variable taking values in the non-negative integers, not identically $0$, and defined as a function of $n$ independent trials $T_1, \ldots, T_n$. Suppose that $X$ is $\mu$-Lipschitz and $r$-certifiable for some $\mu$, $r > 0$. 
    Then, for any $t \ge 0$, we have
    \begin{align*}
        \pr{\left(|X-\E[X]| \ge t + 20\mu\sqrt{r\E[X]} + 64\mu^2r\right)} \leq 4\exp\left(-\frac{t^2}{8\mu^2r(\E[X] + t)}\right).
    \end{align*}
\end{theorem}

Finally, we will need a powerful concentration inequality for low-degree polynomials of independent random variables, due to Kim and Vu~\cite{kimvu2000concentration}. 
Let us first recall the setting and notation needed for the result. 
Let $X_1,\dots, X_n$ be $\{0,1\}$-valued random variables, and let $f : \R^n \to \R$ be a multilinear polynomial, i.e., a polynomial in which each variable appears with degree at most $1$. We note that it is always possible to represent such a polynomial $f$ using a weighted hypergraph $H$ with vertex set $V(H) = [n]$, edge set $E(H)$, and weights $w_e \in \R$ for each edge $e\in E(H)$ as
\[
f(x_1,\dots,x_n) = \sum_{e \in E(H)} w_e \prod_{i \in e} x_i.
\]
Consider $Y \coloneqq f(X_1,\dots,X_n)$.
As $f$ is a function on $\R^n$, for any subset $A \subseteq [n]$, we let $Y_A$ be the $|A|$-fold partial derivative with respect to all indices from $A$.
Then, we have 
\[
Y_A \coloneqq \frac{\partial^{|A|}f}{\prod_{i \in A} \partial x_i} (X_1,\dots,X_n) = \sum_{\substack{e \in E(H)\\e \supseteq A}} w_e \prod_{i \in e \setminus A} X_i.
\]
We are now ready to state the result. 

\begin{theorem}[{Kim--Vu Polynomial Concentration; \cite{kimvu2000concentration}}]
\label{theorem: Kim Vu concentration}
Let $X_1, \dots, X_n$ be independent random variables taking values in $\{0,1\}$, and let $H$ be a 
weighted hypergraph with vertex set $V(H) = [n]$, edge size at most $k$, and positive weight $w_e$ for each edge $e\in E(H)$. Consider the multilinear polynomial  
\[
f(x_1, \dots, x_n) \coloneqq \sum_{e \in E(H)} w_e \prod_{i \in e} x_i,
\]
and let $Y \coloneqq f(X_1,\dots,X_n)$. For each $A \subseteq [n]$ with $|A| \le k$, let $Y_A$ be the $|A|$-fold partial derivative of $f$ with respect to all indices from $A$, evaluated at $(X_1,\dots,X_n)$.
For $0 \le i \le k$, let $E_i \coloneqq 
\max_{A \in \binom{[n]}{i}} \E[Y_A]$. Set $E \coloneqq \max_{0 \le i \le k} E_i$ and $E' \coloneqq \max_{1 \le i \le k} E_i$. Then, for any $\lambda > 1$, we have
\[
\pr{\left(|Y-\E[Y]| > (8\lambda)^k \sqrt{k!\,E E'} \right)} \le 2e^{-\lambda+2} n^{k-1}.
\]
\end{theorem}

\section{Coloring uncrowded hypergraphs: proof of Theorem~\ref{theorem: main}}\label{section: proof of main}

In this section, we prove Theorem~\ref{theorem: main}.
For the reader's convenience, we restate the result here.

\begin{theorem*}[Restatement of Theorem~\ref{theorem: main}]
Let $k \ge 2$ be an integer, and let $\eps > 0$ be arbitrary. There exists $d_0\in\N$ such that the following holds for every $d\ge d_0$.
Let $\cH$ be a $k$-uniform uncrowded hypergraph, and let $\cL\,:\,V(\cH) \to 2^\N$ be a list assignment for $\cH$ such that:
\begin{enumerate}[label=\textup{(C\arabic*)}]
\item for every $v \in V(\cH)$, $|\cL(v)| \ge (1+\eps) {\left((k-1) (d/\log d)\right)^{\frac{1}{k-1}}}$; and
\item for every $v \in V(\cH)$ and $c \in \cL(v)$, $|\{e \in E(\cH, v)\,:\,c \in \bigcap_{u \in e} \cL(u)\}| \le d$.
\end{enumerate}
Then $\cH$ admits an $\cL$-coloring.
\end{theorem*}

Fix $k \ge 2$ and $\eps > 0$. Without loss of generality, we may assume that $\eps<1$. Choose $d_0$ sufficiently large in terms of $k$ and $\eps$, and fix $d \ge d_0$.
Let $\cH$ be a $k$-uniform uncrowded hypergraph, and let $\cL$ be a list assignment for $\cH$ satisfying conditions~\ref{item: list size in main thm} and~\ref{item: color degree in main thm} in the statement of Theorem~\ref{theorem: main}.
Throughout the proof, we let
\begin{align}\label{def: eta, alpha, delta, s}
\eta \coloneqq \frac{\eps}{6}, \qquad \alpha = \frac{\eta}{\log d}, \qquad \delta \coloneqq d^{-\frac{\eta/6}{k-1}},  \qquad \text{and} \qquad s \coloneqq \left\lceil \frac{(1+\eta) \log d \log\log d}{\eta(1-\eta)} \right\rceil.
\end{align}
We define the sequences $(t_i)_{i \in [s]}$ and $(q_i)_{i \in [s]}$ by setting 
\begin{align}\label{def: t_1, q_1}
t_1 \coloneqq d \qquad \text{and} \qquad q_1 \coloneqq (1+\eps) {\left((k-1) \, \frac{d}{\log d}\right)^{\frac{1}{k-1}}},
\end{align}
and, for $i \in [s-1]$, 
\begin{align}\label{def: t_i, q_i}
t_{i+1} \coloneqq (1 + \delta)^{k-1} \cdot t_i \cdot  \keep_i^{k-1} \qquad \text{and} \qquad q_{i+1} \coloneqq (1 - \delta) \cdot q_i \cdot \keep_i,
\end{align}
where
\begin{align}\label{def: keep_i}
\keep_i \coloneqq 
1 - \frac{\alpha^{k-1} t_i}{q_i^{k-1}}\,{\left(\gamma_{i+1}^{k-1} - \gamma_i^{k-1}\right)},
\end{align}
for 
\begin{align}\label{def: uncolor_i, beta_i, gamma_i}
\uncolor_i \coloneqq 1 - \alpha \keep_i, \qquad
\beta_i \coloneqq \prod_{j=1}^{i-1} \uncolor_j, \qquad \text{and} \qquad \gamma_i \coloneqq \sum_{j=1}^{i-1} \beta_j.
\end{align}
We note that $\beta_1 = 1$ and $\gamma_1 = 0$.
Furthermore,
these values are all well-defined. Indeed, $\beta_1$ defines $\gamma_2$, which (together with $t_1$ and $q_1$) in turn defines $\keep_1$, and hence $\uncolor_1$, $t_2$, and $q_2$. Since $\uncolor_1$ defines $\beta_2$, which in turn defines $\gamma_3$, it follows that $\keep_2$ is well-defined, and so on.
We refer the reader to the proof overview Section~\ref{section: proof overview} for an informal motivation behind the definitions of these parameters.

The following lemma states certain properties of our parameters which will be useful in the proof of Theorem~\ref{theorem: main}; we defer the proof of Lemma~\ref{lemma: properties of beta_i, gamma_i, keep_i, uncolor_i} to Section~\ref{subsection: properties}.

\begin{lemma}\label{lemma: properties of beta_i, gamma_i, keep_i, uncolor_i}
For $i \in [s-1]$, the following hold:
\begin{enumerate}[label=\textup{(S\arabic*)}]
\item \label{gamma_i bound} $\gamma_1=0$ and $1 \le \gamma_{i+1} \le \dfrac{\log d}{\eta(1-\eta)}$;
\item \label{keep_i bound} $1- \eta \le \keep_i = 1-\Omega{\left((\log d)^{-k}\right)}$; 
\item \label{uncolor_i bound} $1 - \alpha \le \uncolor_i \le 1 - \alpha(1-\eta)$.
\end{enumerate}
Moreover, for each $i \in [s]$, the following hold:
\begin{enumerate}[label=\textup{(S\arabic*)},resume]
\item \label{beta_i bound} $(\log d)^{-2} \le \beta_i \le (1-\alpha(1-\eta))^{i-1}$;
\item \label{ratio bound} $\dfrac{\log d}{(1+\eps)^{k-1}(k-1)} \le \dfrac{t_i}{q_i^{k-1}} \le \dfrac{\log d}{(1+5\eta)^{k-1}(k-1)}$;
\item \label{q_i bound} $q_i \gg d^{\frac{\eta}{k-1}}$.
\end{enumerate}
\end{lemma}

For each $i \in [s-1]$, by~\ref{keep_i bound} and the definition of $\delta$ given in~\eqref{def: eta, alpha, delta, s}, we note that, for $d$ sufficiently large
\[
(1-\delta) \keep_i \le (1+\delta) \keep_i \le 1.
\]
By the recursive definition of $t_i$ and $q_i$ given in~\eqref{def: t_i, q_i}, it follows that the two sequences $(t_i)_{i \in [s]}$ and $(q_i)_{i \in [s]}$ are decreasing, and so for all $i \in [s]$, we have
\begin{align}\label{eq: compare t_i t_1, q_i q_1}
t_i \le t_1 \qquad \text{and} \qquad q_i \le q_1.
\end{align}

With these parameters in hand, we now make several definitions which will assist in the proof of Theorem~\ref{theorem: main}.
We begin with families of restricted hypergraphs.

\begin{definition}[Family of restriction hypergraphs]\label{def: restriction family}
Let $L$ be a list assignment on some finite set $U$. A \emph{family of restriction hypergraphs with respect to $L$} is a collection of hypergraphs $\cF = \{F_c\,:\, c \in L(U)\}$ with vertex set $U$ such that, for every $c \in L(U)$ and every $e \in E(F_c)$, $c \in \cap_{u \in e}L(u)$.
\end{definition}

We will utilize the above definition when considering the subhypergraphs induced by uncolored vertices containing some specific color $c$ in their lists.
This greatly simplifies the exposition surrounding color degrees.

Next, we define a partial color assignment for a set of vertices.

\begin{definition}[Partial color assignment]\label{def: partial assignment}
Let $U$ be a finite set. A \emph{partial color assignment for $U$} is a function $\psi\,:\,S \to \N$, defined on some subset $S \subseteq U$. We refer to $S$ as the \emph{support} of $\psi$, denoted $\supp(\psi)$. If $\psi_1, \dots, \psi_m$ are partial color assignments for $U$ with pairwise disjoint supports, we write $\bigcup_{j=1}^m \psi_j$ for the partial color assignment for $U$ with support $\bigcup_{j=1}^m \supp(\psi_j)$ that agrees with $\psi_j$ on $\supp(\psi_j)$ for each $j \in [m]$.
\end{definition}

We now define the compatibility of a partial coloring and a family of restriction hypergraphs.
In particular, this formalizes the notion of a proper partial $L$-coloring.

\begin{definition}[Compatible coloring]\label{def: compatible}
Let $L$ be a list assignment on some finite set $U$, and let $\cF = \{F_c\}$ be a family of restriction hypergraphs with respect to $L$. A partial color assignment $\psi$ for $U$ is \emph{$(L, \cF)$-compatible} if:
\begin{enumerate}[label=\textup{(\roman*)}]
\item \label{cond: compatible 1} $\psi(v) \in L(v)$ for every $v \in \supp(\psi)$; and
\item \label{cond: compatible 2} for every $c \in L(U)$, there is no edge $e \in E(F_c)$ with $e \subseteq \supp(\psi)$ and $\psi(v) = c$ for every $v \in e$.
\end{enumerate}
\end{definition}

As mentioned in the proof overview Section~\ref{section: proof overview}, we will repeatedly apply our \textit{nibble procedure} (see Algorithm~\ref{algorithm: nibble step}) to construct a partial coloring, while maintaining that the hypergraph and the lists of available colors satisfy certain properties at each step.
We formally state these properties now.

\begin{definition}[Property $P(i)$]\label{def: P(i)}
Let $L$ be a list assignment on some finite set $U$, and let $\cF = \{F_c\}$ be a family of restriction hypergraphs with respect to $L$. We say that $(L, \cF)$ satisfies $P(i)$ for some $i \in [s]$ if:
\begin{enumerate}[label=\textup{(Q\arabic*)}]
\item\label{Q1} the union $\bigcup_{c \in L(U)} F_c$ is rank-$k$ and uncrowded; moreover, for every $c, c' \in L(U)$ with $c \neq c'$, every $e \in E(F_c)$, and every $e' \in E(F_{c'})$, either $e = e'$ with $|e| = |e'| = k$, or $|e \cap e'| \le 1$;
\item\label{Q2} for every $v \in U$, $|L(v)| = q_i$; and
\item\label{Q3} for every $v \in U$ and $c \in L(v)$, $\deg(F_c, \ell, v) \le \binom{k-1}{\ell-1} (\alpha\gamma_i/q_i)^{k-\ell} \beta_i^{\ell-1} t_i$ for all $2 \le \ell \le k$.
\end{enumerate}
\end{definition}

Before we proceed, we provide a brief justification for the rather involved condition~\ref{Q1}.
The hypergraph $F_c$ contains two kinds of edges:
\begin{enumerate}
    \item $k$-edges $e$ such that every vertex in $e$ is not yet colored and contains $c$ in its list, and 
    \item $\ell$-edges $f$ for $\ell < k$ such that $f \subseteq e$ for some $k$-edge $e$ in the original hypergraph satisfying (i) every vertex in $e \setminus f$ is colored $c$, and (ii) every vertex in $f$ is not yet colored and contains $c$ in its list.
\end{enumerate}
With the above intuition in hand, by linearity, for every $c, c' \in L(U)$ with $c \neq c'$, every $e \in E(F_c)$, and every $e' \in E(F_{c'})$, either $e = e'$ with $|e| = |e'| = k$, or $|e \cap e'| \le 1$.

For each $v \in V(\cH)$, define $L_1(v)$ by removing $\max\{|\cL(v)| - q_1,\, 0\}$ colors from $\cL(v)$ arbitrarily. For each $c \in L_1(V(\cH))$, let $F_{1,c}$ be the hypergraph with vertex set $V(\cH)$ and edge set 
\[
\left\{e \in E(\cH)\,:\,c \in \textstyle\bigcap_{v \in e} L_1(v)\right\}.
\]
Let $\cF_1$ denote the family of all hypergraphs $F_{1,c}$. Note that $\cF_1$ is a family of restriction hypergraphs with respect to $L_1$.

\begin{proposition}\label{prop: P(1)}
$(L_1, \cF_1)$ satisfies $P(1)$.
\end{proposition}

\begin{proof}
We verify each condition of Definition~\ref{def: P(i)} in turn.

Since $E(F_{1,c}) \subseteq E(\cH)$ for every $c \in L_1(V(\cH))$ by construction, we have $\bigcup_{c \in L_1(V(\cH))} E(F_{1,c}) \subseteq E(\cH)$. Since $\cH$ is rank-$k$ and uncrowded, and deleting edges cannot increase the size of any edge nor create a new cycle, the union $\bigcup_{c \in L_1(V(\cH))} F_{1,c}$ is rank-$k$ and uncrowded. The latter part of~\ref{Q1} follows directly from the fact that $\cH$ is $k$-uniform and linear.

For~\ref{Q2}, by construction, we have $|L_1(v)| = \min\{|\cL(v)|, q_1\} = q_1$ for each $v \in V(\cH)$.

For~\ref{Q3}, fix $v \in V(\cH)$, $c \in L_1(v)$, and $2 \le \ell \le k$. Since $\cH$ is $k$-uniform, $\deg(F_{1,c}, \ell, v) = 0$ unless $\ell = k$, in which case
\[
\deg(F_{1,c}, k, v) = \left|\left\{e \in E(\cH, v)\,:\,c \in \textstyle\bigcap_{u \in e} L_1(u)\right\}\right| \le \left|\left\{e \in E(\cH, v)\,:\,c \in \textstyle\bigcap_{u \in e} \cL(u)\right\}\right| \le d,
\]
since $L_1(v) \subseteq \cL(v)$ and by condition~\ref{item: color degree in main thm} of Theorem~\ref{theorem: main}.
\end{proof}

The key result of this section is the following lemma, which states that given a list assignment $L$ on some finite set $U$, and a family $\cF = \{F_c\}$ of restriction hypergraphs with respect to $L$ such that $(L, \cF)$ satisfies $P(i)$, we can construct an $(L, \cF)$-compatible partial color assignment $\psi$ for $U$, a list assignment $\tilde{L}$ on $\tilde{U} \coloneqq U \setminus \supp(\psi)$, and a family $\tilde{\cF}$ of restriction hypergraphs with respect to $\tilde{L}$ such that $(\tilde{L}, \tilde{\cF})$ satisfies $P(i+1)$ and for every $(\tilde{L}, \tilde{\cF})$-compatible partial color assignment $\tilde{\psi}$ for $\tilde{U}$, $\psi \cup \tilde{\psi}$ is $(L,\cF)$-compatible.
In particular, we will repeatedly apply this lemma until a point where we can complete our coloring via a Local Lemma argument to prove our main result.

\begin{lemma}\label{lemma: nibble}
Let $L$ be a list assignment on some finite set $U$, and let $\cF$ be a family of restriction hypergraphs with respect to $L$. Suppose that $(L, \cF)$ satisfies $P(i)$ for some $i \in [s-1]$. There exist a partial color assignment $\psi$ for $U$, a list assignment $\tilde{L}$ on $\tilde{U} \coloneqq U \setminus \supp(\psi)$, and a family $\tilde{\cF}$ of restriction hypergraphs with respect to $\tilde{L}$ such that:
\begin{enumerate}[label=\textup{(R\arabic*)}]
\item\label{cond: psi compatible} $\psi$ is $(L, \cF)$-compatible,
\item\label{cond: P(i+1)} $(\tilde{L}, \tilde{\cF})$ satisfies $P(i+1)$, and
\item\label{cond: reduction} for every $(\tilde{L}, \tilde{\cF})$-compatible partial color assignment $\tilde{\psi}$ for $\tilde{U}$, $\psi \cup \tilde{\psi}$ is $(L,\cF)$-compatible.
\end{enumerate}
\end{lemma}

Our goal is to apply the above lemma $s-1$ times at which point we may complete our coloring via the following result.

\begin{lemma}[Finishing blow]\label{lemma: finishing}
Suppose $(L_s, \cF_s)$ satisfies Property $P(s)$. Then there exists an $(L_s,\cF_s)$-compatible color assignment $\psi_s$ with $\supp(\psi_s) = \supp(L_s)$.
\end{lemma}

\begin{proof}
Let $U_s \coloneqq \supp(L_s)$ and $\cF_s = \{F_c\,:\,c \in L_s(U_s)\}$. Color each vertex $v \in U_s$ independently with a color from $L_s(v)$ chosen uniformly at random; write $\psi_s$ for the resulting color assignment on $U_s$.
For each color $c \in L_s(U_s)$ and each edge $e \in E(F_c)$, let $\cA_{e,c}$ denote the event that $\psi_s(v) = c$ for every $v \in e$,
and set $\cA \coloneqq
\{\cA_{e,c} : c \in L_s(U_s), \, e \in E(F_c)\}$. We use the asymmetric version of the \hyperref[theorem: LLL]{Lov\'asz Local Lemma} (Theorem~\ref{theorem: LLL}) to show that with positive probability none of the events in $\cA$ occur.

It remains to verify the conditions of Theorem~\ref{theorem: LLL}. 
By~\ref{Q2} of property $P(s)$, $|L_s(v)| = q_s$ for every $v \in U_s$, and so
by Lemma~\ref{lemma: properties of beta_i, gamma_i, keep_i, uncolor_i}~\ref{q_i bound}, we have
\begin{align}\label{eq: probability A_{e,c}}
\pr(\cA_{e,c}) = \prod_{v \in e} \frac{1}{|L_s(v)|} \overset{\ref{Q2}}{=} \left(\frac{1}{q_s}\right)^{|e|}
\le \frac{1}{q_s^2}
\overset{\ref{q_i bound}}{\ll} \frac{1}{d^{\frac{2\eta}{k-1}}} \ll \frac{1}{4},
\end{align}
given that $d$ is sufficiently large. Let us next consider the dependencies between the events in $\cA$. Observe that $\cA_{e,c}$ is mutually independent of all events in $\cA \setminus (\cD_{e,c} \cup \{\cA_{e,c}\})$, where
\[
\cD_{e,c} \coloneqq \{\cA_{e',c'}\,:\,c' \in L_s(U_s), \, e' \in E(F_{c'}), \, e' \cap e \neq \varnothing\} \setminus \{\cA_{e, c}\}.
\]
We compute
\begin{align}
\sum_{\cA_{e',c'} \in \cD_{e,c}} \pr(\cA_{e',c'})
&\,\;\le\;\, \sum_{v \in e} \sum_{c' \in L_s(v)} \sum_{e' \in E(F_{c'}, v)} \pr(\cA_{e',c'}) \\
&\,\;=\;\, \sum_{v \in e} \sum_{c' \in L_s(v)} \sum_{e' \in E(F_{c'}, v)} \prod_{u \in e'}\frac{1}{|L_s(u)|} \\
&\overset{\ref{Q2}}{=} \sum_{v \in e} \sum_{c' \in L_s(v)} \sum_{\ell=2}^k \sum_{e' \in E(F_{c'}, \ell, v)} \left(\frac{1}{q_s}\right)^\ell \\
&\overset{\ref{Q2}}{\le} |e| \cdot \max_{v, c'} {\left( q_s \cdot \sum_{\ell=2}^k \frac{\deg(F_{c'}, \ell, v)}{q_s^{\ell}} \right)} \\
&\,\;\le\;\, k \cdot \max_{v, c'} {\left(\sum_{\ell=2}^k \frac{\deg(F_{c'}, \ell, v)}{q_s^{\ell-1}}\right)} \\
&\overset{\ref{Q3}}{\le} k \cdot \sum_{\ell=2}^k \frac{1}{q_s^{\ell-1}} \binom{k-1}{\ell-1} (\alpha\gamma_s/q_s)^{k-\ell} \beta_s^{\ell-1} t_s \\
&\,\;=\;\, \frac{k t_s}{q_s^{k-1}} \left[(\alpha\gamma_s + \beta_s)^{k-1} - (\alpha\gamma_s)^{k-1}\right].
\end{align}
Since $x^{k-1}$ is convex for $k \ge 2$ and $x \ge 0$, we obtain 
\begin{align}\label{eq: dependency bound 1}
\sum_{\cA_{e',c'} \in \cD_{e,c}} \pr(\cA_{e',c'}) \le \frac{k t_s}{q_s^{k-1}} (k-1) (\alpha \gamma_s + \beta_s)^{k-2} \beta_s.
\end{align}
By Lemma~\ref{lemma: properties of beta_i, gamma_i, keep_i, uncolor_i}~\ref{beta_i bound} and the definition of $s$ given in~\eqref{def: eta, alpha, delta, s}, we have
\begin{align}
\beta_s \overset{\ref{beta_i bound}}{\le} (1-\alpha(1-\eta))^{s-1} 
&\le \exp\left(-\alpha(1-\eta)(s-1)\right) \\
&\overset{\eqref{def: eta, alpha, delta, s}}{\le} \exp\left(-(1+\eta)\log\log d+\frac{\eta(1-\eta)}{\log d}\right) \le (\log d)^{-(1+\eta/2)},\label{beta_s bound}
\end{align}
given that $\log\log d > 2(1-\eta)/\log d$.
On the other hand, since
\begin{align}\label{eq: alpha gamma_s}
\alpha\gamma_s \overset{\eqref{def: uncolor_i, beta_i, gamma_i}}{=} \alpha \sum_{j=1}^{s-1} \beta_j \overset{\ref{beta_i bound}}{\le} \alpha \sum_{j=1}^{s-1} (1-\alpha(1-\eta))^{j-1}
= \frac{1- (1-\alpha(1-\eta))^{s-1}}{1-\eta},
\end{align}
we have
\begin{align}
\alpha\gamma_s + \beta_s \overset{\ref{beta_i bound}}{\le} \frac{1- (1-\alpha(1-\eta))^{s-1}}{1-\eta} + (1-\alpha(1-\eta))^{s-1} 
\le \frac{1}{1-\eta}.\label{eq: alpha gamma_s + beta_s}
\end{align}
Combining~\eqref{eq: dependency bound 1},~\eqref{beta_s bound}, and~\eqref{eq: alpha gamma_s + beta_s} with~\ref{ratio bound}, we get
\begin{align}
\sum_{\cA_{e',c'} \in \cD_{e,c}} \pr(\cA_{e',c'}) 
&\le \frac{k \log d}{(k-1)(1+5\eta)^{k-1}} \cdot (k-1) \cdot \frac{(\log d)^{-(1+\eta/2)}}{(1-\eta)^{k-2}} \\
&\le \frac{k}{(1+3\eta)^{k-1}} \cdot (\log d)^{-\eta/2} < \frac{1}{4}.
\end{align}
As a result, by applying the asymmetric version of the \hyperref[theorem: LLL]{Lov\'asz Local Lemma}, we conclude that, with positive probability, $\psi_s$ is $(L_s, \cF_s)$-compatible, thereby completing the proof.
\end{proof}

Assuming Lemma~\ref{lemma: properties of beta_i, gamma_i, keep_i, uncolor_i}, Proposition~\ref{prop: P(1)}, and Lemmas~\ref{lemma: nibble} and~\ref{lemma: finishing} we now prove Theorem~\ref{theorem: main}.

\begin{proof}[Proof of Theorem~\ref{theorem: main}]

Starting from $(L_1, \cF_1)$, which satisfies $P(1)$ by Proposition~\ref{prop: P(1)},
we apply Lemma~\ref{lemma: nibble} iteratively to construct sequences $(U_i)_{i \in [s]}$, $(\psi_i)_{i \in [s-1]}$, $(L_i)_{i \in [s]}$, and $(\cF_i)_{i \in [s]}$, with $U_1 \coloneqq V(\cH)$, as follows. For $i = 1, \dots, s-1$, given that $(L_i, \cF_i)$ satisfies $P(i)$ with $\supp(L_i) = U_i$, Lemma~\ref{lemma: nibble} yields a partial color assignment $\psi_i$ for $U_i$, a list assignment $L_{i+1}$ on $U_{i+1} \coloneqq U_i \setminus \supp(\psi_i)$, and a family $\cF_{i+1}$ of restriction hypergraphs with respect to $L_{i+1}$, such that:
\begin{enumerate}[label=\textup{(T\arabic*)}]
\item\label{S1} $\psi_i$ is $(L_i, \cF_i)$-compatible,
\item\label{S2} $(L_{i+1}, \cF_{i+1})$ satisfies $P(i+1)$, and
\item\label{S3} for every $(L_{i+1}, \cF_{i+1})$-compatible partial color assignment $\psi$ for $U_{i+1}$, $\psi_i \cup \psi$ is $(L_i, \cF_i)$-compatible.
\end{enumerate}
This produces $(L_s, \cF_s)$ satisfying $P(s)$, with $U_s = V(\cH) \setminus \bigcup_{i=1}^{s-1}\supp(\psi_i)$.

We first claim that, for every $1 \le i \le s$ and every $(L_i, \cF_i)$-compatible partial color assignment $\psi$ for $U_i$, the color assignment $\bigcup_{j<i}\psi_j \cup \psi$ is $(L_1, \cF_1)$-compatible. We prove this by induction on $i$. The base case $i=1$ is immediate. For the inductive step, suppose the claim holds for some $i < s$, and let $\psi$ be $(L_{i+1}, \cF_{i+1})$-compatible for $U_{i+1}$. By~\ref{S3}, $\psi_i \cup \psi$ is $(L_i, \cF_i)$-compatible for $U_i$, so by the inductive hypothesis applied to $\psi_i \cup \psi$, the assignment $\bigcup_{j<i}\psi_j \cup (\psi_i \cup \psi) = \bigcup_{j<i+1}\psi_j \cup \psi$ is $(L_1, \cF_1)$-compatible, completing the induction.

In particular, by
the claim with $i = s$, setting $\psi_s$ to be the $(L_s,\cF_s)$-compatible color assignment for $U_s$ given by Lemma~\ref{lemma: finishing}, the assignment $\Psi \coloneqq \bigcup_{j<s}\psi_j \cup \psi_s$ is $(L_1, \cF_1)$-compatible.

It remains to observe that $\Psi$ is an $\cL$-coloring of $\cH$. We note that $\supp(\Psi) = \bigcup_{j<s}(U_j \setminus U_{j+1}) \cup U_s = U_1 = V(\cH)$.
By condition~\ref{cond: compatible 1} of Definition~\ref{def: compatible}, $\Psi(v) \in L_1(v)$ for every $v \in V(\cH)$, and since $L_1(v) \subseteq \cL(v)$ by construction, $\Psi(v) \in \cL(v)$ for every $v \in V(\cH)$. For properness, suppose toward a contradiction that some edge $e \in E(\cH)$ is $c$-monochromatic under $\Psi$ for some color $c$, i.e., $\Psi(v) = c$ for every $v \in e$. By condition~\ref{cond: compatible 1}, $c = \Psi(v) \in L_1(v)$ for every $v \in e$, so $e \in E(F_{1,c})$ by the construction of $F_{1,c}$. This contradicts condition~\ref{cond: compatible 2} of $(L_1,\cF_1)$-compatibility of $\Psi$. Hence, $\Psi$ is a proper coloring of $\cH$ with $\Psi(v) \in \cL(v)$ for every $v \in V(\cH)$ and $\supp(\Psi) = V(\cH)$, i.e., $\Psi$ is an $\cL$-coloring of $\cH$, completing the proof of Theorem~\ref{theorem: main}.
\end{proof}

It remains to prove Lemmas~\ref{lemma: properties of beta_i, gamma_i, keep_i, uncolor_i} and~\ref{lemma: nibble}. We do so in the next two subsections.

\subsection{The recursion: proof of Lemma~\ref{lemma: properties of beta_i, gamma_i, keep_i, uncolor_i}}\label{subsection: properties}
For the reader's convenience, we restate the lemma here.
\begin{lemma*}[Restatement of Lemma~\ref{lemma: properties of beta_i, gamma_i, keep_i, uncolor_i}]
For $i \in [s-1]$, the following hold:
\begin{enumerate}[label=\textup{(S\arabic*)}]
\item $\gamma_1=0$ and $1 \le \gamma_{i+1} \le \dfrac{\log d}{\eta(1-\eta)}$;
\item $1- \eta \le \keep_i = 1-\Omega{\left((\log d)^{-k}\right)}$; 
\item $1 - \alpha \le \uncolor_i \le 1 - \alpha(1-\eta)$.
\end{enumerate}
Moreover, for each $i \in [s]$, the following hold:
\begin{enumerate}[label=\textup{(S\arabic*)},resume]
\item $(\log d)^{-2} \le \beta_i \le (1-\alpha(1-\eta))^{i-1}$;
\item $\dfrac{\log d}{(1+\eps)^{k-1}(k-1)} \le \dfrac{t_i}{q_i^{k-1}} \le \dfrac{\log d}{(1+5\eta)^{k-1}(k-1)}$;
\item $q_i \gg d^{\frac{\eta}{k-1}}$.
\end{enumerate}
\end{lemma*}

Recall that $d \ge d_0$, where $d_0$ is sufficiently large in terms of $k$ and $\eps$. By the definitions of $\delta$ and $s$ given in~\eqref{def: eta, alpha, delta, s}, we note that 
\begin{align}\label{eq: (1-delta) to 2(s-1)}
(1-\delta)^{2(s-1)} \ge e^{-4(s-1)\delta} \ge \exp\left(-\frac{4(1+\eta) \log d \log \log d}{\eta(1-\eta)} \cdot d^{-\frac{\eta/6}{k-1}}\right) \ge 1-\eta/2 = \Omega(1),
\end{align}
and so for all $i \in [s-1]$,
\begin{align}\label{eq: (1+delta)/(1-delta)}
\left(\frac{1+\delta}{1-\delta}\right)^{(k-1)i} \le \frac{1}{(1-\delta)^{2(k-1)i}} \le \frac{1}{(1-\delta)^{2(k-1)(s-1)}} \overset{\eqref{eq: (1-delta) to 2(s-1)}}{\le} \frac{1}{(1-\eta/2)^{k-1}}.
\end{align}

Now let us verify conditions~\ref{gamma_i bound}--\ref{ratio bound} for $i=1$ as a base case. Since $\beta_1 = 1 = (1-\alpha(1-\eta))^0 \ge (\log d)^{-2}$, condition~\ref{beta_i bound} holds for $i = 1$.
This gives 
\[
\gamma_2 = \beta_1 = 1 \le \frac{\log d}{\eta(1-\eta)},
\]
for $d \ge d_0$ sufficiently large, giving condition~\ref{gamma_i bound} for $i=1$. 
Next, since $5\eta \le \eps$ by~\eqref{def: eta, alpha, delta, s}, we have 
\begin{align}\label{eq: ratio_1}
\frac{t_1}{q_1^{k-1}} \overset{\eqref{def: t_1, q_1}}{=} \frac{\log d}{(1+\eps)^{k-1}(k-1)} \le \dfrac{\log d}{(1+5\eta)^{k-1}(k-1)},
\end{align}
giving condition~\ref{ratio bound} for $i=1$. By the definitions of $\keep_i$ and $\alpha$ given in~\eqref{def: keep_i} and~\eqref{def: eta, alpha, delta, s}, respectively, together with $\gamma_1 = 0$ and $\gamma_2 = 1$, it follows that
\[
1 - \keep_1 \overset{\eqref{def: keep_i}}{=} \frac{\alpha^{k-1}t_1}{q_1^{k-1}} \, {\left(\gamma_2^{k-1} - \gamma_1^{k-1}\right)} = \frac{\alpha^{k-1}t_1}{q_1^{k-1}} \overset{\eqref{eq: ratio_1}}{=} \frac{(\eta/(1+\eps))^{k-1}}{(k-1)(\log d)^{k-2}},
\]
which implies
\[
\eta
\ge \frac{(\eta/(1+\eps))^{k-1}}{(k-1)(\log d)^{k-2}} = 1-\keep_1 = \Omega{\left((\log d)^{-k}\right)},
\]
giving condition~\ref{keep_i bound} for $i = 1$.
Since $\uncolor_1 = 1-\alpha\keep_1$ by~\eqref{def: uncolor_i, beta_i, gamma_i}, it follows that 
\[
1 - \alpha \le \uncolor_1 \le 1 - \alpha(1-\eta),
\]
giving~\ref{uncolor_i bound} for $i = 1$.

Now suppose for some $j < s - 1$ that conditions~\ref{gamma_i bound}--\ref{ratio bound} hold for all $i \in [j]$. 
We verify the conditions for $j + 1$. For~\ref{beta_i bound} at $i = j+1$, since $\beta_{j+1} = \prod_{i=1}^{j}\uncolor_i$ by~\eqref{def: uncolor_i, beta_i, gamma_i}, applying~\ref{uncolor_i bound} for each $i \in [j]$ yields
\begin{align}
(1-\alpha(1-\eta))^j 
\ge \beta_{j+1} 
&\ge (1-\alpha)^j 
\ge (1-\alpha)^{s-1} \\
&\ge \exp\left(-\frac{(s-1)\alpha}{1-\alpha}\right) \overset{\eqref{def: eta, alpha, delta, s}}{\ge}
\exp\left(-\frac{(1+\eta) \log\log d}{(1-\eta)^2}\right) \ge (\log d)^{-2}.\label{eq: beta_{j+1}}
\end{align}
For~\ref{gamma_i bound} at $i = j+1$, we use~\ref{beta_i bound} and the above for $i \in [j]$ to get
\begin{align}
1 = \gamma_2 \le \gamma_{j+2} 
\overset{\eqref{def: uncolor_i, beta_i, gamma_i}}{=} \sum_{i=1}^{j+1}\beta_i
&\le \sum_{i=1}^{j+1}(1-\alpha(1-\eta))^{i-1}  \\
&\le \frac{1}{\alpha(1-\eta)} \overset{\eqref{def: eta, alpha, delta, s}}{=} \frac{\log d}{\eta(1-\eta)}.\label{eq: gamma_{j+2}}
\end{align}
Since $x^{k-1}$ is convex for $k \ge 2$ and $x \ge 0$, it follows that
\begin{align}
\gamma_{j+2}^{k-1} - \gamma_{j+1}^{k-1} 
&\le (k-1)\gamma_{j+2}^{k-2}(\gamma_{j+2}-\gamma_{j+1}) 
\overset{\eqref{def: uncolor_i, beta_i, gamma_i}}{=} (k-1) \, \gamma_{j+2}^{k-2} \, \beta_{j+1}  \\ 
&\overset{\ref{beta_i bound}}{\le} (k-1)\gamma_{j+2}^{k-2} (1-\alpha(1-\eta))^j 
\overset{\eqref{eq: gamma_{j+2}}}{\le} (k-1)\left(\frac{\log d}{\eta(1-\eta)}\right)^{k-2}, \label{eq: difference in keep_{j+1} ub}
\end{align}
and that
\begin{align}
\gamma_{j+2}^{k-1} - \gamma_{j+1}^{k-1} 
&\ge (k-1) \gamma_{j+1}^{k-2} (\gamma_{j+2} - \gamma_{j+1})  \\
&\overset{\eqref{def: uncolor_i, beta_i, gamma_i}}{=} (k-1) \gamma_{j+1}^{k-2} \beta_{j+1} \overset{\eqref{eq: gamma_{j+2}}}{\ge} (k-1) \beta_{j+1} \overset{\eqref{eq: beta_{j+1}}}{\ge} (k-1)(\log d)^{-2},
\label{eq: difference in keep_{j+1} lb}
\end{align}
where we used that $\gamma_{j+1} \ge \gamma_2 = 1$.
On the other hand, by the recursive definitions of $q_i$ and $t_i$ given in~\eqref{def: t_i, q_i}, we have
\begin{align}
\frac{t_{j+1}}{q_{j+1}^{k-1}} \overset{\eqref{def: t_i, q_i}}{=} \left(\frac{1+\delta}{1-\delta}\right)^{(k-1)j} \frac{t_1}{q_1^{k-1}} \overset{\eqref{eq: ratio_1}}{=}  \left(\frac{1+\delta}{1-\delta}\right)^{(k-1)j} \frac{\log d}{(1+\eps)^{k-1}(k-1)} \ge \frac{\log d}{(1+\eps)^{k-1}(k-1)}.
\end{align}
Together with~\eqref{eq: (1+delta)/(1-delta)}, it follows that
\begin{align}\label{eq: ratio_{i+1}}
\frac{\log d}{(1+\eps)^{k-1}(k-1)} \le \frac{t_{j+1}}{q_{j+1}^{k-1}}
\le \frac{\log d}{[(1-\eta/2)(1+\eps)]^{k-1}(k-1)} \le \frac{\log d}{(1+5\eta)^{k-1}(k-1)},
\end{align}
where we used that $(1+\eps)(1-\eta/2) = (1+6\eta)(1-\eta/2) = 1 + 5\eta + \eta/2 - 3\eta^2 > 1+5\eta$, completing the proof of~\ref{ratio bound}.
Since 
\[
1 - \keep_{j+1} \overset{\eqref{def: keep_i}}{=} \frac{\alpha^{k-1}t_{j+1}}{q_{j+1}^{k-1}}\,{\left(\gamma_{j+2}^{k-1} - \gamma_{j+1}^{k-1}\right)} = \frac{\eta^{k-1}}{(\log d)^{k-1}} \cdot \frac{t_{j+1}}{q_{j+1}^{k-1}} \cdot \left(\gamma_{j+2}^{k-1} - \gamma_{j+1}^{k-1}\right)
\]
it follows that 
\[
\frac{(\eta/(1+\eps))^{k-1}}{(k-1) (\log d)^{k-2}} \,{\left(\gamma_{j+2}^{k-1} - \gamma_{j+1}^{k-1}\right)} \le 1 - \keep_{j+1} \le \frac{(\eta/(1+5\eta))^{k-1}}{(k-1) (\log d)^{k-2}} \, {\left(\gamma_{j+2}^{k-1} - \gamma_{j+1}^{k-1}\right)}.
\]
Combining this with ~\eqref{eq: difference in keep_{j+1} ub} and~\eqref{eq: difference in keep_{j+1} lb}, we get
\[
\Omega{\left((\log d)^{-k}\right)} = \frac{(\eta/(1+\eps))^{k-1}}{(\log d)^k} \overset{\eqref{eq: difference in keep_{j+1} lb}}{\le} 1 - \keep_{j+1}
\overset{\eqref{eq: difference in keep_{j+1} ub}}{\le} \frac{\eta}{(1+5\eta)^{k-1}(1-\eta)^{k-2}} \le \eta,
\]
where in the last inequality we used $(1+5\eta)(1-\eta) \ge 1$, noting that $\eta = \eps/6 < 1/6$. This completes the verification of~\ref{keep_i bound} for $i = j+1$.
Since $\uncolor_{j+1} = 1 - \alpha\keep_{j+1}$ by~\eqref{def: uncolor_i, beta_i, gamma_i}, it follows that
\[
1 - \alpha \le \uncolor_{j+1} \le 1 - \alpha(1-\eta),
\]
giving~\ref{uncolor_i bound} for $i = j+1$, completing the induction. By induction, conditions~\ref{gamma_i bound}--\ref{ratio bound} hold for all $i \in [s - 1]$.
Applying the calculation in~\eqref{eq: beta_{j+1}} with $j = s - 1$, using~\ref{uncolor_i bound} for every $i \in [s - 1]$, also proves~\ref{beta_i bound} for $i = s$. On the other hand, the same calculation as in~\eqref{eq: ratio_{i+1}}, with $j = s - 1$, proves~\ref{ratio bound} for $i = s$.

It remains to verify~\ref{q_i bound}.
By using the fact that $\log(1-x) \ge \frac{-x}{1-x}$ whenever $|x| < 1$, we obtain that for all $i \in [s-1]$,
\begin{align}
\label{eq: log keep_i}
-\log \keep_i \le \frac{1-\keep_i}{\keep_i} \overset{\ref{keep_i bound}}{\le} \frac{1-\keep_i}{1-\eta}.
\end{align}
As a result,
\begin{align}
-\sum_{j=1}^{s-1} \log\keep_j 
&\ \le \ \frac{(\eta/(1+5\eta))^{k-1}}{(1-\eta)(k-1)(\log d)^{k-2}} \, {\left(\gamma_s^{k-1} - \gamma_1^{k-1}\right)} \\
&\overset{\ref{gamma_i bound}}{\le} \frac{(\eta/(1+5\eta))^{k-1}}{(1-\eta)(k-1)(\log d)^{k-2}} \, {\left[\frac{\log d}{\eta(1-\eta)}\right]^{k-1}} \\
&\ \le \ \frac{\log d}{(k-1)[(1+5\eta)(1-\eta)]^{k-1}(1-\eta)} \\
&\ \le \ \frac{\log d}{(k-1)(1+3\eta/2)},
\end{align}
where in the last inequality we used
\[
[(1+5\eta)(1-\eta)]^{k-1}(1-\eta) \ge [(1+5\eta)(1-\eta)^2]^{k-1} \ge 1 + 3\eta/2,
\]
given that $\eta < 1/6$. Thus,
\begin{align}\label{eq: product log keep_i}
\prod_{j=1}^{s-1} \keep_j \ge d^{-\frac{1}{(k-1)(1+3\eta/2)}}.
\end{align}
Finally, by the recursive definition of $q_i$ given in~\eqref{def: t_1, q_1} and~\eqref{def: t_i, q_i}, we have
\begin{align}
q_i 
&\overset{\eqref{def: t_i, q_i}}{=} (1-\delta)^{i-1}  \cdot q_1 \cdot \prod_{j=1}^{i-1} \keep_j \\
&\overset{\eqref{def: t_1, q_1}}{\ge} (1-\delta)^{s-1}  \cdot \left(\frac{d}{\log d}\right)^{\frac{1}{k-1}} \cdot \prod_{j=1}^{s-1} \keep_j \\
&\overset{\eqref{eq: product log keep_i}}{\ge} (1-\delta)^{s-1}  \left(\frac{d}{\log d}\right)^{\frac{1}{k-1}} d^{-\frac{1}{(k-1)(1+3\eta/2)}} \overset{\eqref{eq: (1-delta) to 2(s-1)}}{\gg} d^{\frac{\eta}{k-1}},
\end{align}
where in the last inequality we used that $(1-\delta)^{s-1} = \Omega(1)$ from~\eqref{eq: (1-delta) to 2(s-1)} and that 
\[
1 - \frac{1}{1+3\eta/2} = \frac{3\eta}{2+3\eta} > \eta.
\]
This concludes the proof of Lemma~\ref{lemma: properties of beta_i, gamma_i, keep_i, uncolor_i}.

\subsection{A single nibble step: proof of Lemma~\ref{lemma: nibble}}\label{subsection: nibble}

In this section, we prove Lemma~\ref{lemma: nibble} modulo a key technical lemma, which we defer to Section~\ref{section: deg}. 
For the reader's convenience, we restate the lemma here.

\begin{lemma*}[Restatement of Lemma~\ref{lemma: nibble}]
Let $L$ be a list assignment on some finite set $U$, and let $\cF$ be a family of restriction hypergraphs with respect to $L$. Suppose that $(L, \cF)$ satisfies $P(i)$ for some $i \in [s-1]$. There exist a partial color assignment $\psi$ for $U$, a list assignment $\tilde{L}$ on $\tilde{U} \coloneqq U \setminus \supp(\psi)$, and a family $\tilde{\cF}$ of restriction hypergraphs with respect to $\tilde{L}$ such that:
\begin{enumerate}[label=\textup{(R\arabic*)}]
\item $\psi$ is $(L, \cF)$-compatible,
\item $(\tilde{L}, \tilde{\cF})$ satisfies $P(i+1)$, and
\item for every $(\tilde{L}, \tilde{\cF})$-compatible partial color assignment $\tilde{\psi}$ for $\tilde{U}$, $\psi \cup \tilde{\psi}$ is $(L,\cF)$-compatible.
\end{enumerate}
\end{lemma*}

We construct the desired partial color assignment $\psi$ and family $\tilde\cF$ by employing a single \textit{nibble} step, outlined in Algorithm~\ref{algorithm: nibble step}.

\vspace{0.3cm}
\begin{breakablealgorithm}
\caption{A nibble step}\label{algorithm: nibble step}
\begin{flushleft}
\noindent\makebox[4.5em][l]{\textbf{Input}:}\parbox[t]{\dimexpr\linewidth-4.5em\relax}{A list assignment $L$ on a finite set $U$ and a family $\cF = \{F_c\}$ of restriction hypergraphs with respect to $L$ such that $(L, \cF)$ satisfies property $P(i)$.
}

\smallskip

\noindent\makebox[4.5em][l]{\textbf{Output}:}\parbox[t]{\dimexpr\linewidth-4.5em\relax}{A partial color assignment $\psi$ for $U$, a list assignment $\tilde{L}$ on $\tilde{U} \coloneqq U \setminus \supp(\psi)$, and a family $\tilde{\cF} = \{\tilde{F}_c\}$ of restriction hypergraphs with respect to $\tilde{L}$.
}
\end{flushleft}

\smallskip

\begin{enumerate}[itemsep=.2cm, label=(\arabic*)]
    \item\label{step: activate}
    Choose a random subset $A \subseteq U$ by including each vertex $v \in U$ independently with probability $\alpha$, where $\alpha$ is defined in~\eqref{def: eta, alpha, delta, s}. The vertices in $A$ are called \emphd{activated} vertices. 
    For every vertex $v \in U$, independently choose a color $c_v$ from $L(v)$ uniformly at random.
    \item\label{step: discard}
    For every vertex $v \in U$, a color $c \in L(v)$ is said to be \emphd{discarded} at $v$ if there exists an edge $e \in E(F_c, v)$ for which $u \in A$ and $c_u = c$ for all $u \in e-v$. When this occurs, we also say that the color $c$ is \emphd{discarded} at $v$ \emphd{via} the edge $e$ (note that $c$ can be discarded at $v$ via multiple edges). 
    We let $D(v)$ denote the set consisting of all discarded colors at $v$.
    \item\label{step: eq}
    Let $\{\eq(v,c)\,:\,v \in U,\, c \in L(v)\}$ be a family of independent Bernoulli random variables with
    \begin{align}\label{def: eq(v,c)}
        \pr(\eq(v,c) = 1) = \frac{\keep_i}{\pr(c \notin D(v))}.
    \end{align}
    Call $\eq(v,c)$ the \emphd{equalizing coin flip} for color $c$ at $v$. For every vertex $v \in U$, define
    \[
    L'(v) \coloneqq \{c \in L(v) \,:\, c \notin D(v), \ \eq(v,c) = 1\}.
    \]
    The colors in $L'(v)$ are called \emphd{kept} colors of $v$.
    
    \item\label{step: assignment}
    Define the partial color assignment $\psi$ for $U$ by setting $\psi(v) \coloneqq c_v$ for every $v \in A$ with $c_v \in L'(v)$. Set $\tilde{U} \coloneqq U \setminus \supp(\psi)$.
    
    \item\label{step: hat-hypergraph}
    For every color $c \in L'(\tilde{U})$, let $F'_c$ be the hypergraph with vertex set $\tilde{U}$ and, for each $2 \le \ell \le k$, 
    \begin{align}
        E(F'_c, \ell) \coloneqq \left\{e' \in \binom{\tilde{U}}{\ell} \,:\, e' \subseteq e \in \textstyle E(F_c), \, c \in \bigcap_{v \in e'} L'(v),~\text{and}~\psi(v) = c ~\text{for all}~v \in e \setminus e' \right\}.
    \end{align}
    \item\label{step: trim} 
    For every vertex $v \in \tilde{U}$, define $\tilde{L}(v)$ by removing $\max\{|L'(v)| - q_{i+1},\, 0\}$ colors from $L'(v)$ arbitrarily.
    
    \item\label{step: final restriction} For every color $c \in \tilde{L}(\tilde{U})$, let $\tilde{F}_c$ be the hypergraph with vertex set $\tilde{U}$ and edge set 
    \[
    E(\tilde{F}_c) \coloneqq \left\{e' \in E(F'_c)\,:\,c \in \textstyle \bigcap_{v \in e'} \tilde{L}(v)\right\}.
    \]
\end{enumerate}
\end{breakablealgorithm}
\vspace{0.3cm}

The goal of this section is to show that the output of Algorithm~\ref{algorithm: nibble step} produces a partial color assignment $\psi(\cdot)$ satisfying the conditions of Lemma~\ref{lemma: nibble} with positive probability.
We also record the following containment that will be used repeatedly. For every $v \in \tilde{U}$, Steps~\ref{step: eq} and~\ref{step: trim} give
\begin{align}\label{eq: lists compare}
\tilde{L}(v) \subseteq L'(v) \subseteq L(v).
\end{align}
We begin with the following lemma, which gives an upper bound on the probability that a color is discarded. 
Notably, it implies that Step~\ref{step: eq} of Algorithm~\ref{algorithm: nibble step} is well-defined, that is, $\pr(\eq(v, c) = 1) \le 1$.

\begin{lemma}\label{lemma: eq well defined}
For every vertex $v \in U$ and color $c \in L(v)$, we have 
\[
\pr(c \in D(v)) \le \sum_{\ell=2}^k \deg(F_c, \ell, v) (\alpha/q_i)^{\ell-1} \le 1 - \keep_i.
\]
\end{lemma}
 
\begin{proof}
Fix a vertex $v \in U$ and a color $c \in L(v)$. Recall that $c$ is discarded at $v$ in Algorithm~\ref{algorithm: nibble step} if there is an edge $e \in E(F_c, v)$ for which $e-v \subseteq A$ and $c_x = c$ for $x \in e-v$. 
Therefore, by condition~\ref{Q2} of property $P(i)$ and a union bound, we have
\begin{align*}
\pr(c \in D(v)) 
&\,\,\le \sum_{\ell=2}^k \deg(F_c, \ell, v) (\alpha/q_i)^{\ell-1} \\
&\overset{\ref{Q3}}{\le} \frac{\alpha^{k-1} t_i}{q_i^{k-1}} \sum_{\ell=2}^k \binom{k-1}{\ell-1} \gamma_i^{k-\ell} \beta_i^{\ell-1} \\
&\,\,\overset{\eqref{def: uncolor_i, beta_i, gamma_i}}{=} \frac{\alpha^{k-1} t_i}{q_i^{k-1}} \, {\left(\gamma_{i+1}^{k-1} - \gamma_i^{k-1}\right)} \overset{\eqref{def: keep_i}}{=} 1-\keep_i,
\end{align*}
as desired.
\end{proof}

The next lemma shows that certain conditions outlined in Lemma~\ref{lemma: nibble} hold deterministically.

\begin{lemma}\label{lemma: deterministic conditions}
    Let $(\psi, \tilde L, \tilde\cF)$ be the output of Algorithm~\ref{algorithm: nibble step} with input $(L, \cF)$. We have
    \begin{enumerate}
        \item $\psi$ is $(L, \cF)$-compatible,
        \item for every $(\tilde{L}, \tilde{\cF})$-compatible partial color assignment $\tilde{\psi}$ for $\tilde{U}$, $\psi \cup \tilde{\psi}$ is $(L,\cF)$-compatible,
        \item the union $\bigcup_{c \in \tilde{L}(\tilde{U})} \tilde{F}_c$ is rank-$k$ and uncrowded, and
        \item for every $c, c' \in \tilde{L}(\tilde{U})$ with $c \neq c'$, every $f \in E(\tilde{F}_c)$, and every $f' \in E(\tilde{F}_{c'})$, either $f = f'$ with $|f| = |f'| = k$, or $|f \cap f'| \le 1$.
    \end{enumerate}
\end{lemma}

\begin{proof}
\stepcounter{ForClaims}
\renewcommand{\theForClaims}{\ref*{lemma: deterministic conditions}}

We prove these conditions through a sequence of claims.
Let us first show that $\psi$ is $(L, \cF)$-compatible, implying that \ref{cond: psi compatible} holds.

\begin{claim}\label{claim: psi compatible}
$\psi$ is $(L, \cF)$-compatible.
\end{claim}

\begin{claimproof}
Condition~\ref{cond: compatible 1} of Definition~\ref{def: compatible} holds since $\psi(v) = c_v \in L(v)$ for every $v \in \supp(\psi)$, by Step~\ref{step: activate}.
For condition~\ref{cond: compatible 2}, suppose toward a contradiction that there exist $c \in L(U)$ and an edge $e \in E(F_c)$ with $e \subseteq \supp(\psi)$ and $\psi(v) = c$ for every $v \in e$. Fix any $u \in e$. By Step~\ref{step: assignment}, $c = \psi(u) = c_u \in L'(u)$, and so $c \notin D(u)$ by the definition of $L'(u)$ in Step~\ref{step: eq}. On the other hand, since $\psi(v) = c$ for every $v \in e$, Step~\ref{step: assignment} also gives $v \in A$ and $c_v = c$ for every $v \in e-u$. By Step~\ref{step: discard}, this implies that $c \in D(u)$, a contradiction.
\end{claimproof}

Next, we verify~\ref{cond: reduction}.

\begin{claim}\label{claim: reduction}
For every $(\tilde{L}, \tilde{\cF})$-compatible partial color assignment $\tilde{\psi}$ for $\tilde{U}$, $\psi \cup \tilde{\psi}$ is $(L,\cF)$-compatible.
\end{claim}

\begin{claimproof}
Write $\phi \coloneqq \psi \cup \tilde{\psi}$. We verify the two conditions in Definition~\ref{def: compatible} for $\phi$.
 
To verify condition~\ref{cond: compatible 1}, fix $v \in \supp(\phi)$. If $v \in \supp(\psi)$, then $\phi(v) = \psi(v) \in L(v)$ by Claim~\ref{claim: psi compatible}. If instead $v \in \supp(\tilde{\psi})$, then $\phi(v) = \tilde{\psi}(v) \in \tilde{L}(v) \subseteq L(v)$ by~\eqref{eq: lists compare}.
 
To verify condition~\ref{cond: compatible 2}, suppose toward a contradiction that there exist $c \in L(U)$ and an edge $e \in E(F_c)$ such that $e \subseteq \supp(\phi)$ and $\phi(u) = c$ for every $u \in e$. Write $e_1 \coloneqq e \cap \supp(\psi)$ and $e_2 \coloneqq e \cap \supp(\tilde{\psi})$, so that $e = e_1 \cup e_2$ and $e_1 \cap e_2 = \varnothing$. We derive a contradiction in each of the three cases $|e_2| = 0$, $|e_2| = 1$, and $|e_2| \ge 2$.
 
Suppose first that $|e_2| = 0$. Then $e \subseteq \supp(\psi)$ and $\psi(u) = c$ for every $u \in e$, contradicting Claim~\ref{claim: psi compatible}, specifically condition~\ref{cond: compatible 2} in Definition~\ref{def: compatible} required for $\psi$ to be $(L, \cF)$-compatible.
 
Suppose next that $|e_2| = 1$, and let $\{w\} = e_2$, so that $e_1 = e - w$. Since $\psi(u) = c$ for every $u \in e_1$, Step~\ref{step: assignment} of Algorithm~\ref{algorithm: nibble step} gives $u \in A$ and $c_u = c$ for every $u \in e - w$. Since $e$ is an edge of $F_c$ containing $w$, the definition of $D(w)$ in Step~\ref{step: discard} gives $c \in D(w)$. It follows from the definition of $L'(w)$ in Step~\ref{step: eq} that $c \notin L'(w)$, and hence $c \notin \tilde{L}(w)$ by~\eqref{eq: lists compare}. This contradicts $\tilde{\psi}(w) = c \in \tilde{L}(w)$.
 
Finally, suppose that $|e_2| \ge 2$. We claim that $e_2 \in E(F'_c)$. Indeed, $e_2 \subseteq \tilde{U}$ by construction, and $e_2 \subseteq e \in E(F_c)$. For every $v \in e_2$, we have $\tilde{\psi}(v) = c \in \tilde{L}(v) \subseteq L'(v)$ by~\eqref{eq: lists compare}, so $c \in \bigcap_{v \in e_2}L'(v)$. Moreover, for every $v \in e \setminus e_2 = e_1$, we have $\psi(v) = c$ by assumption. Hence, $e_2 \in E(F'_c)$ by the definition in Step~\ref{step: hat-hypergraph}. Since $c \in \tilde{L}(v)$ for every $v \in e_2$, Step~\ref{step: final restriction} gives $e_2 \in E(\tilde{F}_c)$. Combined with $\tilde{\psi}(v) = c$ for every $v \in e_2$, this contradicts the $(\tilde{L}, \tilde{\cF})$-compatibility of $\tilde{\psi}$.
\end{claimproof}

Next, we argue that $\bigcup_{c \in \tilde{L}(\tilde{U})} \tilde{F}_c$ is rank-$k$ and uncrowded.
 
\begin{claim}\label{claim: tilde F uncrowded}
The union $\bigcup_{c \in \tilde{L}(\tilde{U})} \tilde{F}_c$ is rank-$k$ and uncrowded.
\end{claim}

\begin{claimproof}
Since $\tilde{F}_c \subseteq F'_c$ by Step~\ref{step: final restriction} for every $c \in \tilde{L}(\tilde{U})$, and edge deletion neither increases the size of any edge nor creates a cycle, it suffices to show that $\bigcup_{c \in \tilde{L}(\tilde{U})} F'_c$ is rank-$k$ and uncrowded.

Every edge of $F'_c$ is a subset of some edge of $F_c$ by Step~\ref{step: hat-hypergraph}, and $F_c \subseteq \bigcup_{c'' \in L(U)}F_{c''}$ is rank-$k$ by condition~\ref{Q1} of property $P(i)$; hence $\bigcup_{c \in \tilde{L}(\tilde{U})} F'_c$ is rank-$k$.

It remains to show that $\bigcup_{c \in \tilde{L}(\tilde{U})} F'_c$ is uncrowded. We first record an observation about the edges of the hypergraphs in $\cF'$. Fix $c \in \tilde{L}(\tilde{U})$ and $g \in E(F'_c)$. By Step~\ref{step: hat-hypergraph}, there exists an edge $e \in E(F_c)$ with $g \subseteq e$ and $\psi(x) = c$ for every $x \in e \setminus g$, which requires that $(e \setminus g) \cap \tilde{U} = \varnothing$, and so $g = e \cap \tilde{U}$.
Moreover, this $e$ is unique. Indeed, if $e, e' \in E(F_c)$ both satisfy the above with respect to the same $g$, then $g \subseteq e \cap e'$, and since $|g| \ge 2$, we have $|e \cap e'| \ge 2$, so $e = e'$ by the linearity of $F_c$, which follows from condition~\ref{Q1} of property $P(i)$ (the hypergraph $\bigcup_{c'' \in L(U)}F_{c''}$ is uncrowded, in particular linear, and $F_c$ is a subhypergraph of it). We write $\pi_c(g) \coloneqq e$ for this unique edge.

Now suppose, for contradiction, that $\bigcup_{c \in \tilde{L}(\tilde{U})}F'_c$ contains a cycle
\[
(g_0, v_0, g_1, v_1, \dots, g_{r-1}, v_{r-1}, g_0)
\]
of length $r$ with $2 \le r \le 4$, where $g_0, \dots, g_{r-1}$ are pairwise distinct edges in the union and $v_0, \dots, v_{r-1}$ are pairwise distinct vertices in $\tilde{U}$ with $v_j \in g_j \cap g_{j+1}$ for each $j$ (indices modulo $r$).

For each $j$, choose a color $c_j \in \tilde{L}(\tilde{U})$ such that $g_j \in E(F'_{c_j})$, and let $e_j \coloneqq \pi_{c_j}(g_j) \in E(F_{c_j}) \subseteq E(\bigcup_{c'' \in L(U)}F_{c''})$. By our earlier observation, $g_j = e_j \cap \tilde{U}$ for each $j$. Hence, if $e_j = e_{j'}$ as subsets of $U$ for some $j \ne j'$, then $g_j = e_j \cap \tilde{U} = e_{j'} \cap \tilde{U} = g_{j'}$, contradicting the fact that the $g_j$'s are distinct. It follows that the $e_j$'s are also pairwise distinct.

Moreover, $v_j \in g_j \cap g_{j+1} \subseteq e_j \cap e_{j+1}$ for each $j$. Hence, $(e_0, v_0, e_1, v_1, \dots, e_{r-1}, v_{r-1}, e_0)$ is a cycle of length $r$ in $\bigcup_{c'' \in L(U)}F_{c''}$, contradicting the fact that $\bigcup_{c'' \in L(U)}F_{c''}$ is uncrowded by condition~\ref{Q1} of property $P(i)$.
\end{claimproof}

Finally, we verify the latter half of condition~\ref{Q1} from Property $P(i+1)$.

\begin{claim}\label{claim: tilde F cross color}
For every $c, c' \in \tilde{L}(\tilde{U})$ with $c \neq c'$, every $f \in E(\tilde{F}_c)$, and every $f' \in E(\tilde{F}_{c'})$, either $f = f'$ with $|f| = |f'| = k$, or $|f \cap f'| \le 1$.
\end{claim}

\begin{claimproof}
Since $\tilde{F}_c \subseteq F'_c$ and $\tilde{F}_{c'} \subseteq F'_{c'}$ by Step~\ref{step: final restriction}, it suffices to prove the same statement for $f \in E(F'_c)$ and $f' \in E(F'_{c'})$. By Step~\ref{step: hat-hypergraph}, there exist $e \in E(F_c)$ with $f \subseteq e$ and $\psi(x) = c$ for every $x \in e \setminus f$, and $e' \in E(F_{c'})$ with $f' \subseteq e'$ and $\psi(x) = c'$ for every $x \in e' \setminus f'$. We consider two cases.

If $e \ne e'$, then by condition~\ref{Q1} of property $P(i)$ applied to $e$ and $e'$, we have $|e \cap e'| \le 1$, and hence $|f \cap f'| \le |e \cap e'| \le 1$.

If $e = e'$, then by condition~\ref{Q1} of property $P(i)$, we have $|e| = k$. Since $\psi$ is a function and $c \ne c'$, no vertex of $e$ can be assigned both $c$ and $c'$ by $\psi$; that is, $(e \setminus f) \cap (e \setminus f') = \varnothing$. Since $f \subseteq \tilde{U}$ and $e \setminus f' \subseteq \supp(\psi) = U \setminus \tilde{U}$, we have $f \cap (e \setminus f') = \varnothing$; combined with $f \subseteq e$, this gives $f \subseteq f'$. Symmetrically, $f' \subseteq f$, so $f = f'$. Finally, if $f = f' \ne e$, then $e \setminus f$ is a nonempty subset of $e$ on which $\psi$ takes the value $c$ (by definition of $e$) and also the value $c'$ (by definition of $e'$, using $e = e'$), contradicting $c \ne c'$. Hence, $f = f' = e$, and $|f| = |f'| = k$.
\end{claimproof}

This covers all conditions, completing the proof of Lemma~\ref{lemma: deterministic conditions}.
\end{proof}

At this point, we note that conditions~\ref{cond: psi compatible} and~\ref{cond: reduction} as well as condition~\ref{Q1} of property $P(i + 1)$ hold deterministically.
In particular, in light of Lemma~\ref{lemma: deterministic conditions}, it suffices to show that $(\tilde{L}, \tilde{\cF})$ satisfies conditions~\ref{Q2} and~\ref{Q3} of property $P(i+1)$ with positive probability.
We will do so by applying the symmetric version of the \hyperref[theorem: sym LLL]{Lov\'asz Local Lemma} (Theorem~\ref{theorem: sym LLL}) to two families of bad events arising from the random choices in Steps~\ref{step: activate}--\ref{step: hat-hypergraph} of Algorithm~\ref{algorithm: nibble step}. We first record the deterministic passage from $(L', \cF')$ to $(\tilde{L}, \tilde{\cF})$.

\begin{obs}\label{obs: bridge}
Suppose that:
\begin{enumerate}[label=\textup{(Q\arabic*')}]
\setcounter{enumi}{1}
\item\label{Q2'} for every $v \in \tilde{U}$, $|L'(v)| \ge q_{i+1}$; and
\item\label{Q3'} for every $v \in \tilde{U}$, $c \in L'(v)$, and $2 \le \ell \le k$, $\deg(F'_c, \ell, v) \le \tbinom{k-1}{\ell-1}\left(\alpha\gamma_{i+1}/q_{i+1}\right)^{k-\ell}\beta_{i+1}^{\ell-1} t_{i+1}$.
\end{enumerate}
Then $(\tilde{L}, \tilde{\cF})$ satisfies conditions~\ref{Q2} and~\ref{Q3} of property $P(i+1)$.
\end{obs}

\begin{proof}
For~\ref{Q2}, fix $v \in \tilde{U}$. Since $|L'(v)| \ge q_{i+1}$, the trimming in Step~\ref{step: trim} removes exactly $|L'(v)| - q_{i+1}$ colors, so $|\tilde{L}(v)| = q_{i+1}$. For~\ref{Q3}, fix $v \in \tilde{U}$, $c \in \tilde{L}(v)$, and $2 \le \ell \le k$. By~\eqref{eq: lists compare}, $c \in L'(v)$, and since $\tilde{F}_c$ is obtained from $F'_c$ by deleting edges (Step~\ref{step: final restriction}), $\deg(\tilde{F}_c, \ell, v) \le \deg(F'_c, \ell, v)$, which is bounded as required by the hypothesis.
\end{proof}

In light of Observation~\ref{obs: bridge}, it is enough to show that \ref{Q2'} and~\ref{Q3'} hold with positive probability.
Note, however, that the events in~\ref{Q2'} and~\ref{Q3'} are defined for $v \in \tilde{U}$ and $c \in L'(v)$.
As $\tilde{U}$ and $L'(v)$ are random sets themselves, working directly with the random variables $|L'(v)|$ and $\deg(F'_c, \ell, v)$ for $v \in \tilde{U}$ and $c \in L'(v)$ presents a challenge.
To overcome this, we will define analogous events for all $v \in U$ and $c \in L(v)$ that will recover the desired conditions in~\ref{Q2'} and~\ref{Q3'}.
This approach is standard in such arguments~\cite{AndersonBernshteynDhawan, bradshaw2025toward, Kim95, iliopoulos2021improved}.
First, for each $v \in U$, let $\cA_v$ denote the event that $|L'(v)| < q_{i+1}$.
We aim to describe a similar event for the color $\ell$-degrees.
Note that by the definition of $F'_c$ in Step~\ref{step: hat-hypergraph} of Algorithm~\ref{algorithm: nibble step}, each edge $e' \in E(F'_c, \ell)$ is of the form $e' = f$ for some pair $(e, f)$ with $e \in \bigcup_{j = \ell}^k E(F_c, j)$ and $f \in \binom{e}{\ell}$ satisfying:
\begin{itemize}
\item for each $u \in f$, $u \in \tilde{U}$ and $c \in L'(u)$; and
\item for each $w \in e \setminus f$, $\psi(w) = c$; 
\end{itemize}
requiring that $f \subseteq \tilde{U}$ while $(e \setminus f) \cap \tilde{U} = \varnothing$, i.e., $e \cap \tilde{U} = f$. Moreover, since $F_c$ is linear (as a result of condition~\ref{Q1} of property $P(i)$) and $|f| = \ell \ge 2$, 
such a pair $(e, f)$ is unique.
Motivated by this observation, for a vertex $v \in U$, a color $c \in L(v)$, and $2 \le \ell \le j \le k$ we introduce the following random variable:
\begin{align}\label{eq: X ell j}
\mathbf{X}_{\ell,j}(v, c) \coloneqq \sum_{e \in E(F_c, j, v)} \sum_{f \in \binom{e-v}{\ell-1}} \I{\left[{\left(\forall u\in f,\, u \in \tilde{U} ~\text{and}~ c \in L'(u)\right)} \wedge {\left(\forall w \in e \setminus (f \cup \{v\}), \, \psi(w) = c\right)}\right]},
\end{align}
and define 
\[
\mathbf{X}_\ell(v,c) \coloneqq \sum_{j = \ell}^k \mathbf{X}_{\ell, j}(v,c).
\]
We note that $\deg(F'_c, \ell, v) = \mathbf{X}_\ell(v,c)$ if $v \in \tilde{U}$ and $c \in L'(v)$.
For each $v \in U$, $c \in L(v)$, and $2 \le \ell \le k$, let $\cE_{v,c,\ell}$ denote the event that
\begin{align}\label{low ell degree}
\mathbf{X}_\ell(v,c) > \binom{k-1}{\ell-1} \left(\frac{\alpha \gamma_{i+1}}{q_{i+1}}\right)^{k-\ell} \beta_{i+1}^{\ell-1}\, t_{i+1}.
\end{align}
The following two lemmas bound the probabilities of the bad events $\cA_v$ and $\cE_{v,c,\ell}$, respectively.

\begin{lemma}\label{lemma: large L'(v)}
For every $v \in U$, $\pr(\cA_v) = \exp\left(-\Omega{\left(d^{\frac{\eta/6}{k-1}}\right)}\right)$.
\end{lemma}

\begin{lemma}\label{lemma: each E v c ell}
For every $v \in U$, $c \in L(v)$, and $2 \le \ell \le k$, $\pr\big(\cE_{v,c,\ell}\big) = \exp\left(-\Omega{\left(d^{\frac{\eta/6}{k-1}}\right)}\right)$.
\end{lemma}

We defer the proof of Lemma~\ref{lemma: large L'(v)} to Section~\ref{subsection: large L'(v)}.
Most of our effort lies in the proof of Lemma~\ref{lemma: each E v c ell}, which we defer to Section~\ref{section: deg}.
With these results in hand, we are ready to prove Lemma~\ref{lemma: nibble}.

\begin{proof}[Proof of Lemma~\ref{lemma: nibble}]

Let $\mathcal{B}$ be the collection of bad events consisting of $\cA_v$ for each $v \in U$, together with $\cE_{v,c,\ell}$ for each $v \in U$, $c \in L(v)$, and $2 \le \ell \le k$. We will apply the symmetric version of the \hyperref[theorem: sym LLL]{Lov\'asz Local Lemma} (Theorem~\ref{theorem: sym LLL}) to $\mathcal{B}$.

By Lemmas~\ref{lemma: large L'(v)} and~\ref{lemma: each E v c ell}, every event in $\mathcal{B}$ occurs with probability at most 
\[
p \coloneqq \exp\left(-\Omega{\left(d^{\frac{\eta/6}{k-1}}\right)}\right).
\]
For each $v \in U$, the events $\cA_v$ and $\cE_{v,c,\ell}$ are determined by the random choices in Steps~\ref{step: activate} and~\ref{step: eq} at the vertices at distance at most $2$ from $v$ in $\bigcup_{c' \in L(U)}F_{c'}$. Hence, each of them is mutually independent of the set of events $\cA_u$ and $\cE_{u,c',\ell'}$ with $u$ at distance at least $5$ from $v$.
Furthermore, we have the following for each vertex $u$ and color $c \in L(u)$ as a result of~\ref{Q3}:
\begin{align}
|N_{F_c}(u)| \,\overset{\ref{Q1}}{=}\, \sum_{\ell=2}^k (\ell-1) \deg(F_c, \ell, u) 
&\,\;\le\;\, (k-1) \sum_{\ell=2}^k \deg(F_c, \ell, u) \\
&\overset{\ref{Q3}}{\le} (k-1) \sum_{\ell=2}^k \binom{k-1}{\ell-1} {\left(\frac{\alpha\gamma_i}{q_i}\right)^{k-\ell}} \beta_i^{\ell-1} t_i \\
&\overset{\ref{gamma_i bound},\,\ref{beta_i bound},\,\ref{q_i bound}}{\le} (k-1) \cdot  t_i \cdot \sum_{\ell=2}^k \binom{k-1}{\ell-1} {\left(\frac{1}{1-\eta}\right)^{k-\ell}}\\
&\overset{\eqref{def: eta, alpha, delta, s}}{\le} (k-1) \cdot t_i \cdot \left(3^{k-1} - 2^{k-1}\right)
\overset{\eqref{eq: compare t_i t_1, q_i q_1},\,\eqref{def: t_1, q_1}}{=} O(d).
\end{align}
It follows that for each vertex $u$, there are $O((q_id)^4) = O(d^8)$ vertices at distance at most $4$ from $u$ in $\bigcup_{c \in L(U)}F_c$.
In particular, each event in $\cB$ is mutually independent of all but at most $O(q_id^8) = O(d^9)$ other events in $\cB$. 
Setting $d_{\mathrm{LLL}} \coloneqq O(d^9)$, we have 
\[
4 p d_{\mathrm{LLL}} = O(d^9) \cdot \exp\left(-\Omega{\left(d^{\frac{\eta/6}{k-1}}\right)}\right) \ll 1
\]
for sufficiently large $d$.
By Theorem~\ref{theorem: sym LLL}, with positive probability no event in $\mathcal{B}$ occurs; that is, $|L'(v)| \ge q_{i+1}$ for every $v \in \tilde{U}$, and the degree bound in (Q3') holds for every $v \in \tilde{U}$, $c \in L'(v)$, and $2 \le \ell \le k$. By Observation~\ref{obs: bridge}, $(\tilde{L}, \tilde{\cF})$ then satisfies conditions~\ref{Q2} and~\ref{Q3} of property $P(i+1)$. Together with Lemma~\ref{lemma: deterministic conditions}, this verifies conditions~\ref{cond: psi compatible}--\ref{cond: reduction} and all conditions of property $P(i+1)$.
\end{proof}

\subsubsection{Concentration of list sizes: proof of Lemma~\ref{lemma: large L'(v)}}\label{subsection: large L'(v)}

Consider $\hat{L}(v) \coloneqq L(v) \setminus L'(v)$. Note that by Step~\ref{step: eq} of Algorithm~\ref{algorithm: nibble step}, a color $c \in L(v)$ belongs to $L'(v)$ if and only if $c \notin D(v)$ and $\eq(v,c) = 1$, and so
\begin{align}\label{eq: prob c in L'(v)}
\pr(c \in \hat{L}(v)) = 1 - \pr(c \in L'(v)) = 1 - \pr(c \notin D(v))\,\pr(\eq(v,c) = 1) \overset{\eqref{def: eq(v,c)}}{=} 1 - \keep_i.
\end{align}
Since $|L(v)| = q_i$ by condition~\ref{Q2} of property $P(i)$, linearity of expectation gives
\begin{align}\label{eq: E hat L(v)}
\E[|\hat{L}(v)|] = \sum_{c \in L(v)} \pr(c \in \hat{L}(v)) = q_i (1-\keep_i).
\end{align}
Since $1-\keep_i = \Omega{\left((\log d)^{-k}\right)}$ and $q_i \gg d^{\frac{\eta}{k-1}}$ by Lemma~\ref{lemma: properties of beta_i, gamma_i, keep_i, uncolor_i}~\ref{keep_i bound} and~\ref{q_i bound}, respectively, we note that
\begin{align}\label{eq: lb E hat L(v)}
\E[|\hat{L}(v)|] \gg d^{\frac{\eta/2}{k-1}}.
\end{align}

We use \hyperref[theorem: Talagrand]{Talagrand's Inequality} (Theorem~\ref{theorem: Talagrand}) to establish concentration for $|\hat{L}(v)|$, where the random trials are, for all vertices $u \in U$, the inclusion of $u$ in $A$, the color choice $c_u$, and the equalizing coin flips for all colors $c' \in L(u)$ at $u$ (all defined in Steps~\ref{step: activate} and~\ref{step: eq} of Algorithm~\ref{algorithm: nibble step}). Before we can apply Talagrand's Inequality, we must show that $|\hat{L}(v)|$ satisfies the Lipschitz and certificate conditions from Definition~\ref{def: lipshcitz and certifiable}.
We do this via a sequence of claims.

\stepcounter{ForClaims}
\renewcommand{\theForClaims}{\ref*{lemma: large L'(v)}}

\begin{claim}\label{claim: Lipschitz |hat L v|}
$|\hat{L}(v)|$ is $1$-Lipschitz.
\end{claim}

\begin{claimproof}
Changing the outcome of the equalizing coin flip for a color $c \in L(v)$ at $v$ can only affect whether $c$ belongs to $\hat{L}(v)$, whereas the equalizing coin flips at any other vertex $u \in U$ have no effect on $\hat{L}(v)$.
Changing the inclusion of a vertex $u \in U$ in $A$ can only affect whether $c_u$ belongs to $\hat{L}(v)$, while changing the color choice of $u$ from $c$ to $c'$ can only remove $c$ from $\hat{L}(v)$ and add $c'$ to it.
Hence, changing the outcome of a single trial changes $|\hat{L}(v)|$ by at most $1$, as claimed.
\end{claimproof}

\begin{claim}\label{claim: certifiable |hat L v|}
$|\hat{L}(v)|$ is $2k$-certifiable.
\end{claim}

\begin{claimproof}
For a color $c \in L(v)$, the event $\{c \in \hat{L}(v)\}$, equivalently $\{c \notin L'(v)\}$, can be certified either by revealing the outcome of its equalizing coin flip at $v$ or by exhibiting an edge $e \in E(F_c, v)$ for which $e-v \subseteq A$ and $c_x = c$ for all $x \in e-v$, requiring at most $2k$ trials.
\end{claimproof}

With Claims~\ref{claim: Lipschitz |hat L v|} and~\ref{claim: certifiable |hat L v|} in hand, we may apply \hyperref[theorem: Talagrand]{Talagrand's Inequality} (Theorem~\ref{theorem: Talagrand}) with $\mu = 1$ and $r = 2k$ to obtain:
\begin{align*}
& \ \ \pr{\left(|\hat{L}(v)| \ge q_i(1-\keep_i) + \delta q_i/2 \right)} \\
&\overset{\eqref{eq: E hat L(v)}}{\le} \pr{\left(||\hat{L}(v)| - \E[|\hat{L}(v)|]| \ge \delta \E[|\hat{L}(v)|]/2\right)} \\
&\overset{\eqref{eq: lb E hat L(v)}}{\le} \pr{\left(||\hat{L}(v)| - \E[|\hat{L}(v)|]| \ge \delta \E[|\hat{L}(v)|]/4 + 20\sqrt{2k \E[|\hat{L}(v)|]} + 128k\right)} \\
&\; \le ~ 4\exp\left(-\frac{\delta^2 \E[|\hat{L}(v)|]}{256k(1+\delta/4)}\right)\\
&\overset{\eqref{eq: lb E hat L(v)}}{=} \exp\left(-\Omega{\left(d^{\frac{\eta/6}{k-1}}\right)}\right),
\end{align*}
where in the last inequality we used $\delta^2 = d^{-\frac{\eta/3}{k-1}}$ and~\eqref{eq: lb E hat L(v)}. As a result, by the recursive definition of $q_{i+1}$ in~\eqref{def: t_i, q_i}, the definition of $\hat{L}(v)$, and the fact that $|L(v)| = q_i$ by condition~\ref{Q2} of property $P(i)$, we obtain
\begin{align}
\pr{\left(|L'(v)| < q_{i+1} \right)} 
&\overset{\eqref{def: t_i, q_i}}{=} 
\pr{\left(|L'(v)| < (1-\delta) q_i \keep_i \right)} \\
&\ \le\  \pr{\left(q_i - |L'(v)| > q_i - (1-\delta) q_i \keep_i\right)} \\
&\overset{\ref{Q2}}{\le}
\pr{\left(|\hat{L}(v)| \ge q_i(1-\keep_i) + \delta q_i\keep_i \right)}\\
&\overset{\ref{keep_i bound}}{\le} 
\pr{\left(|\hat{L}(v)| \ge q_i(1-\keep_i) + \delta q_i/2 \right)} = \exp\left(-\Omega{\left(d^{\frac{\eta/6}{k-1}}\right)}\right),
\end{align}
where the second-to-last inequality follows from the bound $\keep_i \ge 1-\eta > 1/2$ given by Lemma~\ref{lemma: properties of beta_i, gamma_i, keep_i, uncolor_i}~\ref{keep_i bound}.
This completes the proof of Lemma~\ref{lemma: large L'(v)}.

\section{Concentration of color degrees: proof of Lemma~\ref{lemma: each E v c ell}}\label{section: deg}

In this section, we prove Lemma~\ref{lemma: each E v c ell}.
Fix $v \in U$, $c \in L(v)$, and $2 \le \ell \le k$.
For $\ell \le j \le k$, recall that
\begin{align*}
\mathbf{X}_{\ell,j}(v,c) = \sum_{e \in E(F_c, j, v)} \sum_{f \in \binom{e-v}{\ell-1}} \I{\left[{\left(\forall u\in f,\, u \in \tilde{U} ~\text{and}~ c \in L'(u)\right)} \wedge {\left(\forall w \in e \setminus (f \cup \{v\}), \, \psi(w) = c\right)}\right]},
\end{align*}
and that $\mathbf{X}_{\ell}(v,c) = \sum_{j = \ell}^k \mathbf{X}_{\ell, j}(v,c)$.
As $v$ and $c$ are fixed for the remainder of this section, we let 
\[
\mathbf{X}_{\ell, j} = \mathbf{X}_{\ell, j}(v,c) \qquad \text{and} \qquad \mathbf{X}_\ell = \mathbf{X}_\ell(v,c).
\]
The standard strategy for proving results such as Lemma~\ref{lemma: each E v c ell} is to first compute $\E[\mathbf{X}_\ell]$ and then show that $\mathbf{X}_\ell$ is concentrated about its mean.
Unfortunately, the random variable $\mathbf{X}_\ell$ need not be sufficiently concentrated for our purposes (we elaborate further on this below).
Instead, we show that $\mathbf{X}_\ell$ is upper bounded by a different random variable, which is concentrated about its mean.
This technique is fairly standard in graph and hypergraph coloring problems (see, e.g., \cite{Kim95, AndersonBernshteynDhawan, bradshaw2025toward, iliopoulos2021improved}).

As mentioned in Section~\ref{section: proof overview}, a key distinction between our approach and other results that employ similar arguments is that we prove a much tighter concentration result, which in turn leads to significantly less cumbersome recursive computations (see Section~\ref{subsection: properties}).
To achieve this result, we do not compute $\E[\mathbf{X}_\ell]$ directly, but instead work solely with the auxiliary random variables we define.
We now discuss step-by-step how and why we define the desired auxiliary random variables.

Recall from Step~\ref{step: assignment} of Algorithm~\ref{algorithm: nibble step} that $\psi(w) = c$ if $w$ is activated, $c_w = c$, and $c \in L'(w)$. Therefore, 
\begin{align}\label{def: X ell j}
\mathbf{X}_{\ell, j} = \sum_{e \in E(F_c, j, v)}\sum_{f \in \binom{e-v}{\ell-1}} {\left(\prod_{u \in f}\I[u \in \tilde{U}]\right)} {\left(\prod_{w \in e \setminus (f \cup \{v\})} \I[w \in A, \, c_w=c]\right)} {\left(\prod_{x \in e-v} \I[c \in L'(x)] \right)}.
\end{align}
Our goal is to eventually apply \hyperref[theorem: Talagrand]{Talagrand's Inequality} (Theorem~\ref{theorem: Talagrand}).
We note that the random variable $\mathbf{X}_{\ell, j}$ requires a large certificate parameter $r$.
Indeed, certifying $c \in L'(x)$ requires revealing at least one trial for each edge in $E(F_c, x)$, that is, at least $\deg(F_c, x)$ trials.
To remedy this, we express $\mathbf{X}_{\ell, j}$ as the difference of two random variables $\mathbf{Y}_{\ell, j}$ and $\mathbf{Z}_{\ell, j}$ defined as follows:
\begin{align}\label{def: Y ell j} 
\mathbf{Y}_{\ell, j} 
\coloneqq \sum_{e \in E(F_c,j,v)} \sum_{f \in \binom{e-v}{\ell-1}}  {\left(\prod_{u \in f} \I[u \in \tilde{U}]\right)} {\left(\prod_{w \in e \setminus (f \cup \{v\})}\I[w \in A, \, c_w=c] \right)}, 
\end{align}
\begin{align}\label{def: Z ell j} 
\mathbf{Z}_{\ell, j} 
\coloneqq \sum_{e \in E(F_c,j,v)} \sum_{f \in \binom{e-v}{\ell-1}}  {\left(\prod_{u \in f} \I[u \in \tilde{U}]\right)} {\left(\prod_{w \in e \setminus (f \cup \{v\})}\I[w \in A, \, c_w=c] \right)} {\left(1-\prod_{x \in e-v} \I[c \in L'(x)] \right)}.
\end{align}
In other words, we let $\mathbf{Y}_{\ell, j}$ count the number of pairs $(e, f)$ with $e \in E(F_c, j, v)$ and $f \in \binom{e-v}{\ell-1}$ satisfying:
\begin{itemize}
\item for each $u \in f$, $u \in \tilde{U}$; and
\item for each $w \in e \setminus (f \cup \{v\})$, $w \in A$ and $c_w = c$;
\end{itemize}
whereas $\mathbf{Z}_{\ell,j}$ counts the number of all such pairs counted by $\mathbf{Y}_{\ell, j}$ satisfying the extra condition that $c \notin L'(x)$ for some $x \in e-v$.
It turns out that both $\mathbf{Y}_{\ell, j}$ and $\mathbf{Z}_{\ell, j}$ are 
$O(1)$-certifiable.
Unfortunately, these random variables have a large Lipschitz parameter.
In particular, the random choices at $v$ can cause large fluctuations in both $\mathbf{Y}_{\ell, j}$ and $\mathbf{Z}_{\ell, j}$.
To overcome this obstacle, we define new random variables which ``ignore'' the effects of edges containing $v$.

For each vertex $u \in U$, let $D^*(u)$ be the set of colors $c^* \in L(u)$ that were discarded at $u$ via an edge not containing $v$, i.e., there exists an edge $e' \in E(F_{c^*}, u) \setminus E(F_{c^*}, v)$ such that every vertex $x \in e'-u$ satisfies $x \in A$ and $c_x = c^*$.
For every vertex $u \in U$, define 
\begin{align}\label{def: L*}
L^*(u) \coloneqq \{c^* \in L(u) \,:\, c^* \notin D^*(u), \ \eq(u,c^*) = 1\}.
\end{align}
We also let 
\begin{align}\label{def: U*}
U^* \coloneqq \{u \in U \,:\, u \notin A ~\text{or}~ c_u \notin L^*(u)\}.
\end{align}
In other words, $U^*$ consists of the uncolored vertices $u$ such that either $u$ is not activated, the color choice $c_u$ loses its equalizing coin flip, or $c_u$ is discarded via an edge not containing $v$.
Observe that
\begin{align}\label{eq: relate D* L* U*}
D^*(u) \subseteq D(u), \qquad L'(u) \subseteq L^*(u), \qquad U^* \subseteq \tilde{U}.
\end{align}
With these definitions and observations in hand, we introduce three auxiliary random variables.
First, we define $\mathbf{Y}_{\ell, j}^{(1)}$ to be the number of pairs $(e, f)$ with $e \in E(F_c, j, v)$ and $f \in \binom{e-v}{\ell-1}$ satisfying:
\begin{itemize}
\item for each $u \in f$, $u \in U^*$ (i.e., $u \notin A$, $c_u \in D^*(u)$, or $\eq(u, c_u)=0$ by~\eqref{def: L*} and~\eqref{def: U*}); and
\item for each $w \in e \setminus (f \cup \{v\})$, $w \in A$ and $c_w = c$.
\end{itemize}
On the other hand, 
consider a pair $(e, f)$ with $e \in E(F_c, j, v)$ and $f \in \binom{e-v}{\ell-1}$ counted by $\mathbf{Y}_{\ell, j}$ but not by $\mathbf{Y}_{\ell, j}^{(1)}$. Then, for some $u \in f$, we have $u \in \tilde{U}$ and $u \notin U^*$, and so $u \in A$
and $c_u \in D(u) \setminus D^*(u)$.
As a result, there exists an edge $e' \in E(F_{c_u}, u) \cap E(F_{c_u}, v)$ such that every $x \in e'- u$ satisfies $x \in A$ and $c_x = c_u$; and moreover, since $u, v \in e \cap e'$, condition~\ref{Q1} of property $P(i)$ implies that $e' = e$.
In particular, it follows that every $x \in e - v$ satisfies $x \in A$ and $c_x = c_v$. Combining with the condition for each $w \in e \setminus (f \cup \{v\})$, such a pair $(e, f)$ counted by $\mathbf{Y}_{\ell, j}$ but not by $\mathbf{Y}_{\ell, j}^{(1)}$ must satisfy that:
\begin{itemize}
\item for each $x \in e-v$, $x \in A$ and $c_x = c_v$; and
\item for each $w \in e \setminus (f \cup \{v\})$, $c_w = c$.
\end{itemize}
Motivated by this, we set
\begin{align}\label{def: Y ell j 1}
\mathbf{Y}_{\ell, j}^{(1)} \coloneqq \sum_{e \in E(F_c, j, v)}\sum_{f \in \binom{e-v}{\ell-1}} {\left(\prod_{u \in f}\I[u \in U^*]\right)} {\left(\prod_{w \in e \setminus (f \cup \{v\})} \I[w \in A, \, c_w=c]\right)},
\end{align}
\begin{align}\label{def: Y ell j 2}
\mathbf{Y}_{\ell, j}^{(2)} \coloneqq \sum_{e \in E(F_c, j, v)}\sum_{f \in \binom{e-v}{\ell-1}} {\left(\prod_{x \in e-v} \I[x \in A, \, c_x = c_v] \right)} {\left(\prod_{w \in e \setminus (f \cup \{v\})} \I[c_w=c]\right)},
\end{align}
and note that  
\begin{align}\label{eq: compare Y ell j 1, Y ell j, and Y ell j 1 + Y ell j 2}
\mathbf{Y}_{\ell,j}^{(1)} \le \mathbf{Y}_{\ell,j} \le \mathbf{Y}_{\ell,j}^{(1)} + \mathbf{Y}_{\ell,j}^{(2)},
\end{align}
where the first inequality follows from the fact that $U^* \subseteq \tilde{U}$.
Finally, we define $\mathbf{Z}_{\ell,j}^{(1)}$ to be the number of all pairs counted by $\mathbf{Y}_{\ell, j}^{(1)}$ satisfying the extra condition that $c \notin L^*(x)$ for some $x \in e-v$, that is,
\begin{equation}\label{def: Z ell j 1}
\mathbf{Z}_{\ell, j}^{(1)} \coloneqq \sum_{e \in E(F_c, j, v)}\sum_{f \in \binom{e-v}{\ell-1}} {\left(\prod_{u \in f}\I[u \in U^*]\right)} {\left(\prod_{w \in e \setminus (f \cup \{v\})} \I[w \in A, \, c_w=c]\right)}{\left(1-\prod_{x \in e-v} \I[c \in L^*(x)] \right)}.
\end{equation}
Then, we note that
\begin{align}\label{eq: compare Z ell j}
\mathbf{Z}_{\ell,j}^{(1)} \le \mathbf{Z}_{\ell,j} \qquad \text{and} \qquad \mathbf{Z}_{\ell,j}^{(1)} \le \mathbf{Y}_{\ell,j}^{(1)}.
\end{align}
Combining the above with~\eqref{eq: compare Y ell j 1, Y ell j, and Y ell j 1 + Y ell j 2}, we have
\begin{align}\label{eq: compare X ell j and Y ell j 1, Y ell j 2, Z ell j 1}
\mathbf{X}_{\ell, j} 
= \mathbf{Y}_{\ell, j} - \mathbf{Z}_{\ell, j} \le \mathbf{Y}^{(1)}_{\ell,j} + \mathbf{Y}^{(2)}_{\ell,j} - \mathbf{Z}^{(1)}_{\ell,j}.
\end{align}
We will show that the random variables $\mathbf{Y}^{(1)}_{\ell,j}$ and $\mathbf{Y}^{(2)}_{\ell,j}$ are concentrated from above, while $\mathbf{Z}^{(1)}_{\ell,j}$ is concentrated from below.
We begin by computing the expectations.

\begin{lemma}\label{lemma: Y 1 Y 2 Z 1 expectations}
The following hold:
\begin{enumerate}[label=\textup{(Exp\arabic*)}, wide]
    \item\label{item: Y 1 Z 1} $\E{\left[\mathbf{Z}^{(1)}_{\ell,j}\right]} \le \E{\left[\mathbf{Y}^{(1)}_{\ell,j}\right]} \le \dbinom{k-1}{\ell-1}\dbinom{k-\ell}{j-\ell}(\alpha/q_i)^{k-\ell}\uncolor_i^{\ell-1}\beta_i^{j-1}\gamma_i^{k-j}t_i$.

    \item\label{item: Y 2} $\E{\left[\mathbf{Y}^{(2)}_{\ell,j}\right]} \le  \dbinom{k-1}{\ell-1} \dbinom{k-\ell}{j-\ell} (\alpha/q_i)^{k-1}\gamma_i^{k-j}\beta_i^{j-1}t_i$.

    \item\label{item: Y 1 - Z 1} $\E{\left[\mathbf{Y}^{(1)}_{\ell, j} - \mathbf{Z}^{(1)}_{\ell, j}\right]} \le \dbinom{k-1}{\ell-1} \dbinom{k-\ell}{j-\ell} (\alpha/q_i)^{k-\ell} \uncolor_i^{\ell-1}[(1+5/q_i) \, \beta_i \, \keep_i]^{j-1} \gamma_i^{k-j} t_i$.
\end{enumerate}
\end{lemma}

We defer the proof of Lemma~\ref{lemma: Y 1 Y 2 Z 1 expectations} to Section~\ref{subsection: expectation}.
Recall from~\eqref{def: eta, alpha, delta, s} that $\delta = d^{-\frac{\eta/6}{k-1}}$, and set
\begin{align}\label{def: tau}
\tau_j \coloneqq \delta^2 \binom{k-1}{\ell-1}\binom{k-\ell}{j-\ell}(\alpha/q_i)^{k-\ell}\beta_i^{j-1}\gamma_i^{k-j}t_i.
\end{align}
Next, we show that the random variables are concentrated:

\begin{lemma}\label{lemma: Y 1 Y 2 Z 1 concentration}
Suppose that $\deg(F_c, j, v) > 0$. Then the following hold:
\begin{enumerate}[label=\textup{(Conc\arabic*)}, wide]
    \item\label{item: Y 1 conc} $\pr{\left(\mathbf{Y}_{\ell, j}^{(1)} - \E{\left[\mathbf{Y}^{(1)}_{\ell,j}\right]} \ge \tau_j/2\right)} \le \exp\left(-\Omega{\left(d^{\frac{\eta/6}{k-1}}\right)}\right)$.

    \item\label{item: Y 2 conc} $\pr{\left(\mathbf{Y}_{\ell, j}^{(2)} \ge \tau_j/2\right)} \le \exp\left(-\Omega{\left(d^{\frac{\eta/2}{k-1}}\right)}\right)$.

    \item\label{item: Z 1 conc} $\pr{\left(\mathbf{Z}_{\ell, j}^{(1)} - \E{\left[\mathbf{Z}^{(1)}_{\ell,j}\right]} \le -\tau_j\right)} \le \exp\left(-\Omega{\left(d^{\frac{\eta/6}{k-1}}\right)}\right)$.
\end{enumerate}
\end{lemma}

We defer the proof of Lemma~\ref{lemma: Y 1 Y 2 Z 1 concentration} to Section~\ref{subsection: concentration}.
With Lemmas~\ref{lemma: Y 1 Y 2 Z 1 expectations} and~\ref{lemma: Y 1 Y 2 Z 1 concentration} in hand, we are ready to prove Lemma~\ref{lemma: each E v c ell}.

\begin{proof}[Proof of Lemma~\ref{lemma: each E v c ell}]
If $\deg(F_c, j, v) = 0$, then $\mathbf{X}_{\ell, j} = \mathbf{Y}_{\ell, j}^{(1)} = \mathbf{Y}_{\ell, j}^{(2)} = \mathbf{Z}_{\ell, j}^{(1)} = 0$, and so all the bounds below hold trivially for this value of $j$. 
By Lemma~\ref{lemma: Y 1 Y 2 Z 1 concentration} and a union bound over the remaining values of $j$, with probability at least 
\[
1 - 3k\exp\left(-\Omega{\left(d^{\frac{\eta/6}{k-1}}\right)}\right) = 1 - \exp\left(-\Omega{\left(d^{\frac{\eta/6}{k-1}}\right)}\right),
\] 
the following hold simultaneously for all $j$ with $\ell \le j \le k$:
\begin{align}\label{eq: bad event bound}
\mathbf{Y}^{(1)}_{\ell,j} \le \E{\left[\mathbf{Y}^{(1)}_{\ell,j}\right]} + \tau_j/2, \qquad \mathbf{Y}^{(2)}_{\ell,j} \le \tau_j/2, \qquad \mathbf{Z}^{(1)}_{\ell,j} \ge \E{\left[\mathbf{Z}^{(1)}_{\ell,j}\right]} -\tau_j. 
\end{align}
Together with~\eqref{eq: compare X ell j and Y ell j 1, Y ell j 2, Z ell j 1}, this implies that
\begin{align}
\mathbf{X}_{\ell, j} 
\le \mathbf{Y}^{(1)}_{\ell,j} + \mathbf{Y}^{(2)}_{\ell,j} - \mathbf{Z}^{(1)}_{\ell,j} 
\le \left(\E{\left[\mathbf{Y}^{(1)}_{\ell,j}\right]} + \tau_j/2\right) + \tau_j/2 - \left(\E{\left[\mathbf{Z}^{(1)}_{\ell,j}\right]} -\tau_j\right) 
= \E[\mathbf{Y}_{\ell, j}^{(1)} - \mathbf{Z}_{\ell, j}^{(1)}] + 2\tau_j,
\end{align}
and so by Lemma~\ref{lemma: Y 1 Y 2 Z 1 expectations}~\ref{item: Y 1 - Z 1}, we have 
\begin{align}
\mathbf{X}_{\ell, j} 
\le \binom{k-1}{\ell-1}(\alpha/q_i)^{k-\ell}t_i \binom{k-\ell}{j-\ell} \gamma_i^{k-j} \beta_i^{j-1} {\left(\uncolor_i^{\ell-1} [(1+5/q_i)\,\keep_i]^{j-1} + 2\delta^2\right)},
\end{align}
for all $\ell \le j \le k$ with probability at least 
$1 - \exp\left(-\Omega{\left(d^{\frac{\eta/6}{k-1}}\right)}\right)$. 
As a result,
\begin{align}
\mathbf{X}_\ell = \sum_{j = \ell}^k \mathbf{X}_{\ell, j} 
&\le \binom{k-1}{\ell-1}(\alpha/q_i)^{k-\ell}t_i{\left(\uncolor_i^{\ell-1} S_1 + 2 \delta^2 S_2\right)}, \label{eq: X ell composite}
\end{align}
with probability at least 
$1 - \exp\left(-\Omega{\left(d^{\frac{\eta/6}{k-1}}\right)}\right)$, where
\begin{align*}
S_1 \coloneqq \sum_{j = \ell}^k\binom{k-\ell}{j-\ell}\gamma_i^{k-j}[(1+5/q_i)\,\beta_i\,\keep_i]^{j-1},  \qquad 
S_2 \coloneqq \sum_{j = \ell}^k\binom{k-\ell}{j-\ell}\gamma_i^{k-j}\beta_i^{j-1}.
\end{align*}
We note that
\begin{align}
S_2 = \beta_i^{\ell-1}(\gamma_i + \beta_i)^{k-\ell} 
\overset{\eqref{def: uncolor_i, beta_i, gamma_i}}{=} \beta_i^{\ell-1}\gamma_{i+1}^{k-\ell},
\end{align}
and, similarly,
\begin{align*}
S_1 &= \sum_{j = \ell}^k\binom{k-\ell}{j-\ell}\gamma_i^{k-j}[(1+5/q_i)\,\beta_i\,\keep_i]^{j-1} \\
&\le (1+5/q_i)^{k - 1}(\beta_i\,\keep_i)^{\ell - 1}\sum_{j = \ell}^k\binom{k-\ell}{j-\ell}\gamma_i^{k-j}\beta_i^{j-\ell} 
\overset{\eqref{def: uncolor_i, beta_i, gamma_i}}{=} (1+5/q_i)^{k - 1}(\beta_i\,\keep_i)^{\ell - 1}\gamma_{i+1}^{k - \ell}.
\end{align*}
By substituting into~\eqref{eq: X ell composite}, we obtain that 
\begin{align*}
\mathbf{X}_\ell &\le \binom{k-1}{\ell-1}(\alpha\gamma_{i+1}/q_i)^{k-\ell}\beta_i^{\ell-1}t_i \left[\uncolor_i^{\ell-1}\keep_i^{\ell-1} (1+5/q_i)^{k-1} + 2\delta^2 \right] \\
&= \binom{k-1}{\ell-1}(\alpha\gamma_{i+1}/(q_i \, \keep_i))^{k-\ell} (\beta_i \, \uncolor_i)^{\ell-1} \, t_i \, \keep_i^{k-1} {\left[(1+5/q_i)^{k-1} + \frac{2\delta^2}{(\uncolor_i\,\keep_i)^{\ell-1}} \right]},
\end{align*}
and so by the definitions of $q_{i+1}$ and  $t_{i+1}$ given by~\eqref{def: t_i, q_i} and the definition of $\beta_{i+1}$ from~\eqref{def: uncolor_i, beta_i, gamma_i}, we get
\begin{align*}
\mathbf{X}_\ell &\le \binom{k-1}{\ell-1}(\alpha\gamma_{i+1}/q_{i+1})^{k-\ell}\beta_{i+1}^{\ell-1} \, t_{i+1} \, \frac{(1-\delta)^{k-\ell}}{(1 + \delta)^{k-1}} {\left[(1+5/q_i)^{k-1} + \frac{2\delta^2}{(\uncolor_i\,\keep_i)^{\ell-1}} \right]}\\
&\le \binom{k-1}{\ell-1}(\alpha\gamma_{i+1}/q_{i+1})^{k-\ell}\beta_{i+1}^{\ell-1} \, t_{i+1} \, \frac{1}{1 + \delta} {\left[(1+5/q_i)^{k-1} + \frac{2\delta^2}{(\uncolor_i\,\keep_i)^{\ell-1}} \right]},
\end{align*}
with probability at least 
$1 - \exp\left(-\Omega{\left(d^{\frac{\eta/6}{k-1}}\right)}\right)$.
It now suffices to show that
\[
(1+5/q_i)^{k-1} + \frac{2\delta^2}{(\uncolor_i\,\keep_i)^{\ell-1}} \le 1 + \delta.
\]
To this end, note that, by Lemma~\ref{lemma: properties of beta_i, gamma_i, keep_i, uncolor_i}~\ref{keep_i bound} and~\ref{uncolor_i bound}, we have  $\uncolor_i\,\keep_i \ge (1-\eta)^2 \ge 1/2$, and thus, the left-hand side above is at most
\[
(1+5/q_i)^{k-1} + 2^k\delta^2.
\]
Additionally, since $1/q_i \ll d^{-\frac{\eta}{k-1}} = \delta^6$ by Lemma~\ref{lemma: properties of beta_i, gamma_i, keep_i, uncolor_i}~\ref{q_i bound} and the definition of $\delta$ in~\eqref{def: eta, alpha, delta, s},
it follows that
\[
(1+5/q_i)^{k-1} + 2^k\delta^2 \le (1+\delta^6)^{k-1} + 2^k\delta^2 < 1 + 2^{k-1} \delta^6 + 2^k\delta^2 \le 1 + \delta,
\]
for $d$ sufficiently large, completing the proof.
\end{proof}

\subsection{Expectation computations: proof of Lemma~\ref{lemma: Y 1 Y 2 Z 1 expectations}}\label{subsection: expectation}
\stepcounter{ForClaims}
\renewcommand{\theForClaims}{\ref*{lemma: Y 1 Y 2 Z 1 expectations}}

We will further split this section into three subsections to handle the proofs of each item in the statement of Lemma~\ref{lemma: Y 1 Y 2 Z 1 expectations}.
Before we do so, we make a key definition and prove a useful auxiliary result.
For each vertex $u$ and color $c' \in L(u)$, let 
\[N^*_{c'}(u) \coloneqq 
\bigcup\{e'-u : e' \in E(F_{c'},u) \setminus E(F_{c'},v)\}
\qquad \text{and} \qquad \cN^*(u) \coloneqq \bigcup_{c' \in L(u)}N_{c'}^*(u).\]

\begin{obs}\label{obs: disjoint N*}
    Let $e \in E(F_c, v)$ and let $x, y \in e-v$ be distinct.
    For any $c' \in L(x)$ and $c'' \in L(y)$, we have $y \notin N^*_{c'}(x)$ and $N^*_{c'}(x) \cap N^*_{c''}(y) = \varnothing$.
\end{obs}

\begin{proof}
    We first check that $y \notin N^*_{c'}(x)$. Suppose otherwise; then there exists an edge $e_1 \in E(F_{c'}, x)$ for which $v \notin e_1$ and $y \in e_1 - x$. Since $v \in e \in E(F_c)$, this implies that $e \neq e_1$ and $x, y \in e_1 \cap e$, contradicting condition~\ref{Q1} of property $P(i)$.
    
    We next check that $N^*_{c'}(x) \cap N^*_{c''}(y) = \varnothing$. Suppose that there is some vertex $z$ in $N^*_{c'}(x) \cap N^*_{c''}(y)$. Then there exist $e_1 \in E(F_{c'}, x)$ and $e_2 \in E(F_{c''}, y)$ for which $v \notin e_1 \cup e_2$ and $z \in (e_1 - x) \cap (e_2 - y)$. 
    If $e_1 = e_2$, then 
    since $v \in e \in E(F_c)$, this implies that $e \neq e_1$ and $x, y \in e_1 \cap e$, contradicting condition~\ref{Q1} of property $P(i)$. 
    If $e_1 \ne e_2$, then 
    $e$, $e_1$, and $e_2$ are three distinct edges.
    Moreover, applying the first part of the observation to $(x, y)$ shows that $z \ne y$, while applying it to $(y, x)$ shows that $z \ne x$.  
    Since $x \in e \cap e_1$, $y \in e \cap e_2$, and $z \in e_1 \cap e_2$, the three edges $e, e_1, e_2$ together with the three distinct vertices $x, y, z$ form a cycle of length $3$ in the union of the hypergraphs in $\cF$, 
    contradicting condition~\ref{Q1} of property $P(i)$.
\end{proof}

\subsubsection{Upper bound for $\E[\mathbf{Y}^{(1)}_{\ell,j}]$}

By linearity of expectation and~\eqref{def: Y ell j 1}, we have
\begin{align}\label{eq: linear expectation Y ell j 1}
\E{\left[\mathbf{Y}^{(1)}_{\ell, j}\right]} = \sum_{e \in E(F_c, j, v)}\sum_{f \in \binom{e - v}{\ell - 1}}  \E{\left[{\left(\prod_{u \in f}\I[u \in U^*]\right)} {\left(\prod_{w \in e \setminus (f \cup \{v\})}\I[w \in A,\, c_w = c]\right)}\right]}.
\end{align}
Note that it suffices to show that the following holds for each $e \in E(F_c, j, v)$ and $f \in \tbinom{e-v}{\ell-1}$:
\begin{align}
\label{eq: prob Y ell j 1}
\E{\left[{\left(\prod_{u \in f}\I[u \in U^*]\right)} {\left(\prod_{w \in e \setminus (f \cup \{v\})}\I[w \in A,\, c_w = c]\right)}\right]}
\le \uncolor_i^{\ell-1}(\alpha/q_i)^{j-\ell}.
\end{align}
Indeed, by substituting~\eqref{eq: prob Y ell j 1} into~\eqref{eq: linear expectation Y ell j 1} and using~\eqref{eq: compare Z ell j} and~\ref{Q3}, we obtain
\begin{align*}
\E{\left[\mathbf{Z}^{(1)}_{\ell, j}\right]} \overset{\eqref{eq: compare Z ell j}}{\le} \E{\left[\mathbf{Y}^{(1)}_{\ell, j}\right]}
&\le \uncolor_i^{\ell-1}(\alpha/q_i)^{j-\ell} \binom{j-1}{\ell-1} 
\deg(F_c, j, v) \\
&\overset{\ref{Q3}}{\le} \uncolor_i^{\ell-1}(\alpha/q_i)^{j-\ell}\binom{j-1}{\ell-1}\binom{k-1}{j-1}(\alpha\gamma_i/q_i)^{k-j}\beta_i^{j-1}t_i \\
&= \binom{k-1}{\ell-1}\binom{k-\ell}{j-\ell}(\alpha/q_i)^{k-\ell}\uncolor_i^{\ell-1}\beta_i^{j-1}\gamma_i^{k-j}t_i,
\end{align*}
as desired.

Our goal now is to prove~\eqref{eq: prob Y ell j 1} for a fixed edge $e \in E(F_c, j, v)$ and a fixed subset $f \in \binom{e-v}{\ell-1}$. Set $g \coloneqq e \setminus (f \cup \{v\})$ (see Fig.~\ref{fig: edge shrinking}). Our strategy is first to establish that $\{\I[u \in U^*]\,:\,u \in f\} \cup \{\I[w \in A,\, c_w = c]\,:\,w \in g\}$ is a family of mutually independent random variables, and then to compute the expectation of each indicator. We do this via a sequence of claims.

\begin{figure}[htb!]
    \centering

    \begin{tikzpicture}[scale=1.65, every node/.style={scale=1.75}]

    \draw (-2.25,0) ellipse (2.25 and 0.65);
    \node[above] at (-2.25, 0.65) {\scalebox{0.6}{$e$}};

    \node[below right]    at (-2.25, -0.25) {\scalebox{0.5}{$f{\cup}\{v\}$}};
    \fill[blue!20, opacity=0.4]
         (-1.25,0) ellipse (1.25 and 0.3);
    \draw[thick, dashed, black, rounded corners=0pt]
         (-1.25,0) ellipse (1.25 and 0.3);
    
    \node[right]    at (0, 0) {\scalebox{0.6}{$v$}};
    \node[circle,line width=0.2pt,fill=yellow,draw,
          inner sep=0pt,minimum size=2.5pt] at (0,0) {};

    \node[font=\footnotesize] at (-1, 0) {$\cdots$};
    \node[circle,line width=0.2pt,fill=white,draw,
          inner sep=0pt,minimum size=2.5pt] at (-1.5,0) {};
    \node[circle,line width=0.2pt,fill=white,draw,
          inner sep=0pt,minimum size=2.5pt] at (-0.5,0) {};
    \node[circle,line width=0.2pt,fill=white,draw,
          inner sep=0pt,minimum size=2.5pt] at (-2,0) {};

    \fill[green!20, rounded corners=10pt, opacity=0.8]
         (-2.65, -0.2) rectangle (-4.35,0.2);
    \node[below right]    at (-3.75, -0.15) {\scalebox{0.5}{$g$}};
    \node[circle,line width=0.2pt,fill=green!60!black,draw,
          inner sep=0pt,minimum size=2.5pt] at (-3,0) {};
    \node[font=\footnotesize]    at (-3.5, 0) {$\cdots$};
    \node[circle,line width=0.2pt,fill=green!60!black,draw,
          inner sep=0pt,minimum size=2.5pt] at (-4,0) {};

\end{tikzpicture}
    \caption{An edge $e \in E(F_c, j, v)$, a subset $f \in \binom{e - v}{\ell - 1}$, and the remaining set $g \coloneqq e \setminus (f \cup \{v\})$.} 
    \label{fig: edge shrinking}
\end{figure}

\begin{claim}\label{claim: mutual independence Y ell j 1}
    $\{\I[u \in U^*] : u \in f\} \cup \{\I[w \in A,\, c_w = c] : w \in g\}$ is a family of mutually independent random variables.
\end{claim}

\begin{claimproof}
    We identify precisely which underlying random trials determine each of the random variables. The random trials in Algorithm~\ref{algorithm: nibble step} at a vertex $y \in U$ consist of the inclusion of $y$ in $A$, the color choice $c_y$, and the collection of equalizing coin flips $\{\eq(y, c')\,:\,c' \in L(y)\}$. All trials at distinct vertices are independent, and the equalizing coin flips are mutually independent of each other and of all other trials.
    
    Each $\I[u \in U^*]$ is a function of the trials at the vertices in $\{u\} \cup \cN^*(u)$, while each $\I[w \in A,\, c_w = c]$ is a function of the trials at $w$. 
    Since the equalizing coin flips are independent of each other and of all activations and color choices, it suffices to check that the vertex sets $\{u\} \cup \cN^*(u)$ for $u \in f$ and $\{w\}$ for $w \in g$ are pairwise disjoint.
    This follows from Observation~\ref{obs: disjoint N*}, which applies to every choice of colors in the unions defining $\cN^*(u)$.

    Combining these observations with the mutual independence of the equalizing coin flips, we conclude that $\{\I[u \in U^*]\,:\,u \in f\} \cup \{\I[w \in A,\, c_w = c]\,:\,w \in g\}$ is a family of mutually independent random variables.
\end{claimproof}

By Claim~\ref{claim: mutual independence Y ell j 1}, we have
\begin{equation}\label{eq: factoring each term E Y ell j 1}
\E{\left[{\left(\prod_{u \in f}\I[u \in U^*]\right)} {\left(\prod_{w \in e \setminus (f \cup \{v\})}\I[w \in A,\, c_w = c]\right)}\right]} = \prod_{u \in f}\pr(u \in U^*)\prod_{w \in g}(\pr(w \in A)\,\pr(c_w = c)).
\end{equation}
For each $u \in f$, since $D^*(u) \subseteq D(u)$, we have $\pr(c' \notin D^*(u)) \ge \pr(c' \notin D(u))$ for every color $c'$. Using~\eqref{def: eq(v,c)}, we get
\[
\pr(u \in U^*) = 1 - \alpha\sum_{c' \in L(u)}\pr(c_u = c',\,c' \in L^*(u)) \le 1 - \alpha\keep_i = \uncolor_i.
\]
On the other hand, for each $w \in g$, since $\pr(w \in A) = \alpha$ by Step~\ref{step: activate} in Algorithm~\ref{algorithm: nibble step} and $c_w$ is chosen uniformly at random from $L(w)$, we have $\pr(w \in A)\,\pr(c_w = c) = \alpha/q_i$, given that $|L(w)| = q_i$ by condition~\ref{Q2} of property $P(i)$.
Substituting these into~\eqref{eq: factoring each term E Y ell j 1} with $|f| = \ell - 1$ and $|g| = j-\ell$ gives~\eqref{eq: prob Y ell j 1}, completing the proof of Lemma~\ref{lemma: Y 1 Y 2 Z 1 expectations}~\ref{item: Y 1 Z 1}.

\subsubsection{Upper bound for $\E[\mathbf{Y}^{(2)}_{\ell,j}]$}

By linearity of expectation and~\eqref{def: Y ell j 2}, we have
\begin{align}\label{eq: linear expectation Y ell j 2}
\E{\left[\mathbf{Y}^{(2)}_{\ell, j}\right]} = \sum_{e \in E(F_c,j,v)}\sum_{f \in \binom{e-v}{\ell-1}} \E{\left[{\left(\prod_{x \in e-v}\I[x \in A,\, c_x = c_v]\right)} {\left(\prod_{w \in e \setminus (f \cup \{v\})}\I[c_w = c]\right)}\right]}.
\end{align}
It suffices to show that the following holds for each $e \in E(F_c, j, v)$ and $f \in \tbinom{e-v}{\ell-1}$:
\begin{align}
\label{eq: prob Y ell j 2}
\E{\left[{\left(\prod_{x \in e-v}\I[x \in A,\, c_x = c_v]\right)} {\left(\prod_{w \in e \setminus (f \cup \{v\})}\I[c_w = c]\right)}\right]}
\le (\alpha/q_i)^{j-1}.
\end{align}
Indeed, by substituting~\eqref{eq: prob Y ell j 2} into~\eqref{eq: linear expectation Y ell j 2} and applying~\ref{Q3}, we obtain
\begin{align*}
\E{\left[\mathbf{Y}^{(2)}_{\ell, j}\right]}
&\le (\alpha/q_i)^{j-1} \binom{j-1}{\ell-1} 
\deg(F_c, j, v) \\
&\overset{\ref{Q3}}{\le} (\alpha/q_i)^{j-1} \binom{j-1}{\ell-1}\binom{k-1}{j-1}(\alpha\gamma_i/q_i)^{k-j}\beta_i^{j-1}t_i \\
&= \binom{k-1}{\ell-1}\binom{k-\ell}{j-\ell}(\alpha/q_i)^{k-1} \gamma_i^{k-j} \beta_i^{j-1}t_i,
\end{align*}
as desired.

Our goal now is to prove~\eqref{eq: prob Y ell j 2} for a fixed edge $e \in E(F_c, j, v)$ and a fixed subset $f \in \binom{e-v}{\ell-1}$. Set $g \coloneqq e \setminus (f \cup \{v\})$ (see Fig.~\ref{fig: edge shrinking}).
The trial at $v$ (specifically, $c_v$) is independent of the trials at the vertices in $e-v$, which are mutually independent. 

Consider first the case $j > \ell$, so $g \ne \varnothing$. Fix any $w \in g$; then $w \in e-v$, so the first product requires $c_w = c_v$ while the second product requires $c_w = c$, and these two conditions hold together only when $c_v = c$. Since the trials at the vertices in $e - v$ are independent of $c_v$ and mutually independent of each other, conditional on $c_v = c$, we have
\[
\E{\left[{\left(\prod_{x \in e-v}\I[x \in A,\, c_x = c_v]\right)} {\left(\prod_{w \in g}\I[c_w = c]\right)} \,\bigg|\, c_v = c\right]}
= \prod_{x \in e-v} \pr(x \in A) \, \pr(c_x = c) = (\alpha/q_i)^{j-1}.
\]
Averaging over $c_v$ uniform in $L(v)$ and noting that the product of indicators vanishes when $c_v \neq c$,
\begin{align}
\E{\left[{\left(\prod_{x \in e-v}\I[x \in A,\, c_x = c_v]\right)} {\left(\prod_{w \in g}\I[c_w = c]\right)}\right]} 
= \pr(c_v = c) \cdot(\alpha/q_i)^{j-1} \le (\alpha/q_i)^{j-1}.
\end{align}

When $j = \ell$, we have $g = \varnothing$ and the second product is empty. Since the trials at the vertices in $e - v$ are independent of $c_v$ and mutually independent of each other, averaging over $c_v$ uniform in $L(v)$, we have
\begin{align}
\E{\left[\prod_{x \in e-v}\I[x \in A,\, c_x = c_v]\right]} 
&= \sum_{c'' \in L(v)} \pr(c_v=c'') \, \prod_{x \in e-v} \pr(x \in A) \, \pr(c_x = c'') \\
&\le \frac{1}{|L(v)|} \sum_{c'' \in L(v)} (\alpha/q_i)^{j-1} 
= (\alpha/q_i)^{j-1}.
\end{align}
As a result, we get~\eqref{eq: prob Y ell j 2} for all cases, concluding the proof of Lemma~\ref{lemma: Y 1 Y 2 Z 1 expectations}~\ref{item: Y 2}.

\subsubsection{Upper bound for $\E[\mathbf{Y}^{(1)}_{\ell,j} - \mathbf{Z}^{(1)}_{\ell, j}]$}

For each $u, w \in U$, both distinct from $v$, define
\[
\cS(u) \coloneqq \I[u \in U^*] \cdot \I[c \in L^*(u)], \qquad \cR(w) \coloneqq \I[w \in A, \, c_w=c] \cdot \I[c \in L^*(w)].
\]
By linearity of expectation and~\eqref{def: Y ell j 1} and~\eqref{def: Z ell j 1}, we have
\begin{align}\label{eq: linear expectation Y ell j 1 - Z ell j 1}
\E{\left[\mathbf{Y}^{(1)}_{\ell, j} - \mathbf{Z}^{(1)}_{\ell, j}\right]} 
= \sum_{e \in E(F_c,j,v)}\sum_{f \in \binom{e-v}{\ell-1}}
\E{\left[{\left(\prod_{u \in f}\cS(u)\right)} {\left(\prod_{w \in e\setminus (f\cup \{v\})}\cR(w)\right)}\right]}.
\end{align}
Note that it suffices to show that the following holds for each $e \in E(F_c, j, v)$ and $f \in \tbinom{e-v}{\ell-1}$:
\begin{align}\label{eq: prob Y ell j 1 - Z ell j 1}
\E{\left[{\left(\prod_{u \in f}\cS(u)\right)} {\left(\prod_{w \in e\setminus (f\cup \{v\})}\cR(w)\right)}\right]} \le \uncolor_i^{\ell-1} [(1+5/q_i)\,\keep_i]^{j-1}(\alpha/q_i)^{j-\ell}.
\end{align}
Indeed, by substituting~\eqref{eq: prob Y ell j 1 - Z ell j 1} into~\eqref{eq: linear expectation Y ell j 1 - Z ell j 1} and applying~\ref{Q3}, we obtain
\begin{align*}
\E{\left[\mathbf{Y}^{(1)}_{\ell, j} - \mathbf{Z}^{(1)}_{\ell, j}\right]}
&\le \uncolor_i^{\ell-1} [(1+5/q_i)\,\keep_i]^{j-1} (\alpha/q_i)^{j-\ell} \binom{j-1}{\ell-1} \deg(F_c,j,v) \\
&\overset{\ref{Q3}}{\le} \uncolor_i^{\ell-1} [(1+5/q_i)\,\keep_i]^{j-1} (\alpha/q_i)^{j-\ell} \binom{j-1}{\ell-1} \binom{k-1}{j-1}(\alpha\gamma_i/q_i)^{k-j}\beta_i^{j-1}t_i \\
&= \dbinom{k-1}{\ell-1} \dbinom{k-\ell}{j-\ell} (\alpha/q_i)^{k-\ell} \uncolor_i^{\ell-1}[(1+5/q_i) \, \beta_i \, \keep_i]^{j-1} \gamma_i^{k-j} t_i,
\end{align*}
as desired.

Our goal now is to prove~\eqref{eq: prob Y ell j 1 - Z ell j 1} for a fixed edge $e \in E(F_c, j, v)$ and a fixed subset $f \in \binom{e-v}{\ell-1}$. Set $g \coloneqq e \setminus (f \cup \{v\})$  (see Fig.~\ref{fig: edge shrinking}).
Our strategy is first to establish that $\{\cS(u)\,:\,u \in f\} \cup \{\cR(w)\,:\,w \in g\}$
is a family of mutually independent random variables, and then to compute $\E[\cS(u)]$ and $\E[\cR(w)]$ for each $u \in f$ and $w \in g$. We do this via a sequence of claims.

\begin{claim}\label{claim: mutual independence}
    $\{\cS(u)\,:\,u \in f\} \cup \{\cR(w)\,:\,w \in g\}$
    is a family of mutually independent random variables.
\end{claim}
 
\begin{claimproof}
    We identify precisely which underlying random trials determine each of the random variables. The random trials in Algorithm~\ref{algorithm: nibble step} at a vertex $y \in U$ consist of the inclusion of $y$ in $A$, the color choice $c_y$, and the collection of equalizing coin flips $\{\eq(y, c')\,:\,c' \in L(y)\}$. All trials at distinct vertices are independent, and the equalizing coin flips are mutually independent of each other and of all other trials.

    Each $\cS(u)$ is a function of the trials at the vertices in $\{u\} \cup \cN^*(u)$,
    while each $\cR(w)$ is a function of the trials at the vertices in $\{w\} \cup N_c^*(w) \subseteq \{w\} \cup \cN^*(w)$.
    Since the equalizing coin flips are independent of each other and of all activations and color choices, it suffices to check that 
    \[
    \{\{x\} \cup \cN^*(x) \,:\, x \in f\cup g\}
    \]
    is a collection of pairwise disjoint sets.
    However, this is an immediate consequence of Observation~\ref{obs: disjoint N*}, applied to every pair of colors appearing in the unions defining $\cN^*(x)$.

    Combining these observations with the mutual independence of the equalizing coin flips, we conclude that $\{\cS(u)\,:\,u \in f\} \cup \{\cR(w)\,:\,w \in g\}$ is a family of mutually independent random variables.
\end{claimproof}

We set $p_e \coloneqq (\alpha/q_i)^{j-1}$. Since $q_i \gg d^{\frac{\eta}{k-1}}$ by Lemma~\ref{lemma: properties of beta_i, gamma_i, keep_i, uncolor_i}~\ref{q_i bound} and $d$ is sufficiently large, we have 
\begin{align}\label{eq: alpha/q_i bound}
\frac{\alpha}{q_i} \le \frac{1}{q_i} \ll d^{-\frac{\eta}{k-1}} \ll 1,
\end{align}
and so
\begin{align}\label{eq: p_e bound}
\frac{1}{1 - p_e} = 1 + \frac{p_e}{1 - p_e} \le 1 + \frac{(\alpha/q_i)^{j-1}}{1-1/2} \le 1+\frac{2}{q_i}.
\end{align}
With this in hand, we compute upper bounds for $\E[\cR(w)]$ for $w \in g$ and $\E[\cS(u)]$ for $u \in f$. 

\begin{claim}\label{claim: E R w}
Fix $w \in g$. Then, $\E[\cR(w)] \le (1 + 5/q_i) (\alpha \keep_i /q_i)$.
\end{claim}

\begin{claimproof}
Since $\I[w \in A]$, $c_w$, and $\eq(w, c)$ are independent, and $\I[c \notin D^*(w)]$ depends only on the inclusion in $A$ and the color choice at vertices other than $w$,
\[
\E[\cR(w)] = \pr(w \in A)\,\pr(c_w = c)\,\pr(c \notin D^*(w))\,\pr(\eq(w, c) = 1).
\]
For each edge $e' \in E(F_c, w)$, let $\cW_{e'}$ denote the event that $c$ is discarded at $w$ via $e'$; that is, every $x \in e'-w$ satisfies that $x \in A$ and $c_x = c$ (see Step~\ref{step: discard} of Algorithm~\ref{algorithm: nibble step}). 
Since $F_c$ is linear, $\{e' - w\,:\,e' \in E(F_c, w)\}$ is a collection of pairwise disjoint sets, and so the family $\{\cW_{e'}\,:\,e' \in E(F_c, w)\}$ is mutually independent. 
Recall that $c$ is included in $D^*(w)$ if it is discarded at $w$ via an edge $e' \in E(F_c, w) \setminus E(F_c, v)$, and $e$ is the unique edge in $E(F_c, v)$ containing $w$ by the linearity of $F_c$. Together with the mutual independence of $\{\cW_{e'}\,:\,e' \in E(F_c, w)\}$, we get
\begin{align}
\pr(c \notin D(w)) = \pr(c \notin D^*(w))\, \pr(\overline{\cW_e}) = \pr(c \notin D^*(w))\, (1 - p_e), 
\end{align}
and so
\begin{align}\label{eq: c not in D star w}
\pr(c \notin D^*(w))\,\pr(\eq(w, c) = 1) = \frac{\pr(c \notin D^*(w))\,\keep_i}{\pr(c \notin D(w))} \overset{\eqref{def: eq(v,c)}}{=} \frac{\keep_i}{1 - p_e}.
\end{align}
Consequently, by~\eqref{eq: p_e bound}, we have
\begin{align}\label{eq: E R w}
\E[\cR(w)] = \frac{\alpha \keep_i}{q_i(1 - p_e)} \le \left(1 + \frac{2}{q_i}\right)\frac{\alpha \keep_i}{q_i} \le \left(1 + \frac{5}{q_i}\right)\frac{\alpha \keep_i}{q_i},
\end{align}
as claimed.
\end{claimproof}

\begin{claim}\label{claim: E T u}
Fix $u \in f$. Then, $\E[\cS(u)] \le (1 + 5/q_i) (\keep_i \uncolor_i)$.    
\end{claim}

\begin{claimproof}
Since $\I[u \in A]$ is independent of $c_u$, of $\eq(u, c_u)$, of $\eq(u, c)$, and of $\I[c \notin D^*(u)]$,
\begin{align}
\E[\cS(u)] &= \E[\I[c \in L^*(u)](1 - \I[u \notin U^*])] \\
&= \pr(c \in L^*(u)) - \pr(u \in A)\,\pr((c_u \in L^*(u)) \wedge (c \in L^*(u))) \\
&\overset{\eqref{eq: c not in D star w}}{=} \frac{\keep_i}{1 - p_e} - \alpha\,\pr((c_u \in L^*(u)) \wedge (c \in L^*(u))), \label{eq: E T u expansion}
\end{align}
where~\eqref{eq: c not in D star w} is applied at $u$ in place of $w$.
We bound $\pr((c_u \in L^*(u))\wedge(c \in L^*(u)))$ by conditioning on $c_u$. 

When $c_u = c$, the two events coincide, and 
\begin{align}\label{eq: c u equals c contribution}
\pr((c_u \in L^*(u))\wedge(c \in L^*(u)) \mid c_u = c) = \pr(c \in L^*(u)) = \frac{\keep_i}{1 - p_e}.
\end{align}

When $c_u = c'$ for some $c' \ne c$, the equalizing coin flips $\eq(u, c')$ and $\eq(u, c)$ are independent of each other and of $\{(c' \notin D^*(u)) \wedge (c \notin D^*(u))\}$, and so
\begin{multline*}
\pr((c_u \in L^*(u))\wedge(c \in L^*(u)) \mid c_u = c') \\
= \pr((c' \notin D^*(u)) \wedge (c \notin D^*(u))) \, \pr(\eq(u, c') = 1) \, \pr(\eq(u, c) = 1).
\end{multline*}
By applying~\eqref{def: eq(v,c)} to $\eq(u, c')$ and $\eq(u, c)$, this gives \begin{align}\label{eq: c u equals c prime contribution}
\pr((c_u \in L^*(u))\wedge(c \in L^*(u)) \mid c_u = c') = \keep_i^2 \cdot \frac{\pr((c' \notin D^*(u)) \wedge (c \notin D^*(u)))}{\pr(c' \notin D(u)) \, \pr(c \notin D(u))}.
\end{align}
To estimate the joint probability $\pr((c' \notin D^*(u)) \wedge (c \notin D^*(u)))$, we need to establish several observations.
For each color $c'' \in L(u)$ and each edge $e' \in E(F_{c''}, u)$, let $\cW_{e', c''}$ denote the event that $c''$ is discarded at $u$ via $e'$; that is, every $x \in e'-u$ satisfies that $x \in A$ and $c_x = c''$ (see Step~\ref{step: discard} of Algorithm~\ref{algorithm: nibble step}).
Set $\cE' \coloneqq E(F_{c'}, u) \setminus E(F_{c'}, v)$ and $\cE \coloneqq E(F_c, u) \setminus E(F_c, v)$. By the definition of $D^*(u)$, we note that
\[
\{c \notin D^*(u)\} = \bigwedge_{e' \in \cE} \overline{\cW_{e',c}} \qquad \text{and} \qquad \{c' \notin D^*(u)\} = \bigwedge_{e' \in \cE'} \overline{\cW_{e',c'}}.
\]
Since $F_{c'}$ and $F_c$ are linear, $\{\cW_{e', c'}\,:\,e' \in \cE'\}$ and $\{\cW_{e', c}\,:\,e' \in \cE\}$ are families of mutually independent events. Now, let $h' \in \cE'$ and $h \in \cE$. If $h' \neq h$, then $|h' \cap h| \le 1$ by condition~\ref{Q1} of property $P(i)$. This implies that $h' - u$ and $h - u$ are disjoint. Together with the linearity of $F_{c'}$ and $F_c$, it follows that the sets $e' - u$ with $e' \in \cE' \cup \cE$ are pairwise disjoint; moreover, each of the events $\cW_{e', c'}$ with $e' \in \cE'$ and $\cW_{e', c}$ with $e' \in \cE$ is determined by the trials at the vertices in $e' - u$. If $h' = h$, then $h \in \cE' \cap \cE$, and thus by condition~\ref{Q1}, $|h| = k$ and $v \notin h$. Since $\cW_{h, c'}$ and $\cW_{h, c}$ require $c_x = c'$ and $c_x = c$ for every $x \in h - u$, respectively, the two events are mutually exclusive. This implies that
\[
\pr(\overline{\cW_{h, c'}} \wedge \overline{\cW_{h, c}}) = 1 - \pr(\cW_{h, c'}) - \pr(\cW_{h, c}) = 1 - 2(\alpha/q_i)^{k-1},
\]
and so
\[
\frac{\pr(\overline{\cW_{h, c'}} \wedge \overline{\cW_{h, c}})}{\pr(\overline{\cW_{h, c'}}) \, \pr(\overline{\cW_{h, c}})} = \frac{1 - 2(\alpha/q_i)^{k-1}}{(1 - (\alpha/q_i)^{k-1})^2} = 1 - \frac{(\alpha/q_i)^{2(k-1)}}{(1 - (\alpha/q_i)^{k-1})^2}  
\ge 1 - 4(\alpha/q_i)^k,
\]
where in the last inequality we used  $(\alpha/q_i)^{k-1} \le \alpha/q_i \ll 1$, by~\eqref{eq: alpha/q_i bound} and $k \ge 2$.

Combining these observations, we compute the joint probability:
\begin{align*}
\pr((c' \notin D^*(u))\wedge (c \notin D^*(u))) 
&= \pr(c' \notin D^*(u))\,\pr(c \notin D^*(u))\prod_{h \in \cE' \cap \cE} \frac{\pr(\overline{\cW_{h, c'}} \wedge \overline{\cW_{h, c}})}{\pr(\overline{\cW_{h, c'}}) \, \pr(\overline{\cW_{h, c}})}\\
&\ge \pr(c' \notin D^*(u))\,\pr(c \notin D^*(u)) \left(1 - 4 \cdot |\cE' \cap \cE| \cdot (\alpha/q_i)^k\right).
\end{align*}
Since $\cE' \cap \cE \subseteq E(F_c, u)$ and every edge $h \in \cE' \cap \cE$ satisfies $|h| = k$, we have
\[
|\cE' \cap \cE| \cdot (\alpha/q_i)^{k-1} \le \deg(F_c,k,u) (\alpha/q_i)^{k-1} \le 1 - \keep_i \overset{\ref{keep_i bound}}{\le} \eta,
\]
where the second inequality follows from Lemma~\ref{lemma: eq well defined}, and the last inequality follows from Lemma~\ref{lemma: properties of beta_i, gamma_i, keep_i, uncolor_i}~\ref{keep_i bound}. Consequently,
\begin{align}\label{eq: joint witness LB}
\pr((c' \notin D^*(u))\wedge (c \notin D^*(u))) \ge (1 - 4\alpha\eta/q_i)\,\pr(c' \notin D^*(u))\,\pr(c \notin D^*(u)).
\end{align}
Combining~\eqref{eq: c u equals c prime contribution} and~\eqref{eq: joint witness LB}, we obtain
\begin{align*}
\pr((c_u \in L^*(u))\wedge(c \in L^*(u)) \mid c_u = c') 
&\ge (1 - 4\alpha\eta/q_i)\, \keep_i^2 \cdot \frac{\pr(c' \notin D^*(u))}{\pr(c' \notin D(u))} \cdot \frac{\pr(c \notin D^*(u))}{\pr(c \notin D(u))} \\
&\ge (1 - 4\alpha\eta/q_i)\, \frac{\keep_i^2}{1 - p_e},
\end{align*}
noting that $\pr(c' \notin D^*(u)) \ge \pr(c' \notin D(u))$ and that $\pr(c \notin D(u)) = (1 - p_e) \, \pr(c \notin D^*(u))$ by~\eqref{eq: c not in D star w} applied at $u$ in place of $w$.
Averaging~\eqref{eq: c u equals c contribution} and the above over $c_u$ uniform in $L(u)$ gives
\begin{align*}
\pr((c_u \in L^*(u)) \wedge (c \in L^*(u))) 
&\ge \frac{1}{q_i}\cdot\frac{\keep_i}{1 - p_e} + \frac{q_i - 1}{q_i}\left(1 - \frac{4\alpha\eta}{q_i}\right)\frac{\keep_i^2}{1 - p_e} \\
&\ge \frac{\keep_i^2}{1 - p_e}\left(\frac{1}{q_i} + \frac{q_i - 1}{q_i}\left(1 - \frac{4\alpha\eta}{q_i}\right)\right) \\
&= \frac{\keep_i^2}{1 - p_e}\left(1 - \frac{q_i - 1}{q_i} \cdot \frac{4\alpha\eta}{q_i}\right) \ge \frac{\keep_i^2}{1 - p_e}\left(1 - \frac{4\alpha\eta}{q_i}\right),
\end{align*}
where the second inequality uses $\keep_i \le 1$ by Lemma~\ref{lemma: properties of beta_i, gamma_i, keep_i, uncolor_i}~\ref{keep_i bound}. Substituting into~\eqref{eq: E T u expansion},
\begin{align*}
\E[\cS(u)] 
&\le \frac{\keep_i}{1 - p_e} - \alpha\,\frac{\keep_i^2}{1 - p_e}\left(1 - \frac{4\alpha\eta}{q_i}\right) \\
&= \frac{\keep_i(1 - \alpha \keep_i)}{1 - p_e} + \frac{4\alpha^2\eta\,\keep_i^2}{(1 - p_e)\,q_i} \\
&\overset{\eqref{def: uncolor_i, beta_i, gamma_i}}{=} \left(1 + \frac{4\alpha^2\eta\,\keep_i}{\uncolor_i \, q_i}\right) \frac{\keep_i\uncolor_i}{1 - p_e} \\
&\le \left(1 + \frac{2}{q_i}\right) \frac{\keep_i\uncolor_i}{1 - p_e},
\end{align*}
where in the last inequality we used $\alpha \le 1$, $\keep_i \le 1$, $\eta < 1/6$ from~\eqref{def: eta, alpha, delta, s}, and $\uncolor_i \ge 1 - \alpha \ge 1 - \eta$ by Lemma~\ref{lemma: properties of beta_i, gamma_i, keep_i, uncolor_i}~\ref{uncolor_i bound}.
Using~\eqref{eq: p_e bound}, we get 
\[
\E[\cS(u)] \le
\left(1 + \frac{2}{q_i}\right)^2 \keep_i \uncolor_i \le 
\left(1 + \frac{5}{q_i}\right) \keep_i \uncolor_i,
\]
where in the last inequality we used that $1/q_i \ll 1/4$ by~\eqref{eq: alpha/q_i bound}.
This completes the proof of the claim.
\end{claimproof}

By the mutual independence given by Claim~\ref{claim: mutual independence},
\begin{align}\label{eq: factorized expectation}
\E\left[\prod_{u \in f}\cS(u)\prod_{w \in g}\cR(w)\right] = \prod_{u \in f}\E[\cS(u)]\prod_{w \in g}\E[\cR(w)].
\end{align}
Combining~\eqref{eq: factorized expectation} with Claims~\ref{claim: E R w} and~\ref{claim: E T u}, and noting that $|f| = \ell - 1$ and $|g| = j-\ell$,
\begin{align}
\E{\left[\prod_{u \in f}\cS(u)\prod_{w \in g}\cR(w)\right]} 
&\le (1+5/q_i)^{j-1} (\keep_i \uncolor_i)^{\ell-1} (\alpha \keep_i /q_i)^{j-\ell} \\
&= \uncolor_i^{\ell-1} [(1+5/q_i)\,\keep_i]^{j-1}(\alpha/q_i)^{j-\ell},
\end{align}
proving~\eqref{eq: prob Y ell j 1 - Z ell j 1}, which in turn completes the proof of Lemma~\ref{lemma: Y 1 Y 2 Z 1 expectations}~\ref{item: Y 1 - Z 1}.

\subsection{Concentration arguments: proof of Lemma~\ref{lemma: Y 1 Y 2 Z 1 concentration}}\label{subsection: concentration}
\stepcounter{ForClaims}
\renewcommand{\theForClaims}{\ref*{lemma: Y 1 Y 2 Z 1 concentration}}

Once again, we will split this section into three subsections devoted to the proofs of each item in the statement of Lemma~\ref{lemma: Y 1 Y 2 Z 1 concentration}.

\subsubsection{Concentration of $\mathbf{Y}^{(1)}_{\ell, j}$}

Fix $j$ with $\ell \le j \le k$ and $\deg(F_c, j, v) > 0$, and recall the definition of $\tau_j$ from~\eqref{def: tau}.
By~\ref{Q3}, either $j = k$ or $\gamma_i > 0$; in the latter case, $i \ge 2$ and hence $\gamma_i \ge 1$ by Lemma~\ref{lemma: properties of beta_i, gamma_i, keep_i, uncolor_i}~\ref{gamma_i bound}.
Thus, $\gamma_i^{k - j} \ge 1$ in all cases.
Then, by Lemma~\ref{lemma: Y 1 Y 2 Z 1 expectations}~\ref{item: Y 1 Z 1}, whenever the denominators are positive, we have
\begin{align}\label{eq: lb 1 tau}
\frac{\tau_j}{\E{\left[\mathbf{Z}^{(1)}_{\ell,j}\right]}} \ge \frac{\tau_j}{\E{\left[\mathbf{Y}^{(1)}_{\ell,j}\right]}} \ge \frac{\delta^2}{\uncolor_i^{\ell-1}} \overset{\eqref{def: uncolor_i, beta_i, gamma_i}}{\ge} \delta^2.
\end{align}
Moreover, by using Lemma~\ref{lemma: properties of beta_i, gamma_i, keep_i, uncolor_i}~\ref{beta_i bound},~\ref{q_i bound}, and~\ref{ratio bound} together with $2 \le \ell \le j \le k$, we have
\begin{align}
\tau_j &\overset{\ref{gamma_i bound}}{\ge} \delta^2 \cdot (\alpha/q_i)^{k-2} \beta_i^{k-1} t_i \overset{\ref{beta_i bound}}{\ge} \delta^2 \cdot \alpha^{k-2} \cdot (\log d)^{-2(k-1)} \cdot q_i \cdot \frac{t_i}{q_i^{k-1}}\\
&\overset{\eqref{def: eta, alpha, delta, s}}{\ge} d^{-\frac{\eta/3}{k-1}} \cdot \frac{\eta^{k-2}}{(\log d)^{3k-4}} \cdot q_i \cdot \frac{t_i}{q_i^{k-1}} \overset{\ref{q_i bound}}{\ge} d^{\frac{2\eta/3}{k-1}} \cdot \frac{\eta^{k-2}}{(\log d)^{3k-4}} \cdot \frac{t_i}{q_i^{k-1}} \overset{\ref{ratio bound}}{=}  \Omega{\left(d^{\frac{\eta/2}{k-1}}\right)}.\label{eq: tau lower bound}
\end{align}
Together with~\eqref{eq: lb 1 tau} this implies, whenever the denominators are positive, that
\begin{align}\label{eq: lb 2 tau}
\frac{\tau_j^2}{\E{\left[\mathbf{Z}_{\ell, j}^{(1)}\right]}} \ge \frac{\tau_j^2}{\E{\left[\mathbf{Y}_{\ell, j}^{(1)}\right]}} \ge \delta^2\tau_j = \Omega{\left(d^{\frac{\eta/6}{k-1}}\right)}.
\end{align}
With the above inequalities in hand, let us proceed with the concentration argument.

We use \hyperref[theorem: Talagrand]{Talagrand's Inequality} (Theorem~\ref{theorem: Talagrand}) to establish concentration for $\mathbf{Y}_{\ell, j}^{(1)}$, where the random trials are, for all vertices $u \in U \setminus \{v\}$, the inclusion of $u$ in $A$, the color choice $c_u$, and the equalizing coin flips for all colors $c' \in L(u)$ at $u$. To apply Talagrand's Inequality, we must first show that $\mathbf{Y}_{\ell, j}^{(1)}$ satisfies the Lipschitz and certificate conditions outlined in Definition~\ref{def: lipshcitz and certifiable}.
We do this via a sequence of claims as follows.

\begin{claim}\label{claim: Lipschitz Y ell j 1}
$\mathbf{Y}_{\ell, j}^{(1)}$ is $\left[\binom{k-1}{\ell-1} + \binom{k-2}{\ell-2}\right]$-Lipschitz.
\end{claim}

\begin{claimproof}
Fix a vertex $z \in U \setminus \{v\}$. Changing the outcome of the random trial at $z$ (that is, the inclusion of $z$ in $A$, the color choice $c_z$, or one of the equalizing coin flips $\eq(z, c')$ for $c' \in L(z)$) affects $\mathbf{Y}_{\ell, j}^{(1)}$ only through terms indexed by pairs $(e, f)$ such that either of the following holds:
\begin{enumerate}[label=\textup{(\alph*)}]
    \item \label{case: Y direct} $z \in e - v$;
    \item \label{case: Y witness} $z \notin e - v$ but $z \in e' - u$ for some $u \in f$ and some $e' \in E(F_{c_u}, u) \setminus E(F_{c_u}, v)$.
\end{enumerate}

If $z \in e - v$, then since $F_c$ is linear (as a result of condition~\ref{Q1} of property $P(i)$), there is at most one edge $e \in E(F_c, j, v)$ containing $z$.
Therefore, the number of affected pairs $(e, f)$ is at most $\binom{j - 1}{\ell - 1} \le \binom{k - 1}{\ell - 1}$.

Next, suppose that $z \in e' - u$ for some $u \in f$ and some $e' \in E(F_{c_u}, u) \setminus E(F_{c_u}, v)$. Then, we have $z, u \in e'$ and $u, v \in e$, so $u$ is a common neighbor of $z$ and $v$ in $\bigcup_{c'' \in L(U)}F_{c''}$.
Since $\bigcup_{c'' \in L(U)}F_{c''}$ is uncrowded by condition~\ref{Q1} of property $P(i)$, the vertices $z$ and $v$ have at most one such common neighbor $u$. 
Moreover, since $F_c$ is linear, the edge $e \in E(F_c, j, v)$ containing $u$ and $v$ is unique. As a result,
there are at most $\binom{j - 2}{\ell - 2} \le \binom{k - 2}{\ell - 2}$ affected pairs $(e, f)$.

Combining the two cases, changing the trial at $z$ can affect at most $\binom{k-1}{\ell-1} + \binom{k-2}{\ell-2}$ pairs $(e, f)$, and so can change $\mathbf{Y}_{\ell, j}^{(1)}$ by at most $\binom{k-1}{\ell-1} + \binom{k-2}{\ell-2}$, as claimed.
\end{claimproof}

\begin{claim}\label{claim: certifiable Y ell j 1}
$\mathbf{Y}_{\ell, j}^{(1)}$ is $2k^2$-certifiable.
\end{claim}

\begin{claimproof}
Each pair $(e, f)$ with $e\in E(F_c,j,v)$ and $f\in \tbinom{e-v}{\ell-1}$ contributing to $\mathbf{Y}_{\ell, j}^{(1)}$ can be certified by certifying that for each $u \in f$ and each $w \in e \setminus (f \cup \{v\})$,
\begin{enumerate}[(a)]
\item \label{cond: certi Y1 1} $u \in U^*$, and
\item \label{cond: certi Y1 2} $w \in A$ and $c_w = c$.
\end{enumerate}
Condition~\ref{cond: certi Y1 1} can be certified by revealing one of the following:
\begin{itemize}
\item $u$ is not included in $A$ (requiring $1$ trial);
\item $c_u \in D^*(u)$, which is certified by revealing $c_u$ and, for all vertices $x\in e' - u$ in some edge $e' \in E(F_{c_u}, u) \setminus E(F_{c_u}, v)$, revealing that $x\in A$ and $c_x = c_u$ (requiring at most $2|e'| - 1 \le 2k -1$ trials); or
\item $\eq(u, c_u) = 0$, which can be certified by revealing $c_u$ and the corresponding equalizing coin flip (requiring $2$ trials).
\end{itemize}
Condition~\ref{cond: certi Y1 2} can be certified by revealing the trials $\I[w\in A]$ and $c_w$.
Thus, we need at most $(\ell-1) \cdot (2k-1) + (j-\ell) \cdot 2 \le 2k^2$ trials to certify each pair $(e, f)$ contributing to $\mathbf{Y}_{\ell,j}^{(1)}$.
\end{claimproof}
 
If $\E[\mathbf{Y}_{\ell, j}^{(1)}] = 0$, then $\mathbf{Y}_{\ell, j}^{(1)} = 0$ and the desired bound is immediate. 
Otherwise, with Claims~\ref{claim: Lipschitz Y ell j 1} and~\ref{claim: certifiable Y ell j 1} in hand, and writing $\mu_Y \coloneqq \binom{k-1}{\ell-1} + \binom{k-2}{\ell-2}$ for brevity, we apply \hyperref[theorem: Talagrand]{Talagrand's Inequality} with $\mu = \mu_Y$ and $r = 2k^2$ to obtain:
\begin{align*}
\pr{\left(\mathbf{Y}_{\ell, j}^{(1)} - \E{\left[\mathbf{Y}^{(1)}_{\ell,j}\right]} \ge \tau_j/2\right)}
&\overset{\eqref{eq: lb 2 tau}}{\le} \pr{\left(\left|\mathbf{Y}_{\ell, j}^{(1)} - \E{\left[\mathbf{Y}^{(1)}_{\ell,j}\right]}\right| \ge \tau_j/4 + 20 \mu_Y \sqrt{2k^2 \E{\left[\mathbf{Y}^{(1)}_{\ell,j}\right]}} + 128 \mu_Y^2 k^2\right)} \\
&\ \le \ 4\exp\left(-\frac{\tau_j^2}{256 \mu_Y^2 k^2 \left(\E{\left[\mathbf{Y}^{(1)}_{\ell,j}\right]} + \tau_j/4\right)}\right) \\
&\overset{\eqref{eq: lb 1 tau}}{\le} 4\exp\left(-\frac{\delta^2 \tau_j}{256 \mu_Y^2 k^2 \left(1+ \delta^2/4\right)}\right) \\
&\overset{\eqref{eq: lb 2 tau}}{=} \exp\left(-\Omega{\left(d^{\frac{\eta/6}{k-1}}\right)}\right),
\end{align*}
completing the proof of Lemma~\ref{lemma: Y 1 Y 2 Z 1 concentration}~\ref{item: Y 1 conc}.

\subsubsection{Concentration of $\mathbf{Z}^{(1)}_{\ell, j}$}

We apply \hyperref[theorem: Talagrand]{Talagrand's Inequality} (Theorem~\ref{theorem: Talagrand}) to $\mathbf{Z}_{\ell, j}^{(1)}$, where the random trials are, for all vertices $u \in U \setminus \{v\}$, the inclusion of $u$ in $A$, the color choice $c_u$, and the equalizing coin flips for all colors $c' \in L(u)$. To apply Talagrand's Inequality, we must first show that $\mathbf{Z}_{\ell, j}^{(1)}$ satisfies the Lipschitz and certificate conditions outlined in Definition~\ref{def: lipshcitz and certifiable}.
We do so through a sequence of claims.

\begin{claim}\label{claim: Lipschitz Z ell j 1}
$\mathbf{Z}_{\ell, j}^{(1)}$ is $\left[2\binom{k-1}{\ell-1} + \binom{k-2}{\ell-2}\right]$-Lipschitz.
\end{claim}

\begin{claimproof}
Recall from~\eqref{def: Z ell j 1} that $\mathbf{Z}_{\ell, j}^{(1)}$ is obtained from $\mathbf{Y}_{\ell, j}^{(1)}$ in~\eqref{def: Y ell j 1} by multiplying each summand indexed by $(e, f)$ by the additional factor $\left(1 - \prod_{x \in e - v}\I[c \in L^*(x)]\right)$. Fix a vertex $z \in U \setminus \{v\}$. Changing the outcome of the random trial at $z$ (that is, the inclusion of $z$ in $A$, the color choice $c_z$, or one of the equalizing coin flips $\eq(z, c')$ for $c' \in L(z)$) affects $\mathbf{Z}_{\ell, j}^{(1)}$ only through terms indexed by pairs $(e, f)$ such that at least one of the following holds:
\begin{enumerate}[label=\textup{(\alph*)}]
\item\label{case: Z direct} $z \in e - v$; 
\item\label{case: Z u witness} $z \notin e - v$ but $z \in e' - u$ for some $u \in f$ and some $e' \in E(F_{c_u}, u) \setminus E(F_{c_u}, v)$;
\item\label{case: Z x witness} $z \notin e - v$ but $z \in e_1 - x$ for some $x \in e - v$ and some $e_1 \in E(F_c, x) \setminus E(F_c, v)$.
\end{enumerate}
We bound the number of affected pairs $(e, f)$ in each case.

An identical argument to case~\ref{case: Y direct}
in the proof of Claim~\ref{claim: Lipschitz Y ell j 1} 
shows that the number of affected pairs in case~\ref{case: Z direct} is at most $\binom{j - 1}{\ell - 1} \le \binom{k - 1}{\ell - 1}$.

Similarly, the same argument in case~\ref{case: Y witness}
in the proof of Claim~\ref{claim: Lipschitz Y ell j 1} shows that the number of affected pairs in case~\ref{case: Z u witness} is at most $\binom{j - 2}{\ell - 2} \le \binom{k - 2}{\ell - 2}$.

Finally, for case~\ref{case: Z x witness},
let us suppose that $z \notin e - v$ but $z \in e_1 - x$ for some $x \in e - v$ and some $e_1 \in E(F_c, x) \setminus E(F_c, v)$. Then, we have $z, x \in e_1$ and $x, v \in e$, so $x$ is a common neighbor of $z$ and $v$ in $\bigcup_{c'' \in L(U)}F_{c''}$.
Since $\bigcup_{c'' \in L(U)}F_{c''}$ is uncrowded by condition~\ref{Q1} of property $P(i)$, the vertices $z$ and $v$ have at most one such common neighbor $x$. 
Moreover, since $F_c$ is linear, the edge $e \in E(F_c, j, v)$ containing $x$ and $v$ is unique. 
As a result, there are at most $\binom{j - 1}{\ell - 1}$ affected pairs $(e, f)$.

Combining the three cases, changing the trial at $z$ can affect at most $2\binom{k-1}{\ell-1} + \binom{k-2}{\ell-2}$ pairs $(e, f)$, and so can change $\mathbf{Z}_{\ell, j}^{(1)}$ by at most $2\binom{k-1}{\ell-1} + \binom{k-2}{\ell-2}$ as claimed.
\end{claimproof}

\begin{claim}\label{claim: certifiable Z ell j 1}
$\mathbf{Z}_{\ell, j}^{(1)}$ is $2k^2$-certifiable.
\end{claim}

\begin{claimproof}
Each pair $(e, f)$ with $e\in E(F_c,j,v)$ and $f\in \tbinom{e-v}{\ell-1}$ contributing to $\mathbf{Z}_{\ell, j}^{(1)}$ can be certified by certifying that for each $u \in f$ and each $w \in e \setminus (f \cup \{v\})$,
\begin{enumerate}[(a)]
\item \label{cond: certi Z1 1} $u \in U^*$,
\item \label{cond: certi Z1 2} $w \in A$ and $c_w = c$, and
\item \label{cond: certi Z1 3} $c \not \in L^*(x)$ for some $x \in e-v$.
\end{enumerate}
As in the proof of Claim~\ref{claim: certifiable Y ell j 1}, condition~\ref{cond: certi Z1 1} can be certified by revealing one of the following:
\begin{itemize}
\item $u$ is not included in $A$ (requiring $1$ trial);
\item $c_u \in D^*(u)$, which is certified by revealing $c_u$ and, for all vertices $x\in e' - u$ in some edge $e' \in E(F_{c_u}, u) \setminus E(F_{c_u}, v)$, revealing that $x\in A$ and $c_x = c_u$ (requiring at most $2|e'| - 1 \le 2k -1$ trials); or
\item $\eq(u, c_u) = 0$, which can be certified by revealing $c_u$ and the corresponding equalizing coin flip (requiring $2$ trials).
\end{itemize}
Similarly, condition~\ref{cond: certi Z1 2} can be certified by revealing the trials $\I[w\in A]$ and $c_w$.
Finally, condition~\ref{cond: certi Z1 3} can be certified by one of the following:
\begin{itemize}
\item $c \in D^*(x)$ (requiring at most $2(k-1)$ trials); or
\item $\eq(x,c) = 0$ (requiring $1$ trial).
\end{itemize}
Thus, we need at most $(\ell-1) \cdot (2k-1) + (j-\ell) \cdot 2 + 2(k-1) 
\le 2k^2$ trials to certify each pair $(e, f)$ contributing to $\mathbf{Z}_{\ell,j}^{(1)}$.
\end{claimproof}
 
If $\E[\mathbf{Z}_{\ell, j}^{(1)}] = 0$, then $\mathbf{Z}_{\ell, j}^{(1)} = 0$ and the desired bound is immediate. 
Otherwise, with Claims~\ref{claim: Lipschitz Z ell j 1} and~\ref{claim: certifiable Z ell j 1} in hand, and writing $\mu_Z \coloneqq 2\binom{k-1}{\ell-1} + \binom{k-2}{\ell-2}$ for brevity, we apply \hyperref[theorem: Talagrand]{Talagrand's Inequality} with $\mu = \mu_Z$ and $r = 2k^2$ to obtain:
\begin{align*}
\pr{\left(\mathbf{Z}_{\ell, j}^{(1)} - \E{\left[\mathbf{Z}^{(1)}_{\ell,j}\right]} \le -\tau_j\right)}
&\overset{\eqref{eq: lb 2 tau}}{\le} \pr{\left(\left|\mathbf{Z}_{\ell, j}^{(1)} - \E{\left[\mathbf{Z}^{(1)}_{\ell,j}\right]}\right| \ge \tau_j/2 + 20 \mu_Z \sqrt{2k^2 \E{\left[\mathbf{Z}^{(1)}_{\ell,j}\right]}} + 128 \mu_Z^2 k^2\right)} \\
&\ \le \ 4\exp\left(-\frac{\tau_j^2}{64 \mu_Z^2 k^2 \left(\E{\left[\mathbf{Z}^{(1)}_{\ell,j}\right]} + \tau_j/2\right)}\right) \\
&\overset{\eqref{eq: lb 1 tau}}{\le} 4\exp\left(-\frac{\delta^2\tau_j}{64 \mu_Z^2 k^2 \left(1 + \delta^2/2\right)}\right) \\
&\overset{\eqref{eq: lb 2 tau}}{=} \exp\left(-\Omega{\left(d^{\frac{\eta/6}{k-1}}\right)}\right),
\end{align*}
completing the proof of Lemma~\ref{lemma: Y 1 Y 2 Z 1 concentration}~\ref{item: Z 1 conc}.

\subsubsection{Concentration of $\mathbf{Y}^{(2)}_{\ell, j}$}

We claim that, conditional on any fixed value $c' \in L(v)$ of $c_v$, $\mathbf{Y}^{(2)}_{\ell, j}$ is a constant multiple of a sum of independent Bernoulli random variables. Indeed, when $j > \ell$, for any $f \in \binom{e-v}{\ell-1}$ and any $w \in e \setminus (f \cup \{v\})$, the summand in~\eqref{def: Y ell j 2} requires both $c_w = c_v$ and $c_w = c$, and hence vanishes unless $c_v = c$; consequently,
\[
\left(\mathbf{Y}^{(2)}_{\ell, j} \,\big|\, c_v = c' \right) 
= \I[c' = c] \binom{j-1}{\ell-1}\sum_{e \in E(F_c, j, v)}\prod_{x \in e-v}\I[x \in A,\, c_x = c].
\]
When $j = \ell$, the second product in~\eqref{def: Y ell j 2} is empty, and 
\[
\left(\mathbf{Y}^{(2)}_{\ell, \ell} \,\big|\, c_v = c' \right)
= \sum_{e \in E(F_c, \ell, v)} \prod_{x \in e-v}\I[x \in A,\, c_x = c'].
\]
Define
\[
\mathbf{B}_{c'} \coloneqq 
\begin{cases}
\displaystyle \I[c' = c] \sum_{e \in E(F_c, j, v)}\prod_{x \in e-v}\I[x \in A,\, c_x = c] & \text{if}~ j > \ell; \\
\displaystyle \sum_{e \in E(F_c, j, v)} \prod_{x \in e-v}\I[x \in A,\, c_x = c'] & \text{if}~ j = \ell.
\end{cases}
\]
Then, 
\begin{align}
\label{eq: Y j ell 2 and B c'}
\left(\mathbf{Y}^{(2)}_{\ell, j} \,\big|\, c_v = c' \right)  = \binom{j-1}{\ell-1} \mathbf{B}_{c'},
\end{align}
where the binomial coefficient equals $1$ when $j = \ell$. Since $F_c$ is linear by condition~\ref{Q1} of property $P(i)$, the sets in $\{e - v : e \in E(F_c, j, v)\}$ are pairwise disjoint. Hence, $\mathbf{B}_{c'}$ is a sum of independent Bernoulli random variables (and may be identically zero).

For every $c' \in L(v)$, condition~\ref{Q3} gives
\begin{align*}
\E[\mathbf{B}_{c'}] &\le (\alpha/q_i)^{j-1} 
\deg(F_c, j, v) \\
&\overset{\ref{Q3}}{\le} (\alpha/q_i)^{j-1} \binom{k-1}{j-1}(\alpha\gamma_i/q_i)^{k-j}\beta_i^{j-1}t_i \\
&=\frac{1}{\binom{j-1}{\ell-1}} \cdot \binom{k-1}{\ell-1}\binom{k-\ell}{j-\ell} (\alpha/q_i)^{k-1} \gamma_i^{k-j} \beta_i^{j-1}t_i \\
&\overset{\eqref{def: tau}}{=} \frac{1}{\binom{j-1}{\ell-1}} \cdot \frac{(\alpha/q_i)^{\ell - 1}}{\delta^2}\cdot\tau_j 
\le \frac{\delta^4 \tau_j}{\binom{j-1}{\ell-1}},
\end{align*}
where in the last inequality we used the bound $(\alpha/q_i)^{\ell-1} \le \alpha/q_i \ll d^{-\frac{\eta}{k-1}} = \delta^6$, which follows from~\eqref{eq: alpha/q_i bound} and the definition of $\delta$ in~\eqref{def: eta, alpha, delta, s}.
Setting
\[
\tau_{j, c'} \coloneqq \frac{\tau_j}{2\binom{j-1}{\ell-1}} - \E[\mathbf{B}_{c'}],
\]
it follows that, for sufficiently large $d$,
\begin{align}\label{eq: lb t}
\tau_{j, c'} \ge \frac{\tau_j/2 - \delta^4\tau_j}{\binom{j-1}{\ell-1}} \ge \frac{\tau_j/4}{\binom{j-1}{\ell-1}} \overset{\eqref{eq: tau lower bound}}{=} \Omega{\left(d^{\frac{\eta/2}{k-1}}\right)},
\end{align}
and so
\begin{align}\label{eq: compare E Y ell j 2 and t}
\E[\mathbf{B}_{c'}] \le \frac{\delta^4 \tau_j}{\binom{j-1}{\ell-1}} \le 4\delta^4\tau_{j, c'} \ll \tau_{j, c'}.
\end{align}
If $\mathbf{B}_{c'}$ is identically zero, the desired conditional probability is zero. Otherwise, applying the \hyperref[theorem: chernoff]{Chernoff Bound} (Theorem~\ref{theorem: chernoff}) to $\mathbf{B}_{c'}$ gives
\begin{align*}
\pr{\left(\mathbf{Y}_{\ell, j}^{(2)} \ge \tau_j/2 \mid c_v = c'\right)}
&\overset{\eqref{eq: Y j ell 2 and B c'}}{=} \pr{\left(\mathbf{B}_{c'} - \E[\mathbf{B}_{c'}] \ge \tau_{j, c'} \right)} \\
&\ \le \ \exp\left(-\frac{\tau_{j, c'}^2}{2\E[\mathbf{B}_{c'}] + 2\tau_{j, c'}/3}\right) \\
&\overset{\eqref{eq: compare E Y ell j 2 and t}}{\le} \exp\left(-\frac{\tau_{j, c'}}{2 + 2/3}\right) \overset{\eqref{eq: lb t}}{\le} \exp\left(-\Omega{\left(d^{\frac{\eta/2}{k-1}}\right)}\right).
\end{align*}
Averaging over $c_v$ uniform in $L(v)$ finally gives
\[
\pr{\left(\mathbf{Y}^{(2)}_{\ell, j} \ge \tau_j/2\right)} = \sum_{c' \in L(v)}\pr(c_v = c')\,\pr{\left(\mathbf{Y}^{(2)}_{\ell, j} \ge \tau_j/2 \mid c_v = c'\right)}
\le \exp{\left(-\Omega{\left(d^{\frac{\eta/2}{k-1}}\right)}\right)},
\]
completing the proof of Lemma~\ref{lemma: Y 1 Y 2 Z 1 concentration}~\ref{item: Y 2 conc}.

\section{Coloring linear hypergraphs: proof of Corollary~\ref{corollary: linear}}\label{section: linear reduction}

In this section, we will prove Corollary~\ref{corollary: linear} which bounds the chromatic number of linear hypergraphs.
For the reader's convenience, we restate it here.

\begin{corollary*}[Restatement of Corollary~\ref{corollary: linear}]
Let $k \ge 3$ be an integer, and let $\eps > 0$ be arbitrary. There exists $\Delta_0\in\N$ such that every $k$-uniform linear hypergraph $\cH$ of maximum degree at most $\Delta \ge \Delta_0$ satisfies
\[
  \chi(\cH) \le (1+\eps){
  \left(\frac{(4k-5)(k-1)}{k-2}\,\frac{\Delta}{\log\Delta}\right)^{\frac{1}{k-1}}}.
\]
\end{corollary*}

We follow the approach of Frieze and Mubayi~\cite{frieze2013coloring} with several essential modifications and improvements to their argument; see Section~\ref{subsection: linear overview} for an overview of our approach. 
We begin with some terminology. Let $H$ be a $k$-uniform hypergraph, and let $r \ge 3$ be an integer.
A \emph{loose cycle} of length $r$ (or \emph{loose $r$-cycle}) in $H$ is a $k$-uniform subhypergraph of $H$ with $r(k-1)$ vertices and distinct edges $e_1, \dots, e_r$ that can be ordered cyclically so that consecutive edges meet at exactly one vertex and all other pairs of edges are disjoint. 
Thus, if $e_i \cap e_{i+1} = \{v_i\}$, with indices taken modulo $r$, then the vertices $v_1,\dots,v_r$ are distinct.
A loose $r$-cycle in $H$ contains no shorter cycle as a subhypergraph. In particular, every $3$-cycle in a linear hypergraph is loose. 
For a subset $S \subseteq V(H)$, we say that a cycle $F$ contains $S$ if $S \subseteq V(F)$. In particular, $F$ contains a vertex $v \in V(H)$ if $v \in V(F)$.

The next elementary counting lemma is a consequence of linearity.

\begin{lemma}\label{lemma: count loose cycles}
Let $k \ge 3$ be an integer. For every integer $r \ge 3$, there exists a constant $M_r > 0$ depending only on $k$ and $r$ such that the following holds. 
Let $\cH$ be a $k$-uniform linear hypergraph of maximum degree at most $\Delta$. For every nonempty subset $S \subseteq V(\cH)$, the number of loose $r$-cycles in $\cH$ containing $S$ is at most $M_r \Delta^{b_r(|S|)}$, where
\[
b_r(s)\coloneqq\max\left\{0,\,r-1-\left\lceil\frac{s-1}{k-1}\right\rceil\right\}.
\]
\end{lemma}

\begin{proof}
Fix a nonempty subset $S \subseteq V(\cH)$ and set $s \coloneqq |S|$.
Since each loose $r$-cycle contains $r(k-1)$ vertices,
if $s > r(k-1)$ then no loose $r$-cycle can contain $S$,
so we may assume $s \le r(k-1)$.
Let
\[
q \coloneqq \min\left\{r-1,\, \left\lceil\frac{s-1}{k-1}\right\rceil\right\},
\]
so that $b_r(s) = r - 1 - q$.
 
Consider a loose $r$-cycle $F$ containing $S$.
List its edges in cyclic order as $e_1,\dots,e_r$
(indices modulo $r$),
write $e_i\cap e_{i+1} = \{v_i\}$ for each~$i$, and set $f_i\coloneqq e_i - v_{i-1}$.
Then $f_1 \cup \cdots \cup f_r$ is a partition of $V(F)$ into pairwise disjoint blocks, each of size $k-1$.
We say that an index $i\in[r-1]$ is \emph{constrained} if either $i\ge 2$ and $f_i\cap S\ne\varnothing$, or $i=1$ and $|f_1\cap S|\ge 2$.

We claim that there is a cyclic ordering of the edges of $F$
and a vertex $z \in S$ so that at least $q$ indices are constrained.

If every block meets $S$, then $f_i \cap S \ne \varnothing$ for all $2 \le i \le r-1$, giving $r - 2$ constrained indices.
If $q \le r - 2$, choosing any $z \in f_1 \cap S$ suffices. 
If $q = r - 1$, then $s \ge (r-2)(k-1) + 2 = (r-2)(k-2) + r > r$, provided that $k, r \ge 3$, and so there must exist a block containing at least two vertices of $S$; rotating the cyclic order so that this block becomes $f_1$ and choosing $z \in f_1 \cap S$ yields the $(r-1)$-th constrained index.

If there exists at least one block disjoint from $S$, then, since $S$ is nonempty, there exist consecutive indices $j$ and $j+1$ (modulo $r$) with $f_j \cap S = \varnothing$ and $f_{j+1} \cap S \ne \varnothing$.
Rotate the cyclic order so that $e_{j+1}$ becomes $e_1$, and hence $e_j$ becomes $e_r$. Since each block $f_i = e_i - v_{i-1}$ is determined by $e_i$ and its predecessor, a rotation permutes the blocks without changing them; thus now $f_r \cap S = \varnothing$ and $f_1 \cap S \ne \varnothing$. Choose $z \in f_1 \cap S$.
Then $S \setminus \{z\}$ lies entirely in $f_1 \cup \cdots \cup f_{r-1}$, and since each block has $k-1$ vertices, it meets at least $\lceil(s-1)/(k-1)\rceil \ge q$ of these blocks. Each such block yields a constrained index: a block $f_i$ with $i \ge 2$ contributes directly, while $f_1$ meeting $S \setminus \{z\}$ means $|f_1 \cap S| \ge 2$, since $z \in f_1$.
 
Now, for each cycle $F$ containing $S$, fix such an ordering of the edges of $F$, a vertex $z\in f_1\cap S$, a set of $q$ constrained indices, and for each constrained index $i$ a witness $y_i\in f_i\cap S$ (with $y_1\ne z$ when $i=1$).
Since $s\le r(k-1)$, the total number of such preliminary choices is bounded by a constant depending only on $k$ and $r$.

Fix all these choices and build the cycle edge by edge as follows.
\begin{itemize}
\item Edge $e_1$: if index $1$ is constrained, then $e_1$ contains the distinct vertices $z$ and $y_1$, so linearity gives at most one choice; otherwise, $e_1$ is any edge containing $z$, giving at most $\Delta$ choices.
\item Edge $e_i$ for $2\le i\le r-1$: choose $v_{i-1}\in e_{i-1}$ (at most $k$ choices); if $i$ is constrained, then $e_i$ contains distinct vertices
$v_{i-1}$ and $y_i$ (since $y_i\in f_i$ and $v_{i-1}\notin f_i$), so linearity gives at most one choice; otherwise, at most $\Delta$ choices.
\item Edge $e_r$: choose $v_{r-1}\in e_{r-1}$ and $v_0\in e_1$ (at most $k^2$ pairs); since $v_{r-1}$ and $v_0$ are distinct, linearity gives at most one edge containing both.
\end{itemize}

Each of $e_1,\dots,e_{r-1}$ contributes at most a factor of $\Delta$, and each constrained index replaces one such factor by a constant.
With at least $q$ constrained indices, the count per fixed set of preliminary choices is at most $k^r\,\Delta^{r-1-q}$.
Multiplying by the constant number of preliminary choices gives the bound $M_r \Delta^{b_r(s)}$, completing the proof of the lemma.
\end{proof}

With the above lemma in hand, let us show that there exists a partition of the vertex set $V(\cH)$ of a linear hypergraph $\cH$ such that each part induces a subhypergraph satisfying certain local sparsity properties.

\begin{lemma}\label{lemma: partition 1}
Let $k \ge 3$ be an integer, and let $\eps, \eta \in (0,1)$ be arbitrary.
There exist a constant $C > 0$ depending only on $k$, and $\Delta_0 \in \N$, such that every $k$-uniform linear hypergraph $\cH$ of maximum degree at most $\Delta \ge \Delta_0$ admits a partition $U_1 \cup \cdots \cup U_{m_1}$ of $V(\cH)$ with $m_1 \le \Delta^{\frac{3}{4k-5}\left(1 + \frac{\eps}{4k - 6}\right)}$ 
satisfying the following properties for each $i \in [m_1]$: 
\begin{enumerate}
\item $\Delta(\cH[U_i]) \le (1+\eta)\,\frac{\Delta}{m_1^{k-1}}$, and
\item every vertex is contained in at most $2C (\log\Delta)^{4k-5}$ loose $3$- and $4$-cycles in $\cH[U_i]$.
\end{enumerate}
\end{lemma}

\begin{proof}
    
\stepcounter{ForClaims}
\renewcommand{\theForClaims}{\ref*{lemma: partition 1}}
Suppose that $\Delta_0$ is sufficiently large in terms of $k$, $\eps$, and $\eta$.
Fix 
\begin{align}\label{def: m_1 partition}
\gamma \coloneqq \frac{3 \eps}{(4k-5)(4k-6)}, \qquad m_1 \coloneqq \big\lfloor\Delta^{\frac{3}{4k-5} + \gamma}\big\rfloor,
\end{align}
and
\begin{align}\label{def: Delta_1, C}
\Delta_1 \coloneqq (1+\eta)\,\frac{\Delta}{m_1^{k-1}}, \qquad C \coloneqq 2M \sqrt{(4k-5)!} \, [8(12k-6)]^{4k-5}, \qquad K \coloneqq C (\log\Delta)^{4k-5},
\end{align}
where $M$ is the maximum of constants $M_3$ and $M_4$ given by Lemma~\ref{lemma: count loose cycles} for $r = 3$ and $r = 4$, respectively.
Given $k\ge 3$, we note that
\[
1 - \left(\frac{3}{4k-5} + \gamma\right)(k-1) 
= \frac{k-3}{4k-6} + \frac{3(k-1)(1-\eps)}{(4k-5)(4k-6)} > 0,
\]
and so, for sufficiently large $\Delta$,
\begin{align}\label{eq: Delta_1 lb}
\frac{\Delta_1}{1+\eta} = \frac{\Delta}{m_1^{k-1}} \overset{\eqref{def: m_1 partition}}{\ge} 
\Delta^{1 - \left(\frac{3}{4k-5}+\gamma\right)(k-1)} 
\gg \Delta^{\Omega(1)}.
\end{align}

Construct a partition $U_1 \cup \cdots \cup U_{m_1}$ of $V(\cH)$ by independently  coloring each vertex $v \in V(\cH)$ by a color $c_v$ chosen uniformly at random from $[m_1]$. We define the following bad events for each vertex $v \in V(\cH)$.
\begin{enumerate}
\item $\cA_v$: there are more than $\Delta_1$ edges containing $v$ such that all vertices in each edge are colored $c_v$.
\item $\cB^{(3)}_v$: there are more than $K$ loose $3$-cycles containing $v$ such that all vertices in each cycle are colored $c_v$.
\item $\cB^{(4)}_v$: there are more than $K$ loose $4$-cycles containing $v$ such that all vertices in each cycle are colored $c_v$.
\end{enumerate}
We will use the symmetric version of the \hyperref[theorem: sym LLL]{Lov\'asz Local Lemma} (Theorem~\ref{theorem: sym LLL}) to show that there is an outcome where none of the above bad events occur. Let us begin by computing the probabilities of each of these events for each $v \in V(\cH)$. We will do so through a sequence of claims.

\smallskip

Fix a vertex $v \in V(\cH)$.

\begin{claim}\label{claim: prob A_v}
$\pr(\cA_v) < \Delta^{-10}$.
\end{claim}
\begin{claimproof}
Let $\mathbf{X}$ be the number of edges $e$ containing $v$ such that all vertices in $e$ are colored $c_v$. As $\cH$ is linear, $\mathbf{X}$ is binomially distributed with mean 
\[
\E[\mathbf{X}] = \frac{\deg(\cH,v)}{m_1^{k-1}} \le \frac{\Delta}{m_1^{k-1}}.
\]
By the \hyperref[theorem: chernoff]{Chernoff Bound} (Theorem~\ref{theorem: chernoff}), we have
\begin{align}
\pr(\cA_v) = \pr(\mathbf{X} > \Delta_1) 
&\overset{\eqref{def: Delta_1, C}}{=} \pr{\left(\mathbf{X} > (1+\eta)\, \frac{\Delta}{m_1^{k-1}}\right)} \\
&\le \pr{\left(\mathrm{Bin}{\left(\Delta, \frac{1}{m_1^{k-1}}\right)} - \frac{\Delta}{m_1^{k-1}} \ge \frac{\eta\Delta}{m_1^{k-1}}\right)} \\
&\le \exp\left(-\frac{\eta^2\Delta}{3m_1^{k-1}}\right) \overset{\eqref{eq: Delta_1 lb}}{=} \exp\left(-\Delta^{\Omega(1)}\right) \ll \Delta^{-10},
\end{align}
as desired.
\end{claimproof}

\begin{claim}\label{claim: prob B_v 3}
$\pr(\cB^{(3)}_v) < \Delta^{-10}$.
\end{claim}
\begin{claimproof}
Let $\cC_3(\cH, v)$ denote the collection of all loose $3$-cycles in $\cH$ containing $v$; note that $V(F) - v \subseteq N_\cH^2(v)$ for every $F \in \cC_3(\cH, v)$. Consider the auxiliary weighted hypergraph $\cK_3$ on $N_\cH^2(v)$ with edge set $E(\cK_3) \coloneqq \{V(F) - v : F \in \cC_3(\cH, v)\}$ and weight $w_f \coloneqq |\{F \in \cC_3(\cH,v) : V(F) - v = f\}|$ for each $f \in E(\cK_3)$; note that by Lemma~\ref{lemma: count loose cycles}, $w_f \le M \Delta^{b_3(3(k-1))} = O(1)$ for all $f \in E(\cK_3)$. Since $|V(F) - v| = 3k-4$ for every $F \in \cC_3(\cH,v)$, the hypergraph $\cK_3$ is $(3k-4)$-uniform, and $|V(\cK_3)| = |N_\cH^2(v)| = O(\Delta^2)$.
If $\cC_3(\cH,v)=\varnothing$, then the claim is immediate; assume henceforth that it is nonempty. 
Define
\[
\mathbf{Y} \coloneqq \sum_{f \in E(\cK_3)} w_f \prod_{u \in f} \I[c_u = c_v],
\]
noting that $\cB^{(3)}_v = \{\mathbf{Y} > K\}$. Since the colors are chosen independently, the family $\{\I[c_u = c_v] : u \in V(\cH) - v\}$ consists of independent $\{0,1\}$-valued random variables; indeed, conditional on $c_v$, these indicators are independent and their joint distribution does not depend on the value of $c_v$.

For $A\subseteq N_\cH^2(v)$ with $0 \le |A| \le 3k-4$, let
\[
\mathbf{Y}_A \coloneqq \sum_{\substack{f \in E(\cK_3)\\f \supseteq A}} w_f \prod_{u \in f \setminus A} \I[c_u = c_v],
\]
noting that $\mathbf{Y}_\varnothing = \mathbf{Y}$. Then, by Lemma~\ref{lemma: count loose cycles}, we have
\[
\E[\mathbf{Y}_A]=\sum_{\substack{f\in E(\cK_3)\\ f\supseteq A}}w_f\,m_1^{-(3k-4-|A|)}
=\frac{|\{F\in\cC_3(\cH,v):A\subseteq V(F)\}|}{m_1^{3k-4-|A|}} \le \frac{M \Delta^{b_3(|A|+1)}}{m_1^{3k-4-|A|}}.
\]
In particular, when $A = \varnothing$, since $b_3(1) = 2$, we have
\begin{align}\label{eq:E0-3}
E_0 \coloneqq \E[\mathbf{Y}] = \E[\mathbf{Y}_\varnothing] 
\le \frac{M \Delta^2}{m_1^{3k-4}}
\overset{\eqref{def: m_1 partition}}{=}  O{\left(\Delta^{-\frac{k-2}{4k-5}-(3k-4)\gamma}\right)} = o(1).
\end{align}
On the other hand, for a non-empty subset $A \subseteq N_\cH^2(v)$ with $|A|=i$, by Lemma~\ref{lemma: count loose cycles}, we have
\[
|\{F\in\cC_3(\cH,v):A\subseteq V(F)\}| \le \begin{cases}
M \Delta & \text{if}~ 1 \le i \le k-1,\\
M &\text{if}~ k \le i \le 3k-4,
\end{cases}
\]
and therefore, given~\eqref{def: m_1 partition},
\begin{align}\label{eq: Y_A ub}
\E[\mathbf{Y}_A] \leq \max\left\{\frac{M\Delta}{m_1^{2k-3}},\, M\right\} = M.
\end{align}
As a result, setting
\begin{align}
E_i \coloneqq \max_{|A|=i}\E[\mathbf{Y}_A], \qquad E' \coloneqq \max_{1 \le i \le 3k-4} E_i, \qquad \text{and}\qquad E \coloneqq \max_{0\le i\le 3k-4}E_i,
\end{align}
it follows from~\eqref{eq:E0-3} and~\eqref{eq: Y_A ub} that $E = \max\{E_0, E'\} \leq M$, and so
\begin{align}\label{eq:lambda-3}
\lambda_3 \coloneqq \frac{1}{8} \left(\frac{K}{2\sqrt{(3k-4)!\, EE'}}\right)^{\frac{1}{3k-4}} \ge \frac{1}{8} \left(\frac{K}{2M \sqrt{(3k-4)!}}\right)^{\frac{1}{3k-4}} 
\overset{\eqref{def: Delta_1, C}}{>} (12k-6) \log\Delta \gg 1.
\end{align}
By applying the \hyperref[theorem: Kim Vu concentration]{Kim--Vu Polynomial Concentration Inequality} (Theorem~\ref{theorem: Kim Vu concentration}) to $\mathbf{Y}$ with $\lambda_3$ in place of $\lambda$,
we obtain:
\begin{align*}
\pr(\cB^{(3)}_v) \le \pr(\mathbf{Y} > K) 
&\overset{\eqref{eq:E0-3}}{\le} \pr{\left(\mathbf{Y} - \E[\mathbf{Y}] > K/2\right)} \\
&\overset{\eqref{eq:lambda-3}}{\le}\pr{\left(|\mathbf{Y}-\E[\mathbf{Y}]|>(8\lambda_3)^{3k-4}\sqrt{(3k-4)!\,EE'}\right)}\\
& \le 2e^{-\lambda_3+2}\,|N_\cH^2(v)|^{3k-5} = O{\left(e^{-\lambda_3}\Delta^{2(3k-5)}\right)} \ll \Delta^{-10},
\end{align*}
as claimed.
\end{claimproof}

\begin{claim}\label{claim: prob B_v 4}
    $\pr(\cB^{(4)}_v)<\Delta^{-10}$.
\end{claim}

\begin{claimproof}
Let $\cC_4(\cH, v)$ denote the collection of all loose $4$-cycles in $\cH$ containing $v$; note that $V(F) - v \subseteq N_\cH^3(v)$ for every $F \in \cC_4(\cH, v)$. Consider the auxiliary weighted hypergraph $\cK_4$ on $N_\cH^3(v)$ with edge set $E(\cK_4) \coloneqq \{V(F) - v : F \in \cC_4(\cH, v)\}$ and weight $w_f \coloneqq |\{F \in \cC_4(\cH,v) : V(F) - v = f\}|$ for each $f \in E(\cK_4)$; note that by Lemma~\ref{lemma: count loose cycles}, $w_f \le M \Delta^{b_4(4(k-1))} = O(1)$ for all $f \in E(\cK_4)$. Since $|V(F) - v| = 4k-5$ for every $F \in \cC_4(\cH,v)$, the hypergraph $\cK_4$ is $(4k-5)$-uniform, and $|V(\cK_4)| = |N_\cH^3(v)| = O(\Delta^3)$.
If $\cC_4(\cH,v)=\varnothing$, then the claim is immediate; assume henceforth that it is nonempty. 
Define
\[
\mathbf{Z} \coloneqq \sum_{f \in E(\cK_4)} w_f \prod_{u \in f} \I[c_u = c_v],
\]
noting that $\cB^{(4)}_v = \{\mathbf{Z} > K\}$. Since the colors are chosen independently, the family $\{\I[c_u = c_v] : u \in V(\cH) - v\}$ consists of independent $\{0,1\}$-valued random variables; indeed, conditional on $c_v$, these indicators are independent and their joint distribution does not depend on the value of $c_v$.

For $A\subseteq N_\cH^3(v)$ with $0 \le |A| \le 4k-5$, let
\[
\mathbf{Z}_A \coloneqq \sum_{\substack{f \in E(\cK_4)\\f \supseteq A}} w_f \prod_{u \in f \setminus A} \I[c_u = c_v],
\]
noting that $\mathbf{Z}_\varnothing = \mathbf{Z}$. Then, by Lemma~\ref{lemma: count loose cycles}, we have
\[
\E[\mathbf{Z}_A]=\sum_{\substack{f\in E(\cK_4)\\ f\supseteq A}}w_f\,m_1^{-(4k-5-|A|)}
=\frac{|\{F\in\cC_4(\cH,v):A\subseteq V(F)\}|}{m_1^{4k-5-|A|}} \le \frac{M \Delta^{b_4(|A|+1)}}{m_1^{4k-5-|A|}}.
\]
In particular, when $A = \varnothing$, since $b_4(1) = 3$, we have
\begin{align}\label{eq:E0}
E_0 \coloneqq \E[\mathbf{Z}] = \E[\mathbf{Z}_\varnothing] 
\le \frac{M \Delta^3}{m_1^{4k-5}}
\overset{\eqref{def: m_1 partition}}{=} O{\left(\Delta^{-(4k-5)\gamma}\right)} = o(1).
\end{align}
On the other hand, for a non-empty subset $A \subseteq N_\cH^3(v)$ with $|A|=i$, by Lemma~\ref{lemma: count loose cycles}, we have
\[
|\{F\in\cC_4(\cH,v):A\subseteq V(F)\}| \le \begin{cases}
M \Delta^2 & \text{if}~ 1 \le i \le k-1,\\
M \Delta & \text{if}~ k \le i \le 2k-2,\\
M &\text{if}~ 2k-1 \le i \le 4k-5,
\end{cases}
\]
and therefore, given~\eqref{def: m_1 partition},
\begin{align}\label{eq: Z_A ub}
\E[\mathbf{Z}_A] \leq \max\left\{\frac{M\Delta^2}{m_1^{3k-4}},\, \frac{M\Delta}{m_1^{2k-3}},\, M\right\} = M.
\end{align}
As a result, setting
\begin{align}
E_i \coloneqq \max_{|A|=i}\E[\mathbf{Z}_A], \qquad E' \coloneqq \max_{1 \le i \le 4k-5} E_i, \qquad \text{and}\qquad E \coloneqq \max_{0\le i\le 4k-5}E_i,
\end{align}
it follows from~\eqref{eq:E0} and~\eqref{eq: Z_A ub} that $E = \max\{E_0, E'\} \leq M$, and so
\begin{align}\label{eq:lambda}
\lambda_4 \coloneqq \frac{1}{8} \left(\frac{K}{2\sqrt{(4k-5)!\, EE'}}\right)^{\frac{1}{4k-5}} \ge \frac{1}{8} \left(\frac{K}{2M \sqrt{(4k-5)!}}\right)^{\frac{1}{4k-5}} 
\overset{\eqref{def: Delta_1, C}}{\ge} (12k-6) \log\Delta \gg 1.
\end{align}
By applying the \hyperref[theorem: Kim Vu concentration]{Kim--Vu Polynomial Concentration Inequality} (Theorem~\ref{theorem: Kim Vu concentration}) to $\mathbf{Z}$ with $\lambda_4$ in place of $\lambda$, we obtain:
\begin{align*}
\pr(\cB^{(4)}_v) \le \pr(\mathbf{Z} > K) 
&\overset{\eqref{eq:E0}}{\le} \pr{\left(\mathbf{Z} - \E[\mathbf{Z}] > K/2 \right)} \\
&\overset{\eqref{eq:lambda}}{\le}\pr{\left(|\mathbf{Z}-\E[\mathbf{Z}]|>(8\lambda_4)^{4k-5}\sqrt{(4k-5)!\,EE'}\right)}\\
& \le 2e^{-\lambda_4+2}\,|N_\cH^3(v)|^{4k-6} = O{\left(e^{-\lambda_4}\Delta^{3(4k-6)}\right)} \ll \Delta^{-10},
\end{align*}
as claimed.
\end{claimproof}

By Claims~\ref{claim: prob A_v},~\ref{claim: prob B_v 3}, and~\ref{claim: prob B_v 4}, each of the events $\cA_v$, $\cB_v^{(3)}$, and $\cB_v^{(4)}$, for any $v \in V(\cH)$, occurs with probability at most $p \coloneqq \Delta^{-10}$.
For each $v \in V(\cH)$, the events $\cA_v$, $\cB_v^{(3)}$, and $\cB_v^{(4)}$ are determined by the colors of the vertices at distance at most $3$ from $v$ in $\cH$. Hence, each of them is mutually independent of the set of events $\cA_u$, $\cB_u^{(3)}$, and $\cB_u^{(4)}$ with $u$ at distance at least $7$ from $v$.
Since $|N_\cH(u)| \le (k-1)\Delta$ for each vertex $u$, there are $O(\Delta^6)$ vertices at distance at most $6$ from each vertex, and so each event is mutually independent of all but at most $O(\Delta^6)$ other events.
Setting $d_{\mathrm{LLL}} \coloneqq \Theta(\Delta^6)$, we have
\[
4pd_{\mathrm{LLL}} \leq \Theta(\Delta^6) \cdot \Delta^{-10} \ll 1.
\]
By Theorem~\ref{theorem: sym LLL}, with positive probability none of the events $\cA_v$, $\cB_v^{(3)}$, and $\cB_v^{(4)}$ with $v \in V(\cH)$ occurs.
\end{proof}

The main result of this section is the following lemma, which shows that we can partition $V(\cH)$ into ``few'' parts such that each part induces an uncrowded hypergraph with ``small'' maximum degree.

\begin{lemma}\label{lemma: partition for coloring}
Let $k \ge 3$ be an integer, and let $\eps, \eta \in (0,1)$ be arbitrary.
There exists $\Delta_0 \in \N$ such that, for every $\Delta \ge \Delta_0$, every $k$-uniform linear hypergraph $\cH$ of maximum degree at most $\Delta$ admits a partition $V_1 \cup \cdots \cup V_m$ of $V(\cH)$ with $m \le \Delta^{\frac{3}{4k-5} + \eps}$ such that each induced subhypergraph $\cH[V_i]$ is uncrowded and has maximum degree at most $(1+\eta)\, \frac{\Delta}{m^{k-1}}$.
\end{lemma}

\begin{proof}
\stepcounter{ForClaims}
\renewcommand{\theForClaims}{\ref*{lemma: partition for coloring}}
Fix $\eta' \coloneqq \eta/3$, so that $(1 + \eta')^2 \leq 1 + \eta$.
By applying Lemma~\ref{lemma: partition 1} to $\cH$ with $\eta'$ in place of $\eta$, we obtain a partition $U_1 \cup \dots \cup U_{m_1}$ of $V(\cH)$ with $m_1 \le \Delta^{\frac{3}{4k-5}\left(1+\frac{\eps}{4k - 6}\right)}$, where, for each $i \in [m_1]$, the induced subhypergraph $\cH[U_i]$ satisfies
$\Delta(\cH[U_i]) \leq \Delta_1$ with
\begin{align}\label{def: Delta_1 in partition lemma 2}
\Delta_1 \coloneqq (1+\eta')\, \frac{\Delta}{m_1^{k-1}} \ge (1+\eta') \Delta^{1-\frac{3(k-1)}{4k-5}\left(1+\frac{\eps}{4k-6}\right)} \gg \Delta^{\Omega(1)},
\end{align}
and every vertex is contained in at most $2K \coloneqq 2C(\log\Delta)^{4k-5}$ loose $3$- and~$4$-cycles (in total) in $\cH[U_i]$ for some constant $C > 0$ depending only on $k$.

We now turn our attention to each $U_i$ and further partition these sets to obtain the claimed partition.
Fix $\cH' \coloneqq \cH[U_i]$, and let
\begin{align}\label{def: Delta_2, m_2}
m_2 \coloneqq \left\lceil (\log\Delta)^5 \right\rceil, \qquad \Delta_2 \coloneqq (1+\eta) \, \frac{\Delta}{(m_1m_2)^{k-1}}.
\end{align}
By the choice of $\eta'$ and the definition of $\Delta_1$ in~\eqref{def: Delta_1 in partition lemma 2}, we have
\begin{align}\label{eq: relate Delta_2, Delta}
\Delta_2
\ge (1+\eta')^2 \, \frac{\Delta}{(m_1m_2)^{k-1}} 
=(1+\eta') \, \frac{\Delta_1}{m_2^{k-1}}.
\end{align}
Construct a partition $W_1 \cup \cdots \cup W_{m_2}$ of $V(\cH')$ by independently  coloring each vertex $v \in V(\cH')$ by a color $c_v'$ chosen uniformly at random from $[m_2]$. We define the following bad events for each vertex $v \in V(\cH')$.
\begin{enumerate}
\item $\cA_v'$: there are more than $\Delta_2$ edges containing $v$ such that all vertices in each edge are colored $c_v'$.
\item $\cB_v'$: there exists a loose $3$- or $4$-cycle containing $v$ in which all vertices are colored $c_v'$.
\end{enumerate}
We will use the general version of the \hyperref[theorem: general LLL]{Lov\'asz Local Lemma} (Theorem~\ref{theorem: general LLL}) to show that there is an outcome where none of the above bad events occur. Let us begin by computing the probabilities of each of these events for each $v \in V(\cH')$. We will do so through a sequence of claims.

\smallskip

Fix a vertex $v \in V(\cH')$.

\begin{claim}\label{claim: prob A_v'}
$\pr(\cA_v') = \exp\left(-\Omega{\left(\dfrac{\Delta_1}{(\log\Delta)^{5k-5}}\right)}\right)$.
\end{claim}
\begin{claimproof}
Let $\mathbf{X}'$ be the number of edges containing $v$ in which all vertices are colored $c_v'$. As $\cH'$ is linear, $\mathbf{X}'$ is binomially distributed with mean 
\begin{align}\label{eq: mean X'}
\E[\mathbf{X}'] = \frac{\deg(\cH',v)}{m_2^{k-1}} \le \frac{\Delta_1}{m_2^{k-1}}.
\end{align}
By the \hyperref[theorem: chernoff]{Chernoff Bound} (Theorem~\ref{theorem: chernoff}), we have
\begin{align}
\pr(\cA_v') = \pr(\mathbf{X}' > \Delta_2) 
&\overset{\eqref{eq: relate Delta_2, Delta}}{\le} \pr{\left(\mathbf{X}' > (1+\eta')\, \frac{\Delta_1}{m_2^{k-1}}\right)}\\
&\le \pr{\left(\mathrm{Bin}{\left(\Delta_1, \frac{1}{m_2^{k-1}}\right)} - \frac{\Delta_1}{m_2^{k-1}} \ge \frac{\eta'\Delta_1}{m_2^{k-1}}\right)} \\
&\le \exp\left(-\frac{(\eta')^2\Delta_1}{3m_2^{k-1}}\right)
\overset{\eqref{def: Delta_2, m_2}}{=} \exp\left(-\Omega{\left(\frac{\Delta_1}{(\log\Delta)^{5k-5}}\right)}\right),
\end{align}
as desired.
\end{claimproof}

\begin{claim}\label{claim: prob B_v'}
$\pr(\cB_v') \le 2C(\log\Delta)^{-(11k-15)}$.
\end{claim}
\begin{claimproof}
Since $v$ is contained in at most $2K$ loose $3$- and~$4$-cycles in $\cH'$, and the probability that all vertices of such a cycle are colored $c_v'$ is either $m_2^{-(3k-4)}$ or $m_2^{-(4k-5)}$, by a union bound, we obtain that
\begin{align}
\pr(\cB_v')
&\le 2K \cdot m_2^{-(3k-4)}
\overset{\eqref{def: Delta_2, m_2}}{\le} 2C(\log\Delta)^{(4k-5) - 5(3k-4)} = 2C(\log\Delta)^{-(11k-15)},
\end{align}
as desired.
\end{claimproof}

For each $v\in V(\cH')$, set
\[
x_{\cA_v'}\coloneqq e^{-(\log\Delta)^2},
\qquad
x_{\cB_v'}\coloneqq4C(\log\Delta)^{-(11k-15)}.
\]
We verify that these assignments satisfy the conditions of Theorem~\ref{theorem: general LLL}.
To this end, fix $v \in V(\cH')$.
We note that each event $\cA_{v'}'$ is determined by the colors of $v'$ and its neighbors in $\cH'$, while each event $\cB_{v'}'$ is determined only by the colors of vertices belonging to the at most $2K$ loose $3$- and $4$-cycles in $\cH'$ containing $v'$. 
It follows that, for some constants $c_1, c_2, c_3 > 0$ depending only on $k$,
\begin{align*}
    |\{v' : \cA_{v'}' \sim \cA_v'\}| &\le c_1 \Delta_1^2, &
    |\{v' : \cB_{v'}' \sim \cA_v'\}| &\le c_2 K\Delta_1, \\
    |\{v' : \cA_{v'}' \sim \cB_v'\}| &\le c_2 K\Delta_1, &
    |\{v' : \cB_{v'}' \sim \cB_v'\}| &\le c_3 K^2.
\end{align*}
Using these bounds together with Claims~\ref{claim: prob A_v'}
and~\ref{claim: prob B_v'}, one can verify that
\begin{align*}
x_{\cA_v'} &\prod_{v' : \cA_{v'}' \sim \cA_v'}
(1 - x_{\cA_{v'}'})
\prod_{v' : \cB_{v'}' \sim \cA_v'}
(1 - x_{\cB_{v'}'}) \\
&\ge e^{-(\log\Delta)^2}
\left(1 - e^{-(\log\Delta)^2}\right)^{c_1\Delta_1^2}
\left(1 - 4C(\log\Delta)^{-(11k-15)}\right)^{c_2 K\Delta_1} \\
&\ge \exp\left(
-(\log\Delta)^2
- 2c_1\Delta_1^2 e^{-(\log\Delta)^2}
- \frac{8c_2 C K\Delta_1}
{(\log\Delta)^{11k-15}}
\right) \\
&= \exp\left(
-o\left(\frac{\Delta_1}{(\log\Delta)^{5k-5}}\right)
\right)
\ge \pr(\cA_v'),
\end{align*}
and that
\begin{align*}
    x_{\cB_v'} &\prod_{v' : \cA_{v'}' \sim \cB_v'}
        (1 - x_{\cA_{v'}'})
      \prod_{v' : \cB_{v'}' \sim \cB_v'}
        (1 - x_{\cB_{v'}'}) \\
    &\ge 4C(\log\Delta)^{-(11k-15)}
      \left(1 - e^{-(\log\Delta)^2}\right)^{c_2 K\Delta_1}
      \left(1 - 4C(\log\Delta)^{-(11k-15)}\right)^{c_3 K^2} \\
    &\ge 4C(\log\Delta)^{-(11k-15)}
      \left(1 - c_2 K\Delta_1 e^{-(\log\Delta)^2}\right)
      \left(1 - 4c_3 CK^2(\log\Delta)^{-(11k-15)}\right) \\
    &\ge 2C(\log\Delta)^{-(11k-15)}
    \ge \pr(\cB_v').
\end{align*}
Therefore, the conditions of Theorem~\ref{theorem: general LLL} are satisfied, and so there exists an outcome of the random coloring under which none of the events $\cA_v'$, $\cB_v'$ occur.
For each $i \in [m_1]$, fix such an outcome, and denote the resulting partition of $U_i$ by $W_{i,1} \cup \cdots \cup W_{i,m_2}$. The final partition consists of the nonempty sets $V_{i,j} \coloneqq W_{i,j}$.
Each $\cH[V_{i,j}]$ has maximum degree at most $\Delta_2$ and contains no loose $3$- or $4$-cycle.
Since linearity excludes $2$-cycles, every $3$-cycle in $\cH$ is loose, and every $4$-cycle that is not loose contains a $3$-cycle as a subhypergraph, the absence of loose $3$- and $4$-cycles in $\cH[V_{i,j}]$ implies that $\cH[V_{i,j}]$ is uncrowded.

Let $m$ be the total number of nonempty parts. 
Since $m \le m_1 m_2$, it follows that
\[
m \le \Delta^{\frac{3}{4k-5}\left(1+\frac{\eps}{4k-6}\right)} \left\lceil (\log\Delta)^5 \right\rceil  \le \Delta^{\frac{3}{4k-5}+\eps}
\]
for sufficiently large~$\Delta$, and that the maximum degree of each $\cH[V_{i,j}]$ is at most
\[
\Delta_2 \overset{\eqref{def: Delta_2, m_2}}{=} (1+\eta) \, \frac{\Delta}{(m_1m_2)^{k-1}} \le (1+\eta)\,\frac{\Delta}{m^{k-1}},
\]
completing the proof.
\end{proof}

We are ready to deduce Corollary~\ref{corollary: linear} from Theorem~\ref{theorem: list chromatic} by using Lemma~\ref{lemma: partition for coloring}. 
\begin{proof}[Proof of Corollary~\ref{corollary: linear}]
It suffices to prove the result for $0 < \eps < 1$.
Let $\cH$ be a $k$-uniform linear hypergraph of maximum degree at most $\Delta \ge \Delta_0$, where $\Delta_0$ is sufficiently large in terms of $k$ and $\eps$. We fix 
\begin{align}\label{def: eps', eta for linear}
\eps' \coloneqq \frac{\eps}{7} \qquad \text{and} \qquad \eta \coloneqq \frac{(k-2)\eps'}{(4k-5)(k-1)}.
\end{align}
By Lemma~\ref{lemma: partition for coloring}, we can find a partition $V_1 \cup \dots \cup V_m$ of the vertex set of $\cH$, with $m \le \Delta^{\frac{3}{4k-5} + \eta}$, such that each induced subhypergraph $\cH[V_i]$ is uncrowded and has maximum degree $\Delta_i$ at most 
\begin{align}\label{def: linear D}
D \coloneqq (1+\eps')\, \frac{\Delta}{m^{k-1}} \ge (1+\eps')\, \Delta^{1 - \left(\frac{3}{4k-5} + \eta\right)(k-1)} \ge  (1+\eps')\, \Delta^{\frac{k-2}{4k-5} - (k-1)\eta}.
\end{align}
Moreover, given~\eqref{def: eps', eta for linear} and $k \ge 3$, it follows that
\begin{align}\label{eq: D lb}
D \ge (1+\eps')\, \Delta^{\frac{k-2}{4k-5} (1-\eps')} = \Delta^{\Omega(1)}.
\end{align}
As a result, we may apply Theorem~\ref{theorem: list chromatic} to each $\cH[V_i]$ for $i \in [m]$ to get
\begin{align}
\chi(\cH[V_i]) \le (1+\eps') {\left((k-1)\,\frac{D}{\log D}\right)^{\frac{1}{k-1}}} 
&\le \frac{1+\eps'}{m} {\left((1+\eps') \cdot (k-1) \cdot \frac{\Delta}{\log \left((1+\eps')\Delta/m^{k-1}\right)}\right)^{\frac{1}{k-1}}} \\
&\overset{\eqref{eq: D lb}}{\le} \frac{1+\eps'}{m} {\left(\frac{1+\eps'}{1-\eps'}\cdot \frac{(k-1)(4k-5)}{k-2}\cdot\frac{\Delta}{\log \Delta}\right)^{\frac{1}{k-1}}} \\
&\le \frac{1+\eps}{m}  {\left(\frac{(k-1)(4k-5)}{k-2}\,\frac{\Delta}{\log \Delta}\right)^{\frac{1}{k-1}}},
\end{align}
where in the last inequality we used that, given~\eqref{def: eps', eta for linear} and $k\ge 3$,
\[
(1+\eps') \left(\frac{1+\eps'}{1-\eps'}\right)^{\frac{1}{k-1}}
< (1+\eps')^2 (1+2\eps')
< (1+3\eps') (1+2\eps') \le 1+\eps.
\]
By using a disjoint set of colors for each set $V_i$, it follows that
\begin{align}
\chi(\cH) \le \sum_{i=1}^m \chi(\cH[V_i]) 
\le (1+\eps) \left(\frac{(k-1)(4k-5)}{k-2}\,\frac{\Delta}{\log \Delta}\right)^{\frac{1}{k-1}},
\end{align}
completing the proof of Corollary~\ref{corollary: linear}.
\end{proof}

\section{Palette sparsification for uncrowded hypergraphs}\label{section: algorithms}

In this section, we derive the following palette sparsification theorem from Theorem~\ref{theorem: main}.

\begin{theorem}\label{theorem: palette}
Let $k \ge 2$ be an integer, and let $\eps > 0$ be arbitrary.
There exists a constant $C > 0$ such that, for every $\gamma \in (0, 1)$, there exists $\Delta_0 \in \N$ with the following property.
Let $\Delta\ge \Delta_0$, $n \in \N$, and let $q, q' \in \N$ with $q' \le q$ be such that
\[
q \ge (1+\eps) {\left(\frac{k-1}{\gamma}\,\frac{\Delta}{\log \Delta}\right)^{\frac{1}{k-1}}} \qquad \text{and} \qquad q' \ge \Delta^{\frac{\gamma}{k-1}} + C(\log n)^{\frac{1}{k}}.
\]
Let $\cH$ be an $n$-vertex $k$-uniform uncrowded hypergraph of maximum degree at most $\Delta$.
For each vertex $v \in V(\cH)$, independently sample a subset $L(v) \subseteq [q]$ of size $q'$ uniformly at random.
Then, with probability at least $1 - 1/n$, $\cH$ admits an $L$-coloring.
\end{theorem}

Fix $k \ge 2$ and $\eps > 0$. We may assume $\eps < 1$, as larger values of $\eps$ only strengthen the hypothesis in Theorem~\ref{theorem: palette}. Let
\begin{align}\label{eq: eta and C}
\eta \coloneqq \frac{\eps}{10} \qquad \text{and} \qquad C \coloneqq \left(\frac{4}{\eta^2}\right)^{\frac{1}{k}}.
\end{align}
Fix any $\gamma \in (0, 1)$. Let $\Delta_0 \in \N$ be sufficiently large in terms of $k$, $\eps$, and $\gamma$, and let $\Delta \ge \Delta_0$. Let $n \in \N$, and let $q, q' \in \N$ with $q' \le q$ satisfy the bounds in Theorem~\ref{theorem: palette}. Let $\cH$ be an $n$-vertex $k$-uniform uncrowded hypergraph of maximum degree at most $\Delta$.

For each vertex $v \in V(\cH)$, independently choose a subset $L(v) \subseteq [q]$ of size $q'$ uniformly at random. For each color $c \in [q]$, let $F_c$ be the hypergraph with vertex set $V(\cH)$ and edge set
\[
\left\{e \in E(\cH)\,:\,c \in \textstyle\bigcap_{v \in e} L(v)\right\}.
\]
For each $v \in V(\cH)$, let
\begin{align}\label{eq: pruned lists}
L'(v) \coloneqq \left\{c \in L(v) \,:\, \deg(F_c, v) \le d\right\},
\end{align}
where
\begin{align}\label{def: d for palette sparsification}
d \coloneqq (1+\eta)\,(q')^{k-1} \left(\frac{\Delta}{q^{k-1}} + 1\right).
\end{align}
We first bound the probability that the pruned list at a fixed vertex is small.

\begin{lemma}\label{lemma:Lprime}
For every vertex $v \in V(\cH)$,
\[
\pr\left(|L'(v)| < (1-\eta)q'\right) \le \frac{1}{n^2}.
\]
\end{lemma}

\begin{proof}
Fix $v \in V(\cH)$, and condition on a realization of $L(v)$. All probabilities and expectations below are taken under this conditioning. Let $t \coloneqq \lceil\eta q'\rceil$. If $|L'(v)| < (1-\eta)q'$, then at least $t$ colors $c \in L(v)$ satisfy $\deg(F_c,v)>d$. Consequently, there exists a set $T \subseteq L(v)$ of size $t$ such that
\[
\sum_{c \in T}\deg(F_c,v)>td.
\]

Fix a set $T \subseteq L(v)$ of size $t$. For each $e \in E(\cH,v)$ and $c \in T$, let $\mathbf{X}_{e,c}$ be the indicator of the event that $c \in L(u)$ for every $u \in e-v$, and set
\[
\mathbf{X} \coloneqq \sum_{c \in T}\sum_{e \in E(\cH,v)}\mathbf{X}_{e,c},
\]
noting that $\mathbf{X}=\sum_{c \in T}\deg(F_c,v)$.

We now show that $\{\mathbf{X}_{e,c}:e \in E(\cH,v),\,c \in T\}$ is a family of negatively correlated random variables. Fix $I \subseteq E(\cH,v)\times T$. For each edge $e \in E(\cH,v)$, let $a_e \coloneqq \left|\{c \in T:(e,c) \in I\}\right|$. Then $\sum_{e \in E(\cH,v)}a_e=|I|$. Since $\cH$ is linear, the sets $e-v$ for $e \in E(\cH,v)$ are pairwise disjoint. For each $e \in E(\cH,v)$ and $u \in e-v$, the probability that $L(u)$ contains every color $c \in T$ with $(e,c) \in I$ is $\prod_{j=0}^{a_e-1}(q'-j)/(q-j)$. Since the lists $L(u)$ for $u\ne v$ are sampled independently, we have
\[
\E\left[\prod_{(e,c) \in I}\mathbf{X}_{e,c}\right] = \prod_{e \in E(\cH,v)}\left(\prod_{j=0}^{a_e-1}\frac{q'-j}{q-j}\right)^{k-1}\le \left(\frac{q'}{q}\right)^{(k-1)|I|} = \prod_{(e,c) \in I}\E[\mathbf{X}_{e,c}],
\]
as required.

By linearity of expectation, we have
\begin{equation}\label{eq: mean palette}
\E[\mathbf{X}] = t\deg(\cH,v)\left(\frac{q'}{q}\right)^{k-1}\le t\Delta\left(\frac{q'}{q}\right)^{k-1} = \frac{td}{1+\eta}-t(q')^{k-1} < td.
\end{equation}
It follows that
\begin{equation}\label{eq: td - EX squared}
(td-\E[\mathbf{X}])^2 \ge \left(\frac{\eta}{1+\eta}\,td+t(q')^{k-1}\right)^2 \ge \frac{4\eta t^2d(q')^{k-1}}{1+\eta} \ge 2\eta t^2d(q')^{k-1}.
\end{equation}
By the \hyperref[lemma:NegativeChernoff]{Chernoff Bound for negatively correlated variables} (Theorem~\ref{lemma:NegativeChernoff}), we have
\begin{align*}
\pr(\mathbf{X}>td) \le \exp\left(-\frac{(td-\E[\mathbf{X}])^2}{2\E[\mathbf{X}]+2(td-\E[\mathbf{X}])/3}\right) &\le \exp\left(-\frac{(td-\E[\mathbf{X}])^2}{2td}\right) \\
&\overset{\eqref{eq: td - EX squared}}{\le} \exp\left(-\eta t (q')^{k-1}\right) \le \exp\left(-\eta^2(q')^k\right),
\end{align*}
where in the last inequality we used that $t\ge\eta q'$.

Taking a union bound over all $t$-element subsets $T\subseteq L(v)$ gives
\[
\pr\left(|L'(v)|<(1-\eta)q'\right) \le \binom{q'}{t}\exp\left(-\eta^2(q')^k\right) \le 2^{q'}\exp\left(-\eta^2(q')^k\right) \le \exp\left(-\frac{\eta^2(q')^k}{2}\right),
\]
where the last inequality holds for sufficiently large $\Delta$, since $(q')^{k-1}\ge\Delta^\gamma$. Using $(q')^k\ge C^k\log n$ and the definition of $C$ in~\eqref{eq: eta and C}, we obtain
\[
\pr\left(|L'(v)|<(1-\eta)q'\right) \le \exp\left(-\frac{\eta^2C^k}{2}\log n\right) = \frac{1}{n^2}.
\]
Since the bound holds for every realization of $L(v)$, averaging over $L(v)$ completes the proof.
\end{proof}

\begin{proof}[Proof of Theorem~\ref{theorem: palette}]
Since $d\ge(q')^{k-1}\ge\Delta^\gamma$, we have $\log d\ge\gamma\log\Delta$. Combined with~\eqref{def: d for palette sparsification} and the lower bound on $q$, this gives, for sufficiently large $\Delta$,
\begin{align}
(1+\eta)\left((k-1)\frac{d}{\log d}\right)^{\frac{1}{k-1}} &\overset{\eqref{def: d for palette sparsification}}{=} (1+\eta)\left(\frac{(k-1)(1+\eta)(q')^{k-1}(\Delta/q^{k-1}+1)}{\log d}\right)^{\frac{1}{k-1}} \nonumber\\
&\le (1+\eta)^{\frac{k}{k-1}}q'\left(\frac{k-1}{\gamma\log\Delta}+\frac{1}{(1+\eps)^{k-1}}\right)^{\frac{1}{k-1}}\nonumber\\
&< (1-\eta)q',
\label{eq: lower bound on L'(v)}
\end{align}
where in the last inequality we used that $\Delta$ is sufficiently large and that, since $0<\eta=\eps/10<1/10$,
\[
\frac{(1+\eta)^{\frac{k}{k-1}}}{1+\eps} \le \frac{(1+\eta)^2}{1+10\eta} < 1-\eta.
\]

By the definition of $L'(v)$, for every vertex $v\in V(\cH)$ and color $c\in L'(v)$,
\[
\left|\left\{e\in E(\cH,v): c\in\textstyle\bigcap_{u\in e}L'(u)\right\}\right| \le \deg(F_c,v)\le d.
\]
Thus, $L'$ satisfies condition~\ref{item: color degree in main thm} of Theorem~\ref{theorem: main} deterministically.

By Lemma~\ref{lemma:Lprime} and a union bound, with probability at least $1-1/n$, we have $|L'(v)|\ge(1-\eta)q'$ for every vertex $v\in V(\cH)$. Together with~\eqref{eq: lower bound on L'(v)}, this shows that $L'$ also satisfies condition~\ref{item: list size in main thm} of Theorem~\ref{theorem: main}, with $\eta$ in place of $\eps$. Moreover, since $d \ge \Delta^{\gamma}$ and $\Delta_0$ is sufficiently large in terms of $k$, $\eps$, and $\gamma$, we have $d \ge d_0$, where $d_0$ is given by Theorem~\ref{theorem: main} with $\eta$ in place of $\eps$. Consequently, by Theorem~\ref{theorem: main}, $\cH$ admits an $L'$-coloring with probability at least $1-1/n$. Since $L'(v)\subseteq L(v)$ for every vertex $v$, this is also an $L$-coloring of $\cH$, completing the proof of Theorem~\ref{theorem: palette}.
\end{proof}

\subsection*{Acknowledgments}

We are grateful to Oliver Janzer for helpful discussions during the initial stages of this project. The authors used ChatGPT~5.6 Sol to proofread this manuscript and polish parts of the writing.

\printbibliography

\end{document}